\documentclass[11pt,a4paper,reqno]{amsart}
\usepackage{amssymb,epic,eepic,hyperref}

\providecommand{\BBb}[1]{{\mathbb{#1}}}

\providecommand{\cal}[1]{{\mathcal{#1}}}   

\newcommand{\ang}[1]{\langle#1\rangle}
\newcommand{\bignt}[1]{\lfloor#1\rfloor}
\newcommand{\Bcirc}{\overset{\lower 1.5pt%
              \hbox{$@,@,@,@,@,\scriptscriptstyle\circ$}}B{}}
\newcommand{\Binfty}{\overset{\lower 1.5pt%
              \hbox{$@,@,@,@,@,\scriptscriptstyle\infty$}}B{}}
\newcommand{\bigdot}{\mathbin{\raise.65\jot\hbox{$\scriptscriptstyle\bullet$}}}

\newcommand{\C}{{\BBb C}}

\newcommand{\Dm}{\BBb D}
\newcommand{\uDm}{\underline{\BBb D}}

\newcommand{\dual}[2]{\langle\,#1,\,#2\,\rangle}

\newcommand{\erd}{\overset{\lower 1pt\hbox{\large.}}{e}
                  \overset{\lower 1pt\hbox{\large.}}{r}}

\newcommand{\Fcirc}{\overset{\lower 1.5pt%
               \hbox{$@,@,@,@,@,\scriptscriptstyle\circ$}}F{}}
\newcommand{\fracc}[2]{{
                \textstyle\frac{#1}{\raise 1pt\hbox{$\scriptstyle #2$}}}}
\newcommand{\fracp}{\fracc1p}

\newcommand{\fracci}[2]{{\frac{#1}{\raise 1pt\hbox{$\scriptscriptstyle #2$}}}}
\newcommand{\fracpi}{\fracci1p}

\newcommand{\grad}{\operatorname{grad}}

\newcommand{\kd}[1]{\boldsymbol{[\![}{#1}\boldsymbol{]\!]}}

\newcommand{\lap}{\operatorname{\Delta}}

\newcommand{\M}{{\scriptstyle M}}
\newcommand{\Mi}{{\scriptscriptstyle M}}

\newcommand{\norm}[2]{\mathinner{\|}#1\,|#2\|}
\newcommand{\Norm}[2]{\mathinner{\bigl\|\,#1\,\big|#2\bigr\|}}

\newcommand{\op}[1]{\operatorname{#1}}

\newcommand{\N}{\BBb N}

\newcommand{\Pl}{\BBb P}

\newcommand{\R}{{\BBb R}}
\newcommand{\Rn}{{\BBb R}^{n}}

\providecommand{\rom}[1]{\upn{#1}}

{
\end{minipage}%
\par\medskip\noindent}%
\newcounter{enmcount}\renewcommand{\theenmcount}{{\rm\arabic{enmcount}}}
\renewenvironment{enumerate}{%
\begin{list}{{\llap{\rm(\theenmcount)}}}{\setlength{\labelwidth}{\leftmargin}%
\usecounter{enmcount}}}{\end{list}}
\newcounter{rmcount}\renewcommand{\thermcount}{{\rm\roman{rmcount}}}

\newcounter{Rmcount}\renewcommand{\theRmcount}{{\rm\Roman{Rmcount}}}

\newcommand{\smlnt}[1]{\lceil#1\rceil}

\newcommand{\supp}{\operatorname{supp}}
\newcommand{\Z}{\BBb Z}

\renewcommand{\check}[1]{\overset{{\scriptscriptstyle \vee}}{#1}}
\renewcommand{\hat}[1]{\overset{{\scriptscriptstyle \wedge}}{#1}}

\numberwithin{equation}{section}

\newtheorem{thm}{Theorem}
\numberwithin{thm}{section}
\newtheorem{prop}[thm]{Proposition}
\newtheorem{lem}[thm]{Lemma}
\newtheorem{cor}[thm]{Corollary}

\theoremstyle{definition}
\newtheorem{defn}[thm]{Definition}

\newtheorem{exmp}[thm]{Example}
 
\theoremstyle{remark}
\newtheorem{rem}[thm]{Remark}

\title[Pointwise multiplication]{Pointwise multiplication\\ of Besov and Triebel--Lizorkin spaces}
\author[Johnsen]{Jon Johnsen}
\address{Mathematical Institute\\University of Copenhagen\\%
Universitetsparken 5\\DK-2100 Copenhagen O\\Denmark}
\email{jjohnsen@math.ku.dk}
\begin{document}
\newcommand{\name}[1]{{#1\/}}           


\begin{abstract}
It is shown that {\em para-multiplication\/} applies to a certain product
$\pi(u,v)$ defined for appropriate $u$ and $v\in\cal S'(\Rn)$.
Boundedness of $\pi(\cdot,\cdot)$ is investigated
for the anisotropic Besov and
Triebel--Lizorkin spaces\,---\,i.e., for $B^{\Mi,s}_{p,q}$ and
$F^{\Mi,s}_{p,q}$ with $s\in\R$ and $p$ and $q\in\,]0,\infty]$ (though
$p<\infty$ in the $F$-case)\,---\,with a treatment of the generic as well
as various borderline cases.

When $\max(s_0,s_1)>0$ the spaces $B^{\Mi,s_0}_{p_0,q_0}\oplus 
B^{\Mi,s_1}_{p_1,q_1}$ and $F^{\Mi,s_0}_{p_0,q_0}\oplus
F^{\Mi,s_1}_{p_1,q_1}$ to which $\pi(\cdot,\cdot)$ applies are
determined. For generic $F^{s_0}_{p_0,q_0}\oplus F^{s_1}_{p_1,q_1}$
the receiving $F^{s}_{p,q}$ spaces are characterised.

It is proved that $\pi(f,g)=f\cdot g$ holds for functions
$f$  and $g$ when $f\cdot g\in L_{1,\op{loc}}$, roughly speaking. In addition,
$\pi(f,u)=fu$ when $f\in\cal O_M$ and $u\in\cal S'$.

Moreover, for an {\em  arbitrary\/} open set $\Omega\subset\Rn$, a
product $\pi_\Omega(\cdot,\cdot)$ is
defined by {\em lifting\/} to $\Rn$. Boundedness of $\pi$ on $\Rn$ is
shown to carry over to $\pi_\Omega$ in general.
\end{abstract}
\maketitle
\section{Introduction} \label{intr-sect}
For the pointwise multiplication of functions, given as
 \begin{equation}
 \mu(f,g)(x)=f\cdot g(x)= f(x) g(x),
 \label{1.1}
 \end{equation}
the {\em differentiability and integrability\/} properties of $\mu(f,g)$ 
are examined and expressed in terms of such properties of the two 
factors $f$ and $g$.

To exemplify this, note that on one hand  $\mu(f,g)$ can have
integrability properties determined by $f$ and $g$, since for 
$0<p,q\le\infty$ H\"older's inequality shows the implication
 \begin{equation}
 f\in L_p(\Rn),\quad g\in L_q(\Rn),\quad
 \fracp +\fracc 1q =\fracc 1r\quad\Longrightarrow
 \quad\mu(f,g)\in L_r(\Rn).
 \label{1.2}
 \end{equation}
On the other hand $\mu(f,g)$ can have differentiability
properties determined by $f$ and $g$, since
it follows from Leibniz' rule that
 \begin{equation}
 f\in C^s(\Rn),\quad g\in C^t(\Rn),\quad r=\min(s,t)
 \quad\Longrightarrow\quad\mu(f,g)\in C^r(\Rn),
 \label{1.3}
 \end{equation}
where $C^s(\Rn)$ with $s\in\R_+\!\setminus\N$  is the following
H\"older space, with global properties suited for Fourier analysis,
 \begin{equation}
 C^s(\Rn)=\big\{\,u\in C^{\bignt s}(\Rn)\bigm|
 \smash[b]{
            \sum_{|\alpha|=\bignt s} }
 \sup_{x\ne y} \frac{|D^\alpha u(x)-D^\alpha u(y)|}{|x-y|^{s-\bignt s}}
 <\infty\,\big\};
 \label{1.4}
 \end{equation}
see Section \ref{notation-sect} for the notation.

More generally one can describe both properties simultaneously, even
with fractional derivatives in the $L_p$-sense.
To do so we follow \name{M. Yamazaki} and \name{W. Sickel}, cf.~\cite{Y1} and
\cite{S}, in their use of
Fourier analysis and para-multiplication in the framework of 
the scale of {\em anisotropic Besov spaces\/} $B^{\Mi,s}_{p,q}(\Rn)$,
for $s\in\R$ and $ 0<p,q\le\infty$, and the scale of anisotropic 
{\em Triebel--Lizorkin spaces\/} $F^{\Mi,s}_{p,q}(\Rn)$; 
the latter are only considered for $0<p<\infty$. See
Section~\ref{spaces-ssect} for details on these spaces.

\bigskip

However, for technical reasons it is convenient to replace
$\mu(\cdot,\cdot)$ by a product $\pi(u,v)$ defined
 for $u$ and $v\in\cal S'(\Rn)$ as 
 \begin{equation}
 \pi(u,v)=\lim_{k\to\infty}\cal F^{-1}(\psi_k\cal Fu)\cdot
                               \cal F^{-1}(\psi_k\cal Fv),
 \label{1.5}
 \end{equation}
when the limit exists in $\cal D'(\Rn)$ for each $\psi\in
C^\infty_0(\Rn)$ that equals $1$ on a neighbourhood of $0$; the limit
is required to be independent of the functions $\psi$. Hereby $\psi_k(\xi)=
\psi(2^{-k\Mi}\xi)$ denotes a quasi-homogeneous dilation, 
cf.~Section~\ref{notation-sect}.

An advantage of $\pi(\cdot,\cdot)$ is that it allows the use of
para-multiplication, cf.~Section~\ref{paramult-sect}, but it is a drawback
that one has to examine 

 \begin{enumerate}
  \item[(I)]\sl whether $ \pi(f,g)=\mu(f,g)=f\cdot g$,
        whenever $f$ and $g$ are functions, and whether $\pi(\cdot,\cdot)$ in
        general has properties similar to those of $\mu(\cdot,\cdot)$.
 \end{enumerate}
In this paper the analysis of $\pi(\cdot,\cdot)$ is centred around the
following 

 \medskip

{\bf Main questions\/}: For each $j=0$, $1$ and $2$, let $A_j$
denote either a Besov space $B^{\Mi,s_j}_{p_j,q_j}(\Rn)$ or a
Triebel--Lizorkin space $F^{\Mi,s_j}_{p_j,q_j}(\Rn)$.
 \begin{enumerate}
  \item[(II)] {\sl Which conditions on $(s_0,p_0,q_0)$ and
         $(s_1,p_1,q_1)$ are necessary and sufficient for 
         $\pi(\cdot,\cdot)$ to be a bounded bilinear operator 
         \begin{equation}
         \pi(\cdot,\cdot)\colon A_0\oplus A_1\to A
         \label{1.6}
         \end{equation}
        for some Besov or Triebel--Lizorkin space $A$? }
  \item[(III)] {\sl And in the affirmative case,
        \sl which conditions on $(s_2,p_2,q_2)$ are necessary and
        sufficient for obtaining \eqref{1.6} with $A=A_2$?}
 \end{enumerate}

For convenience,
any case where, e.g., $A_0$ and $A_1$ are Besov spaces and $A_2$ 
is a Triebel--Lizorkin space is referred to as a $BBF$ case. 
It is also practical to let ``$\bigdot$''
denote a space which can be either a Besov or a Triebel--Lizorkin space.
In this terminology question (III) above has a version for each of the
$\bigdot\bigdot\bigdot$ cases, whereas (II) has $BB\bigdot$,
$BF\bigdot$ and $FF\bigdot$ versions.

A solution to the problem for $\mu(\cdot,\cdot)$ in \eqref{1.1} ff.
above is gathered by establishing answers to (I), (II) and (III).
In these directions it is obtained in this article that:
 \begin{enumerate}
  \item Para-multiplication allows an {\em almost exhaustive\/} discussion
        of (II) and (III) in the $BBB$ and $FFF$ cases.

   In more details, in Section~\ref{ncss-sect} below 
   the set of necessary conditions is enlarged, and afterwards, in 
   Section~\ref{suff-sect}, para-multiplication is used to show that 
   the new set of conditions is sufficient too, except in some
   borderline cases.

   In fact, for $\max(s_0,s_1)>0$ a complete answer is given to
   question~(II) in the $BB\bigdot$ and $FF\bigdot$ cases. For the
   isotropic $FFF$ cases the receiving $A_2$ spaces in (III) is 
   completely characterised for generic $A_0$ and $A_1$. In the
   general generic $BBB$ and $FFF$ cases 
   a few gaps remain open concerning (III).
 
  \item The identity $\pi(f_0,f_1)=f_0\cdot f_1$ holds when $f_j\in
    L_{p_j,\op{loc}}\cap\cal S'$ for $p_j\in\,]0,\infty]$ such that 
    $\fracc1{p_0}+\fracc1{p_1}\le1$
    (i.e., when \eqref{1.2} gives rise 
    to a product in $L_{1,\op{loc}}\subset\cal D'$).

    Moreover, $\pi(f,u)=fu$ for $f\in\cal O_M$ and $u\in\cal S'$.
 
  \item For an {\em arbitrary\/} open set $\Omega\subset\Rn$, a
    product, $\pi_\Omega(\cdot,\cdot)$, on $\Omega$ can be defined by
    lifting to $\Rn$, that is to say, by letting
     \begin{equation}
      \pi_\Omega(u,v)=\lim_{k\to\infty}r_\Omega( \cal F^{-1}(\psi_k\cal Fu')
             \cdot\cal F^{-1}(\psi_k\cal F v')),
     \label{1.6'}
     \end{equation}
    when the limit exists in $\cal D'(\Omega)$ for $u'$ and $v'\in\cal
    S'(\Rn)$ such that $r_\Omega u'=u$ and $r_\Omega v'=v$
    (cf.~Definition~\ref{subset-defn} below).

    As a consequence boundedness of $\pi(\cdot,\cdot)$ as in \eqref{1.6} 
    generally implies boundedness of $\pi_\Omega(\cdot,\cdot)$
    in the corresponding spaces over $\Omega$.

   \item When $\omega\subset\Omega$ is an open set, then
    $\pi_\Omega(u,v)=0$ in $ \omega$, if either $r_\omega u=0$ or
    $r_\omega v=0$ (which trivially holds for $\mu$).
 \end{enumerate}

Concerning earlier contributions to this subject the paper
\cite{Y1} deals with the situation where $p_0=p_1=p_2$, whereas
\cite{S} treats the cases with $p_0=p_1\ne p_2$. Since then the general
problem with $p_0\ne p_1$ has been addressed by \name{H.~Amann} 
in \cite{Ama} and
by \name{Sickel} in \cite{Sic91}, and in fact it is done in both papers for
$m$ factors, with $m\ge2$. However, the former of these treats
only Besov spaces $B^{s}_{p,q}$ with $p$ and
$q\in[1,\infty]$ and Sobolev spaces $W^s_p$ (of functions with values
in Banach spaces); and in some cases only the closures of $\cal S$ in
these spaces are covered. In the latter paper the full scale of 
$F^{s}_{p,q}$ spaces is considered. However, these references do not
provide any new necessary conditions for the general problem. 

In a recent work of \name{Sickel} and \name{H.~Triebel} 
\cite{ST} the case $p_0\ne p_1$
is also studied and a rather complete set of necessary conditions is 
given. However,
there the scope is restricted to the isotropic situation where $0\le
s_0=s_1=s_2$ and $s_j-\fracc n{p_j}\in\,]-n,0[$ holds for $j=0$ and $1$. 

It should be mentioned, that the questions (II) and (III) have
been considered much earlier even for $p_0\ne p_1$. In fact, for
Sobolev spaces $W^s_p$ there is a treatment (for $m$ factors) by
\name{R.~Palais} \cite{Pal}, and Besov (as well as Sobolev) spaces were
considered by \name{J.~L.~Zolesio} in \cite{Zol}. \name{B.~Hanouzet} 
\cite{Han} treated Besov spaces with $p$ and $q\in[1,\infty]$.

As a general reference the recent monograph by \name{M.~Oberguggenberger}
\cite{Ober} is mentioned. The more classical multiplier subject is
treated by \name{V.~A.~Maz'ya} and \name{T.~O.~Shaposhnikova} \cite{MazShap}
and in \name{R.~S.~Strichartz'} paper \cite{Strch}, e.g. 

\bigskip

The definition of $\pi(\cdot,\cdot)$, cf.~\eqref{1.5}, has been introduced
independently in \cite{Sic91} and \cite{JJ93} but without the
$\psi$-independence. In a related context, this requirement has been
shown to be necessary by  \name{J.~F.~Colombeau} and
\name{Oberguggenberger} \cite{ColOber90}, cf.~also Remark~\ref{indp-rem} below.

The results in (2)--(4)\,---\,that address (I) above\,---\,are 
important for the applications of para-multiplication, in particular for
problems on domains $\Omega\subset\Rn$. In \cite{S} the result in (2)
was observed  in the rather restricted case with {\em
global\/} spaces for which $\fracp+\fracc1q=1$. Seemingly (3) and (4)
are unprecedented (at least when $\max(p,q)=\infty$ or when $\Omega$
is non-smooth). (4) is used to show (2) and (3).

In comparison with \cite{Ama} and \cite{Sic91} the present article
treats anisotropic spaces (though only for $m=2$) and it gives 
sufficient conditions which in 
 \begin{itemize}
  \item the $BBB$ cases generalise those in \cite{Ama}, since we 
allow $p$ and $q$ to be arbitrary in $\,]0,\infty]$ and treat the full
spaces (instead of closures of $\cal S(\Rn)$), and in addition our
statements on the sum-exponents $q$ are sharper, 
  \item the $BBB$ and $FFF$ cases cover various 
{\em borderline cases\/} in (II) above, whereas \cite{Sic91} 
does not deal with these at all. (\name{Sickel}, however, also treats 
intersections like $F^{s}_{p,q}\cap L_\infty$.)
 \end{itemize}
The sufficient conditions here are supplemented by a set of necessary
conditions which in the generic $BBB$ and $FFF$ cases leaves only a
few open questions, cf.~(1).

When restricted to the case $0<s_0=s_1=s_2$, our sufficient conditions
coincide with those contained in \cite{ST}, whereas the
necessary conditions there are slightly sharper, in fact also sufficient.
Moreover, they include a study
of the case $0=s_0=s_1=s_2$, cf.~Remark~\ref{ST-rem} below. 

Altogether new sufficient conditions for 
the $BBB$ and $FFF$ cases are given here. In addition
we include a rather sharp set of necessary conditions, valid for general
anisotropic problems. Moreover, the results (2), (3) and (4) above
are proved.

\bigskip

An overview of the necessary and
sufficient conditions for multiplication, cf.~(II) and (III), is given
in the beginning of Section~\ref{suff-sect} below. Comments on the  
applications can be found in Section~\ref{appl-sect}.

\bigskip

{\em Thanks\/} are due to M.~Yamazaki and to W.~Sickel for conversations that
have led to improvements of the results.

\section{Notation and preliminaries} \label{notation-sect} 
 
For a normed or quasi-normed space $X$
we denote by $\norm xX$ the norm of the vector $x$. (Recall that 
$X$ is quasi-normed when the triangle inequality is weakened to
$\norm{x+y}X\le c(\norm xX+\norm yX)$ for some 
$c\ge1$ independent of $x$ and $y$. The prefix ``quasi-'' is omitted
when confusion is unlikely to occur.) For $X_1\times X_2$
the quasi-norm $\norm{x_1}{X_1}+\norm{x_2}
{X_2}$ is used, and considered in this way we write $X_1\oplus X_2$.

As simple examples there is $L_p(\Rn)$ and $\ell_p:=\ell_p(\N_0)$ for 
$p\in\,]0,\infty]$, where $c=2^{\fracpi-1}$ is possible for $p<1$.
However, it is a stronger fact that 
 \begin{equation}
 \norm{f+g}{L_p}\le(\norm{f}{L_p}^p+\norm {g}{L_p}^p)^{\fracpi},
 \quad\text{ for}\quad0<p\le1,
 \label{1.27}
 \end{equation}
which has an exact analogue for the $\ell_p$ spaces.

For a bilinear operator $B(\cdot,\cdot)\colon X_1\oplus
X_2\to Y$, continuity is equivalent to the
existence of a constant $c$ such that $\norm{B(x_1,x_2)}{Y}\le
c\norm{x_1}{X_1}\norm{x_2}{X_2}$ and to boundedness. 
A map $T\colon X\to\cap Y_j$ 
is continuous, if $T\colon X\to Y_j$ is continuous for each~$j$.

The space of compactly supported smooth functions is denoted 
by $C^\infty_0(\Omega)$ or $\cal D(\Omega)$, 
when $\Omega\subset\Rn$ is open;
then $\cal D'(\Omega)$ is the dual space of distributions on $\Omega$.
$\ang{u,\varphi}$ denotes the duality between $u\in\cal D'(\Omega)$
and $\varphi\in C^\infty_0(\Omega)$.

The Schwartz space is denoted by $\cal S(\Rn)$, and the tempered distributions
by $\cal S'(\Rn)$. The seminorms on $\cal S(\Rn)$ are taken to be
$\norm{\psi}{\cal S,\alpha,\beta}=
\sup\bigl\{\,|x^\alpha D^\beta\psi| \bigm| x\in\Rn\,\bigr\}$ for
$\alpha,\beta\in\N_0^n$, or equivalently $\norm{\psi}{\cal
S,N}=\max\bigl\{\,\norm{\psi}{\cal S,\alpha,\beta}\bigm|
|\alpha|,|\beta|\le N\,\bigr\}$ for $N\in\N_0$.

The Fourier transform is denoted by $\cal Fu(\xi)=\hat u(\xi)=\int_{\Rn}
e^{-ix\cdot\xi}u(x)\,dx$, and  the notation $\cal F^{-1}v(x)=\check
v(x)$ is used for its inverse. As customary in Fourier analysis 
 \begin{equation}
 C(\Rn)=\{\,f\in L_\infty(\Rn)\mid \text{$f$ is uniformly continuous}\,\} 
 \end{equation}
and $\norm{f}{C(\Rn)}=\sup|f|$. Moreover,
$C^k(\Rn)=\{\,f\mid D^\alpha f\in C(\Rn),\ |\alpha|\le k\,\}$ 
with the semi-norms $\norm{f}{C^k}=
\sup\bigl\{\,|D^\alpha f(x)|\bigm| x\in\Rn,\ |\alpha|\le k\,\bigr\}$. 
$C^\infty=\cap_k C^k$. 

When $\Omega\subset\Rn$ is open, the restriction
$r_\Omega\colon \cal D'(\Rn)\to\cal D'(\Omega)$ is the
transpose of the extension by 0 outside of $\Omega$, denoted
$e_\Omega\colon C^\infty_0(\Omega)\to C^\infty_0(\Rn)$.

For $t\in\R$, $t_\pm=\max(0,\pm t)$ and
$\bignt t$  denotes the largest integer $\le t$, whereas  $\smlnt t$
is the smallest integer $\ge t$. Moreover, $\min^+(s,t):=\min(s,t,s+t)$.

For each given assertion  we shall follow \name{D. E. Knuth}'s
suggestion in \cite{K} and let $\kd{assertion}$ denote 1 respectively 0
when the assertion is true respectively false.

\subsection{The spaces} \label{spaces-ssect}
To make the considerations in this paper more applicable the
aniso\-tro\-pic versions of the Besov and Triebel--Lizorkin spaces are treated
(at a marginal extra cost). The reader can easily specialise to the 
isotropic case, if desired, since then $\M=(1,\dots,1)$, $|\M|=n$ 
and $[x]=|x|$ below.

The definitions are recalled from 
\cite{Y1}: First each coordinate $x_j$ in $\Rn$ is given a weight
$m_j\ge1$, such that $\min m_j=1$, and $\M=(m_1,\dots,m_n)$ with
$|\M|=m_1+\dots+m_n$. The action of $t\in\R_+=\,]0,\infty[\,$ on
$x\in\Rn$ is defined by $t^\Mi x=(t^{m_1}x_1,\dots,t^{m_n}x_n)$, and
$t^{s\Mi}x=(t^s)^\Mi x$ for $s\in\R$, so that $t^{-\Mi}x=(t^{-1})^\Mi x$.
The anisotropic numerical value $[x]$ associated with $\M$ is 
introduced for $x=0$ as $[0]=0$ and otherwise  as the unique 
positive $t$ such that $t^{-\Mi}x\in S^{n-1}$, i.e., such that
 \begin{equation}
 \left(\frac {x_1}{t^{m_1}}\right)^2+\dots+
 \left(\frac {x_n}{t^{m_n}}\right)^2=1\,.
 \label{1.13}
 \end{equation}
See  for example \cite{Y1} for properties of and remarks on $[x]$. As
examples one could let $\M=(1,\dots,1,2)$ in a treatment of, say, the
Navier--Stokes equations or of the parabolic operator
$\partial_n-(\partial^2_1+\dots+\partial^2_{n-1})$; then
$[x]=\bigl(\frac12(|x'|^2+(|x'|^4+4x_n^2)^{\frac12})\bigr)^{\frac12}$ for
$|x'|^2=x^2_1+\dots+x^2_{n-1}$. 

Secondly a partition of unity, $1=\sum_{j=0}^\infty\Phi_j$, is 
constructed: From a fixed $\Psi\in C^\infty(\R)$, such
that $\Psi(t)=1$ for $0\le t\le\tfrac{11}{10}$ and $\Psi(t)=0$
for $\tfrac{13}{10}\le t$, the functions
 \begin{equation}
 \Psi_j(\xi)= \kd{j\in\N_0} \Psi(2^{-j}[\xi])
 \label{1.14} 
 \end{equation}
are introduced and used to define
 \begin{equation}
  \Phi_j(\xi)=\Psi_j(\xi)-\Psi_{j-1}(\xi),
  \quad\text{ for}\quad j\in\Bbb Z\,.
  \label{1.15} 
 \end{equation}
Thirdly there is then a decomposition, with (weak) convergence 
in $\cal S'$,
 \begin{equation}
 u =\sum_{j=0}^\infty\,\cal F^{-1}\Phi_j\cal Fu\,,
 \quad\text{ for every}\quad u\in\cal S'\,.
 \label{1.16} 
 \end{equation}
Here it is understood that $\cal F^{-1}\psi\cal F u=\cal
F^{-1}(\psi\hat u)$ for $\psi\in\cal S$ and $u\in\cal S'$. 
Moreover, $u_k:=\cal F^{-1}\Phi_k\cal F u$ and
$u^k:=\cal F^{-1}\Psi_k\cal Fu$, and also for general $\psi\in\cal S$
we write $u^k=\cal F^{-1}\psi_k\cal Fu$ when $\psi_k=\psi(2^{-k\Mi}\cdot)$.

\smallskip

Now the {\em anisotropic Besov space}, $B^{\Mi,s}_{p,q}(\Rn)$, with {\em weight
$\M$,  smoothness
index $s\in\R$, integral-exponent $p\in\left]0,\infty\right]$ 
{\rm and} sum-exponent} $q\in\left]0,\infty\right]$, is defined as
 \begin{equation}
 B^{\Mi,s}_{p,q}(\Rn)=\bigl\{\,u\in\cal S'(\Rn)\bigm|
 \Norm{ \{2^{sj} \norm{\cal F^{-1}\Phi_j\cal
 Fu(\cdot)}{L_p} \}_{j=0}^\infty}{\ell_q} <\infty\,\bigr\},
 \label{1.17} 
 \end{equation}
and the {\em anisotropic Triebel--Lizorkin space}, $F^{\Mi,s}_{p,q}(\Rn)$, 
with {\em weight $\M$, smoothness
index $s\in\R$, integral-exponent $p\in\left]0,\infty\right[\,$ 
{\rm and} sum-exponent} $q\in\left]0,\infty\right]$ as
 \begin{equation}
 F^{\Mi,s}_{p,q}(\Rn)=\bigl\{\,u\in\cal S'(\Rn)\bigm|
 \Norm{ \norm{ \{2^{sj}\cal F^{-1}\Phi_j\cal
 Fu\}_{j=0}^\infty}{\ell_q} (\cdot)}{L_p} <\infty\,\bigr\}\,.
 \label{1.18}
 \end{equation} 
For the history of (the isotropic versions of) the
spaces we refer to \cite{T2,T3}. 

The spaces $B^{\Mi,s}_{p,q}$ and $F^{\Mi,s}_{p,q}$ are quasi-Banach
spaces with the quasi-norms  given by the finite expressions in
\eqref{1.17} and \eqref{1.18}. Concerning an analogue of \eqref{1.27} one has
 \begin{equation}
 \norm{f+g}{B^{\Mi,s}_{p,q}}\le(\norm f{B^{\Mi,s}_{p,q}}^\lambda+
 \norm {g}{B^{\Mi,s}_{p,q}}^\lambda)^{\frac{1}{\lambda}},
 \quad\text{ for $\lambda=\min(1,p,q)$},
 \label{1.27'}
 \end{equation}
with a similar result for the Triebel--Lizorkin spaces.

\subsection{Embeddings} \label{embd-ssect}
In the rest of this subsection the explicit mention of the
restriction $p<\infty$ concerning the $F^{s}_{p,q}$ spaces is omitted. E.g., 
\eqref{1.19} below should be read with $p\in\,]0,\infty]$ in the~Besov part,
and with $p\in\,]0,\infty[\,$ in the Triebel--Lizorkin part. 

The spaces $B^{\Mi,s}_{p,q}(\Rn)$ and $F^{\Mi,s}_{p,q}(\Rn)$ are
complete, for $p$ and $q\ge1$ they are Banach spaces, and in any case
$\cal S(\Rn)\hookrightarrow B^{\Mi,s}_{p,q}(\Rn),\,F^{\Mi,s}_{p,q}(\Rn)
\hookrightarrow\cal S'(\Rn)$ are continuous. 

By a modification of the proof in \cite{T2} 
of the cases with $\M=(1,\dots,1)$ and $p<\infty$, the image of 
$\cal S(\Rn)$ is dense in 
$B^{\Mi,s}_{p,q}(\Rn)$ and in $F^{\Mi,s}_{p,q}(\Rn)$ for 
$p$ and $q<\infty$, and similarly $C^\infty(\Rn)$ is dense 
in $B^{\Mi,s}_{\infty,q}(\Rn)$ for $q<\infty$. 

The definitions imply that $B^{\Mi,s}_{p,p}(\Rn)=F^{\Mi,s}_{p,p}(\Rn)$,
and they imply the existence of {\em simple\/} embeddings for 
$s\in\R,\,\,p\in\left]0,\infty\right]$ 
and $o$ and $q\in\left]0,\infty\right]$: 
  \begin{gather}
  B^{\Mi,s}_{p,q}(\Rn)\hookrightarrow B^{\Mi,s}_{p,o}(\Rn),
          \quad F^{\Mi,s}_{p,q}(\Rn)\hookrightarrow F^{\Mi,s}_{p,o}(\Rn),
                              \quad\text{when $q\le o$}, 
  \label{1.19}\\ 
  B^{\Mi,s}_{p,q}(\Rn)\hookrightarrow 
                  B^{\Mi,s-\varepsilon}_{p,o}(\Rn),
          \quad F^{\Mi,s}_{p,q}(\Rn)\hookrightarrow 
                  F^{\Mi,s-\varepsilon}_{p,o}(\Rn),\quad\text{when
                  $ \varepsilon>0,$}
   \label{1.19'} \\ 
  B^{\Mi,s}_{p,\min(p,q)}(\Rn)\hookrightarrow F^{\Mi,s}_{p,q}(\Rn)
     \hookrightarrow B^{\Mi,s}_{p,\max(p,q)}(\Rn).
 \label{1.19''}
 \end{gather}
There are Sobolev embeddings if $s-\fracc{|\M|}p\ge
t-\fracc{|\M|}r$ and $r>p$, cf.\ \cite{Y1}, 
 \begin{gather}
 B^{\Mi,s}_{p,q}(\Rn)\hookrightarrow B^{\Mi,t}_{r,o}(\Rn),
 \quad\text{ provided $q\le o$ when $ s-\fracc{|\M|}p
 =t-\fracc{|\M|}r$},
 \label{1.20} \\
 F^{\Mi,s}_{p,q}(\Rn)\hookrightarrow F^{\Mi,t}_{r,o}(\Rn),
 \quad\text{ for any $ o$ and $q\in\,]0,\infty].$}
 \label{1.20'}
 \end{gather}
Furthermore, Sobolev embeddings also exist between the two scales,
in fact under the assumptions $\infty\ge p_1>p>p_0>0$ and 
$s_0-\fracc{|\M|}{p_0}=s-\fracc{|\M|}p=s_1-\fracc{|\M|}{p_1}$ 
 \begin{equation}
  \begin{gathered}
  B^{\Mi,s_0}_{p_0,q_0}(\Rn)\hookrightarrow F^{\Mi,s}_{p,q}(\Rn)
  \hookrightarrow B^{\Mi,s_1}_{p_1,q_1}(\Rn),\\
  \text{for $q_0< p< q_1$ and for $q_0=p\le q$ and $q\le
  p=q_1$, if $\M \ne(1,\dots,1)$;}\\
  \text{and for $q_0\le p$ and $p\le q_1$, if $\M=(1,\dots,1)$.}  
  \end{gathered}
  \label{1.21} 
 \end{equation}
This is obtained from (\ref{1.20}), \eqref{1.20'} and (\ref{1.19''}) 
except for the sharpened 
results for $\M=(1,\dots,1)$, which are interpolation results due 
to \name{J.~Franke}, \cite{F3}, and \name{B.~Jawerth}, \cite{J}, respectively.
 
\bigskip

Concerning relations to other spaces, one has that
$B^{\Mi,s}_{\infty,\infty}(\Rn)=C^{\Mi,s}(\Rn)$ when $s>0$ and 
$\frac{s}{m_k}\notin\N$ for $k=1$,\dots,$n$ 
(the anisotropic H\"older spaces),
and that $L_p(\Rn)=F^{\Mi,0}_{p,2}(\Rn)$ for $1<p<\infty$. Moreover, the 
$F^{\Mi,s}_{p,2}$ equal the anisotropic Bessel-potential spaces 
$H^{\Mi,s}_p$, for $s\in\R$ and $1<p<\infty$, hence\,---\,in the
isotropic case, which is indicated by omission of $\M$\,---\,the 
$F^m_{p,2}(\Rn)$ equal the classical
Sobolev spaces $W^m_p(\Rn)$ for $m\in\N$, and $H^s_{2}(\Rn)=
F^s_{2,2}(\Rn)=B^s_{2,2}(\Rn)$ for $s\in\R$. 
See \cite{Y1}, \cite[Rem.\ 4.4]{Y2} and  
\cite{T2,T3} for these and other identifications.

Furthermore, one finds by use of
(\ref{1.16}), (\ref{1.17}) and (\ref{1.20}), when $0<p,q\le\infty$, that
 \begin{equation}
  \begin{gathered}
  B^{\Mi,s}_{p,q}(\Rn)\hookrightarrow B^{\Mi,0}_{\infty, 1}(\Rn)\hookrightarrow
  C(\Rn)\hookrightarrow L_\infty(\Rn)\hookrightarrow B^{\Mi,0}_{\infty,\infty}
  (\Rn),\\
  \text{if $s>\fracc{|\M|}p$, or if $s=\fracc{|\M|}p$ and $q\le1$}. 
  \end{gathered}
  \label{1.22} 
 \end{equation} 
Then \eqref{1.21} gives for the Triebel--Lizorkin spaces that for 
$0<q\le\infty$
 \begin{equation}
  \begin{gathered}
  F^{\Mi,s}_{p,q}(\Rn)\hookrightarrow B^{\Mi,0}_{\infty,1}(\Rn)\hookrightarrow
  C(\Rn)\hookrightarrow L_\infty(\Rn),\\
  \text{if $s>\fracc{|\M|}p$, or $s=\fracc{|\M|}p$ and either $p<1$}
  \text{ or $p=1\ge q$;}\\
  \text{when $\M=(1,\dots,1)$ also for $s=\fracc{|\M|}p$ with
  $p\le1$}.
  \end{gathered}
  \label{1.23}
 \end{equation}
Moreover, when $|\M|(\fracc{1}p-1)_+\le s<\fracc {|\M|}p$
one has, with $\tfrac{|\M|}{t}=\fracc{|\M|}p-s$, that
 \begin{equation}
 \begin{gathered}
   F^{\Mi,s}_{p,q}(\Rn)\hookrightarrow\bigcap\{\, L_r(\Rn)\mid
  p\le r\le t\,\},\\
  \text{ provided $q \le 1+\kd{1<p} $ if $s=0$}.
 \end{gathered}
 \label{1.24}
 \end{equation}
Indeed, when $s>|\M|(\fracp-1)_+$ and $r=t$ one can use \eqref{1.20'} 
and the fact that $F^{\Mi,0}_{t,2}=L_t$. Hence \eqref{1.16} is a series 
of functions in $L_t$, that converges in $L_t$ since $F^{M,0}_{t,1}
\hookrightarrow L_t$. Then, from $F^{\Mi,s}_{p,q}\hookrightarrow 
F^{\Mi,0}_{p,1}$, it follows that there is also convergence 
in $L_p$, the limits being the same a.e. Thus \eqref{1.24} follows 
for $r=p$ and therefore also for the intermediate values.
This extends to $s=|\M|(\fracp-1)$ when $s\ge0$, since 
$F^{\Mi,s}_{p,q}\hookrightarrow B^{\Mi,0}_{1,1}\hookrightarrow L_1$ then.
 
Likewise (\ref{1.21}) implies for $|\M|(\fracc1p-1)_+\le s
<\fracc{|\M|}p$ and $0<p,q\le\infty$ that, with $t$ as above,
 \begin{equation}
  B^{\Mi,s}_{p,q}(\Rn)\hookrightarrow\bigcap\{\, L_r(\Rn)\mid
  p\le r< t\,\}
 \label{1.25}
 \end{equation}
Here $r=t$ can be included in general when $q<t$, and if either $t\le2$ 
or $\M=(1,\dots,1)$ also when $q\le t$. For $s=0$, one has
$B^{\Mi,0}_{p,q}\hookrightarrow L_p$ for $q\le\min(2,p)$. 
(Cf.~\cite[p.~97]{T3} for the pitfalls in the case $p<1$.) 

The `intermediate value property' for the $L_p$ spaces (used above) gives that
 \begin{equation}
 F^{\Mi,s}_{p_0,q}(\Rn)\cap F^{\Mi,s}_{p_1,q}(\Rn) \subset
  \bigcap\{\,F^{\Mi,s}_{r,q}(\Rn)\mid p_0\le r\le p_1\,\}, 
 \label{1.26}
  \end{equation} 
holds for $p_0\le p_1$, and similarly for the Besov spaces.

\subsection{Convergence theorems} \label{Ythm-ssect} 
As a basic tool Yamazaki's theorems are recalled. They will be applied to
(the series defining) 
the para-multiplication operators $\pi_1$, $\pi_2$ and $\pi_3$ in
Theorem~\ref{basic-thm} below.

\begin{thm} \label{Y1-thm} \ Let $s\in\R$, let $p$ and
$q\in\,]0,\infty]$ and suppose $u_j\in\cal S'(\Rn)$ satisfies
 \begin{equation}
  \supp \hat u_j\subset\bigl\{\,\xi\bigm| \kd{j>0}A^{-1}2^j\le [\xi]\le
 A2^j\,\bigr\},\quad\text{ for}\quad j\in\N_0,
 \label{1.30}
 \end{equation}
for some $A>0$. Then the following holds, if $p<\infty$ in \rom{(2)}:
 \begin{itemize}
 \item[{\rm (1)\/}] If $\Norm{\{2^{sj}
                     \norm{u_j}{L_p}\}^\infty_{j=0}}{\ell_q}=B<\infty$, 
 then the series
 $\sum_{j=0}^\infty u_j$ converges in $\cal S'(\Rn)$ to a limit $u\in
 B^{\Mi,s}_{p,q}(\Rn)$ and the estimate $\norm{u}{B^{\Mi,s}_{p,q}}\le CB$
 holds for some constant $C=C(n,\M,A,s,p,q)$. 

 \item[{\rm (2)\/}]  If 
 $\Norm{\norm{\{2^{sj}u_j\}^\infty_{j=0}}{\ell_q}(\cdot)}{L_p}=B<\infty$, 
 then the series
 $\sum_{j=0}^\infty u_j$ converges in $\cal S'(\Rn)$ to a limit $u\in
 F^{\Mi,s}_{p,q}(\Rn)$ and the estimate $\norm{u}{F^{\Mi,s}_{p,q}}\le CB$
 holds for some constant $C=C(n,\M,A,s,p,q)$. 
 \end{itemize}
\end{thm}

The second of these theorems states that the spectral conditions on
the series $\sum_{j=0}^\infty u_j$ can be relaxed if the smoothness
index $s$ is sufficiently large.

\begin{thm} \label{Y2-thm} \ Let $s\in\R$, let $p$ and 
$q\in\,]0,\infty]$ and suppose $u_j\in\cal S'(\Rn)$ satisfies
 \begin{equation}
 \supp \hat u_j\subset\bigl\{\,\xi\bigm| [\xi]\le A2^j\,\bigr\},\quad
 \text{ for}\quad j\in\N_0,
 \label{1.31}
 \end{equation}
for some $A>0$. Then the following holds, if $p<\infty$ in \rom {(2)}:
 
\begin{itemize}
\item[{\rm (1)\/}] If $s>|\M|(\fracp-1)_+$ and
if $\Norm{\{2^{sj}\norm{u_j}{L_p}\}^\infty_{j=0}}{\ell_q}=B<\infty$, 
then the series
$\sum_{j=0}^\infty u_j$ converges in $\cal S'(\Rn)$ to a limit $u\in
B^{\Mi,s}_{p,q}(\Rn)$ and the estimate $\norm{u}{B^{\Mi,s}_{p,q}}\le CB$
holds for some constant $C=C(n,\M,A,s,p,q)$. 

\item[{\rm (2)\/}] If $s>|\M|(\tfrac1{\min(p,q)}-1)_+$, and
if $\Norm{\norm{\{2^{sj}u_j\}^\infty_{j=0}}{\ell_q}(\cdot)}{L_p}=B<\infty$, 
then the series
$\sum_{j=0}^\infty u_j$ converges in $\cal S'(\Rn)$ to a limit $u\in
F^{\Mi,s}_{p,q}(\Rn)$ and the estimate $\norm{u}{F^{\Mi,s}_{p,q}}\le CB$
holds for some constant $C=C(n,\M,A,s,p,q)$. 
\end{itemize}
\end{thm}

For the proofs of Theorems \ref{Y1-thm} and \ref{Y2-thm} the reader 
is referred to \cite{Y1}. In part Theorem \ref{Y2-thm} is based on 
\cite[Lemma 3.8]{Y1}, which we for later reference shall state for
$s<0$ in a slightly generalised version (that is proved similarly):

\begin{lem} \label{Y-lem} 
For each $s<0$ and $q$ and $r\in\,]0,\infty]$ there exists $c<\infty$ such that
 \begin{equation}
 \Norm{ \bigl\{ 2^{sj}(\textstyle{\sum_{k=0}^j} |a_k|^r)^{\fracci1r} \bigr\}
 ^\infty_{j=0} }{\ell_q}
 \le c\norm{ \{2^{sj}a_j\}^\infty_{j=0} }{\ell_q}
 \label{3.7'} 
 \end{equation} 
holds for any sequence $\{a_j\}_{j=0}^\infty$ of complex numbers (with
modification for $r=\infty$).
\end{lem}

We shall also pay attention to the cases with $s=|\M|(\fracp-1)_+$.

\begin{exmp} \label{Y2-exmp}
In case (1) of Theorem \ref{Y2-thm} above it is not possible to relax 
the~condition
$s>|\M|(\fracp-1)_+$ much: 

On one hand, when $s=|\M|(\fracp-1)$ and $q>1$ the sequence 
$\{ k^{-1}\check\Psi_k \}_{k\in\N}$ has finite norm as required there,
but the series $\sum_{k=1}^\infty k^{-1}\check\Psi_k$ is not convergent
in $\cal S'(\Rn)$ since
$\ang{\check\Psi_1+\dots+k^{-1}\check\Psi_{k},\varphi}$ equals
$\varphi(0)(1+\dots+k^{-1})$ when 
$\supp \hat\varphi\subset\{\,\xi\mid\Psi_0=1\,\}$.

On the other hand, for $s=0$ and $q>1$ one may consider $\{ k^{-1}\check
\Psi_0\}_{k\in\N}$. Again $\{k^{-1}\norm{\check\Psi_0}{L_p} \}_{k\in\N}\in
\ell_q$, and with $\varphi$ as above
$\dual{(1+\dots+k^{-1})\check\Psi_0}{\varphi}$ diverges when
$\varphi(0)\ne0$.
Obviously the second example applies also to case (2) of
Theorem~\ref{Y2-thm}. 
\end{exmp}

The cases with $s=0$ may be partly covered, even without
spectral conditions:

\begin{prop} \label{zero-prop}
Let $q\le1\le p\le\infty$ and let $u_j\in\cal S'(\Rn)$ for $j\in\N_0$.
 \begin{itemize}
  \item[(1)] If $\bigl(\sum_{j=0}^\infty\norm{u_j}{L_p}^q\bigr)^{\fracci1q}
  =:B<\infty$,
  then $\sum_{j=0}^\infty u_j$ converges in $L_p(\Rn)$ to a limit
  $u$ with $\norm{u}{L_p}\le B$.
   \item[(2)]  If, for $p<\infty$,
  $\Norm{(\sum_{j=0}^\infty |u_j(\cdot)|^q)^{\fracci1q}}{L_p}=:B<\infty$,
  then $\sum_{j=0}^\infty u_j$ converges in $L_p(\Rn)$ to a limit
  $u$ with $\norm{u}{L_p}\le B$.
 \end{itemize}
\end{prop}
\begin{proof} In case (1) the embedding $\ell_q\hookrightarrow \ell_1$
implies that $\sum_{j=0}^\infty \norm{u_j}{L_p}\le B$. Then the
completeness of $L_p$ gives the convergence of $\sum u_j$. Concerning
(2) one has that 
$\norm{\sum_{j\in J}u_j}{L_p}\le 
\Norm{(\sum_{j\in J}|u_j|^q)^{\fracci1q}}{L_p}$ when $J\subset\N_0$. For 
$J$ of the form $\{\,m,\dots,m+k\}$ this gives the  convergence by
majorisation with $(\sum_{j=0}^\infty |u_j|^q)^{\fracci pq}$. Then
$J=\{\,0,\dots,m\,\}$ gives $\norm{u}{L_p}\le B$.
\end{proof}

It is seen from Example~\ref{Y2-exmp} above that this result can not
be extended to higher values of $q$. 
For the borderline cases with $s=\fracc{|\M|}p -|\M|$ in
Theorem~\ref{Y2-thm} we have

\begin{prop} \label{border-prop}
Let $p$ and $q\in\,]0,1]$ and let $s=\fracc{|\M|}p-|\M|$. Let moreover
$u_j\in\cal S'(\Rn)$ satisfy the spectral condition in \eqref{1.31}.
 \begin{itemize}
  \item[(1)] If $\bigl(\sum_{j=0}^\infty 2^{sjq}\norm{u_j}{L_p}^q
  \bigr)^{\fracci1q}=:B<\infty$,
  then $\sum_{j=0}^\infty u_j$ converges in $L_1(\Rn)$ to a limit
  $u$ with $\norm{u}{L_1}\le cB$.
   \item[(2)]  If
  $\Norm{(\sum_{j=0}^\infty |2^{sj}u_j(\cdot)|^q)^{\fracci1q}}{L_p}
  =:B<\infty$,
  then $\sum_{j=0}^\infty u_j$ converges in $L_1(\Rn)$ to a limit
  $u$ with $\norm{u}{L_1}\le c B$.
 \end{itemize}
\end{prop}
\begin{proof} In view of \eqref{1.31} the
quasi-homogeneous Nikolski\u\i --Plancherel--Polya inequal\-ity  asserts
that $\norm{u_j}{L_1}\le
C(A2^j)^{\fracci{|\Mi|}p-|\Mi|}\norm{u_j}{L_p}= C_1
2^{sj}\norm{u_j}{L_p}$, cf.~\cite[Prop.~2.13]{Y1}. In (1) the preceding
proposition therefore yields the claim. In (2) it is not a restriction
to assume assume that $p\le q\le 1$, and then one may reduce to case~(1)
cf.~the proof of
$F^{\Mi,s}_{p,q}\hookrightarrow B^{\Mi,s}_{p,\max(p,q)}$.
\end{proof}

It would be interesting and useful to know if the limits in
Proposition~\ref{border-prop} belong to $B^{\Mi,s}_{p,q}$ and
$F^{\Mi,s}_{p,q}$, respectively.

\section{Products of tempered distributions} \label{paramult-sect}

The results for $\mu(f,g)$ will be obtained from the more general
product $\pi(u,v)$ defined\,---\,for each $\M$\,---\,as follows:

\begin{defn} \label{pi-defn}
Let $\psi\in C^\infty_0(\Rn)$ satisfy $\psi(x)=1$ for $x$ in a
neighbourhood of $x=0$, and let $\psi_k(x)=\psi(2^{-k\Mi}x)$.
Denote then, for $u$ and $v\in\cal S'(\Rn)$,
 \begin{equation}
 \pi_\psi(u,v)=\lim_{k\to\infty}\cal F^{-1}(\psi_k\cal Fu)\cdot
                               \cal F^{-1}(\psi_k\cal Fv),
 \label{1.7}
 \end{equation}
when the limit exists in $\cal D'(\Rn)$.

The {\em product\/} $\pi(u,v)$ is defined as $\pi(u,v)=\pi_\psi(u,v)$,
when $\pi_\psi(u,v)$ exists for all such $\psi$ and is independent of $\psi$. 
\end{defn}

Since $\cal F^{-1}\psi\cal Fu$ is a smooth function by the 
Paley--Wiener theorem, the multiplication on the right hand side 
of \eqref{1.7} makes sense. The limit is taken in $\cal D'(\Rn)$ in
order that Definition~\ref{subset-defn} below gives back
Definition~\ref{pi-defn} when it is applied to $\Omega=\Rn$.

\begin{exmp} \label{pi-exmp}   
One has $\pi(\chi,\delta_0)=\frac{1}{2}\delta_0$ (for any $\M$), by
a direct computation, when $\chi(x)=\kd{x_n>0}$ denotes the
characteristic function of the half-space $\Rn_+$ and $\delta_0$ is
the delta measure at the origin in $\Rn$.

On the real line, $\pi(x_+^{-\frac12},x_-^{-\frac12})=
\pi\delta_0$ for the locally integrable functions 
$x_{\pm}^{-\frac{1}{2}}:=\kd{x\lessgtr0}|x|^{-\frac{1}{2}}$, as one
may verify similarly to \cite[Ex.~2.3]{Ober}. 
Thus it may be said that $\pi(\cdot,\cdot)$ differs significantly 
from $\mu(\cdot,\cdot)$, cf.~(I) in the introduction.
\end{exmp}

For the analysis of $\pi(u,v)$ it is convenient to
introduce the para-multiplication operators $\pi_1(\cdot,\cdot)$, 
$\pi_2(\cdot,\cdot)$ and $\pi_3(\cdot,\cdot)$. Let $\psi$ be as in
Definition~\ref{pi-defn}. 

With $u_j=\cal F^{-1}\varphi_j\cal Fu$\,---\,where $\varphi_j=
\psi_j-\psi_{j-1}$ (and $\psi_j\equiv0$ for $j<0$) similarly to 
\eqref{1.15} above\,---\,the operators
$\pi_1,\pi_2$ and $\pi_3$ are defined as follows, cf. \cite{Y1},
 \begin{align} 
 \pi_1(u,v)&=\sum_{j=0}^\infty\, (u_0+\dots+u_{j-2})\cdot v_j,
 \label{1.8}\\
 \pi_2(u,v)&=\sum_{j=0}^\infty\, (u_{j-1}\cdot v_j+u_j\cdot v_j
                +u_j\cdot v_{j-1}),
 \label{1.9}\\
 \pi_3(u,v)&=\sum_{j=0}^\infty\, u_j\cdot (v_0+\dots+v_{j-2})
                 =\pi_1(v,u).
 \label{1.10}
 \end{align}
This technique is due to \name{J.~Peetre} and  \name{H.~Triebel},
see \cite{Pee}, \cite{T-pmlt} and  \cite{T0}. The para-{\em differential\/}
notion was introduced by \name{J.~M.~Bony} in \cite{Bon} and extended in 
\cite{Y1,Y2}. 

Here the series defining $\pi_2$ is
regrouped to have $u_j\cdot v_{j-1}$ instead of $u_{j+1}\cdot
v_j$, for this facilitates the passage to the second line in \eqref{1.11}
below. Moreover, the $\pi_j$ are considered here for each $\psi$
ocurring in Definition~\ref{pi-defn} above.

\begin{lem} \label{pi-lem}
The limit $\pi_\psi(u,v)$ exists and
 \begin{equation}
 \pi_\psi(u,v)=\pi_1(u,v)+\pi_2(u,v)+\pi_3(u,v),
 \label{1.10'}
 \end{equation} 
whenever the series defining $\pi_j(u,v)$ converges in $\cal D'(\Rn)$
for $j=1$, $2$ and $3$ (for a $\psi$ as in Definition~\ref{pi-defn}).
\end{lem}
\begin{proof} It is found by the construction of the 
$\varphi_j$ that $\varphi_0+\dots+\varphi_k=\psi_k$, so 
 \begin{equation}
  \begin{aligned}
 \pi_1(u,v)+\pi_2(u,v)+\pi_3(u,v)&=\lim_{k\to\infty}
 \sum_{j=0}^k\,((u_0+\dots+u_j)\cdot v_j 
 \\[-3.5\jot]
 &\hphantom{=\lim_{k\to\infty}\sum_{j=0}^k\,((u_0+}
  \  +u_j\cdot(v_0+\dots+v_{j-1}))
 \\
 &=\,\smash[b]{\lim_{k\to\infty}}\cal F^{-1}(\psi_k\cal Fu)\cdot
                               \cal F^{-1}(\psi_k\cal Fv).
 \end{aligned}
 \label{1.11}
 \end{equation}
This proves \eqref{1.10'}.
\end{proof}

Hence, when each $A_0,A_1$ and $A_2$ is chosen independently 
as a Besov space or as a Triebel--Lizorkin space and the
$\pi_j(\cdot,\cdot)$ are continuous $A_0\oplus A_1\to A_2$, it follows
that $\pi_\psi(\cdot,\cdot)\colon A_0\oplus A_1\to A_2$ is a continuous 
bilinear operator.

Then when $A_0\hookrightarrow L_{p}$ and $A_1\hookrightarrow
L_{q}$ with $0\le\fracc{1}{p}+\fracc{1}{q}\le1$, e.g., Proposition~%
\ref{lp-prop} below yields $\pi_\psi=\mu$ (and hence
$\psi$-independence), and thus $\mu\colon A_0\oplus A_1\to A_2$ is continuous.

In general it is necessary to verify that the results obtained for
$\pi_\psi$ by use of \eqref{1.10'} do not depend on $\psi$. 
In Section~\ref{indp-ssect} below it is seen that when the $A_j$ are
Besov or Triebel--Lizorkin spaces,
$\pi_\psi(\cdot,\cdot)\colon A_0\oplus A_1\to A_2$ is an extension by 
continuity of (a restriction of) $\mu$ or of the product on 
$\cal O_M\times\cal S'$. For this reason we shall not emphasise 
the $\psi$-dependence of $\pi_1$, $\pi_2$  and $\pi_3$ in the following.

\bigskip

The applications of Theorems~\ref{Y1-thm} and \ref{Y2-thm} to the 
$\pi_j$ are contained in Section~\ref{estm-sect} below. However,
here it is observed that the spectral 
conditions in Theorems~\ref{Y1-thm} and \ref{Y2-thm} are
satisfied by the terms in the sums in \eqref{1.8}, \eqref{1.9} 
and \eqref{1.10} (recall \eqref{1.16} ff.): 

If $[\xi]\le r\Rightarrow \psi(\xi)=1$ and 
$[\xi]>R\Rightarrow\psi(\xi)=0$, then the identity
$[t^{\Mi}\xi]=t[\xi]$ gives $\supp\psi_k\subset\{\,\xi\mid [\xi]\le
R2^{k}\,\}$ while $\supp\varphi_k\subset\{\,\xi\mid
r2^{k-1}\le[\xi]\le R2^k\,\}$. So for $u$ and $v\in\cal S'$ and $j\in\N_0$,
 \begin{gather} 
 \supp \cal F(u^{j-2} v_j+u_j v^{j-2})\subset
 \bigl\{\,\xi\bigm| |\tfrac{R}{4}-\tfrac{r}{2}|2^k\le[\xi]\le
 \tfrac{5R}{4} 2^k\,\bigr\},
 \label{3.5'}\\[\jot]
 \supp \cal F(u_{j-1} v_j+u_j v_j+u_j v_{j-1})\subset
 \bigl\{\,\xi\bigm| [\xi]\le R2^{k+1}\,\bigr\},
 \label{3.6}
 \end{gather}
since $\supp\cal F(u_j\cdot v_k)\subset\supp\varphi_j\hat u+
\supp\varphi_k\hat v$ and 
$|[\xi]-[\eta]|\le[\xi+\eta]\le[\xi]+[\eta]$. Note that
$R\ne 2r$ yields $|\tfrac{R}{4}-\tfrac{r}{2}|>0$.

\begin{rem} \label{PSI-rem}
It will be convenient later on to assume that $\psi=\Psi_0$, where
$\Psi_0$ is defined in \eqref{1.14}. For this purpose it is observed
that Theorem~\ref{Y1-thm} gives the inequality
 \begin{equation}
 \Norm{ \{2^{sj}\norm{\cal F^{-1}\varphi_j\cal F u}{L_p}\}_{j=0}^\infty
    }{\ell_q}\le c\norm{u}{B^{\Mi,s}_{p,q}}
 \end{equation}
for a constant $c$ independent of $u$ and a similar inequality
for the $F^{\Mi,s}_{p,q}$ spaces.
\end{rem}

\begin{rem}  \label{indp-rem}
The product $\pi(u,v)$ is by definition obtained by a
simultaneous regularisation of both factors. On one hand it may be
seen as in \cite[Ex.~2.3]{Ober} that regularisation of one
factor only gives a limit depending on $\psi$ when one considers the
product of $x_+^{-\frac12}$ and $x_-^{-\frac12}$.
 
On the other hand, this dependence may occasionally
disappear by regularising both
factors by means of the same $\psi$ (as is the case with the
product of $x_+^{-\frac12}$ and $x_-^{-\frac12}$), but even then the limit
$\pi_\psi(u,v)$ may still depend on $\psi$ in some cases. For this
reason $\psi$-independence is required in Definition~\ref{pi-defn}. 

For a thourough discussion of such questions for the
Antosik--Mikusi\'nski--Sikorski product one may consult
\cite[Ch.~2]{Ober}. For this product there is in \cite{ColOber90}
given an example of $\psi$-dependence, namely when multiplying
$H\otimes\delta_0$ and $\delta_0\otimes H$ in dimension $n=2$, 
but the example does not carry over with its conclusions to 
$\pi(\cdot,\cdot)$. 
\end{rem}

\subsection{Relations to other products} \label{pi-ssect}
Recall that the usual product $\cal E(\Rn)\times\cal D'(\Rn)
\to \cal D'(\Rn)$ restricts to a bilinear operator
 \begin{equation}
 \cal O_M(\Rn)\times\cal S'(\Rn)\to \cal S'(\Rn),
 \label{-3.1}
 \end{equation}
where $\cal O_M$ denotes the spaces of {\em slowly increasing\/}
functions. With $\ang{x}=(1+|x|^2)^{\frac{1}{2}}$
 \begin{equation}
 \cal O_M(\Rn)=\bigl\{\,f\in\cal E(\Rn)\bigm| \forall\alpha\in\N_0^n
\exists a,c>0: |D^\alpha f(x)|\le c\ang{x}^a\,\bigr\}.
 \label{-3.2}   
 \end{equation}
($\cal O_M(\Rn)$ consists of the pointwise multipliers of
$\cal S(\Rn)$ and $\cal S'(\Rn)$.)

\begin{prop} \label{slw-prop}
Let $\psi\in\cal S(\Rn)$ with $\psi(0)=1$ and a weight $\M$ be given.
Then 
 \begin{equation}
 \cal F^{-1}(\psi_k\cal Ff)\cdot\cal F^{-1}(\psi_k\cal Fu)
 \xrightarrow[k\to\infty]{\ } fu \quad\text{in}\quad\cal D'(\Rn)
 \label{-3.3}
 \end{equation}
for every $f\in\cal O_M(\Rn)$ and every $u\in\cal S'(\Rn)$. 

In particular $\pi(f,u)$ is defined and equal to $fu$.
\end{prop}
\pagebreak[5]
\begin{proof} From $\ang{x-y}^s\le c_s\ang{x}^s\ang{y}^s$, valid for 
$s>0$, it follows that $f^k=\check\psi_k*f$ is an element of $\cal O_M(\Rn)$, 
so $f^ku^k=(\check\psi_k*f)\cdot
\cal F^{-1}(\psi_k\hat u)$ makes sense in $\cal S'(\Rn)$.

For each $\varphi\in C^\infty_0(\Rn)$, it suffices for the convergence
 \begin{equation}
 \dual{f^ku^k}{\varphi}-\dual{fu}{\varphi}
 =\dual{u^k}{f^k\varphi-f\varphi}+ \dual{u^k-u}{f\varphi}
 \xrightarrow[k\to\infty]{\ } 0, 
 \label{-3.6}
 \end{equation}
that $f^k\varphi\to f\varphi$ in $\cal S(\Rn)$ for
$k\to\infty$, for the family $\{u^k\}_{k\in\N}$ of operators $\cal
S(\Rn)\to \C$ is equicontinuous. Thus we have arrived at a ``linear'' 
problem.

However, that $f^k\varphi\to f\varphi$ in $\cal S(\Rn)$ follows if 
$\sup\bigl\{\,|f^k(x)-f(x)|\bigm| x\in K\,\bigr\}\to 0$
for $k\to\infty$ for every $f\in\cal O_M$, when $K:=\supp\varphi$. 
Concerning the latter convergence we may assume $f$ to be real so that   
 \begin{equation}
 |\check\psi_k*f(x)-f(x)|
 \le\int|\check\psi(z)||\grad f(x-\theta(z)2^{-k\Mi}z)\cdot2^{-k\Mi}z)|\,dz,
 \label{-3.7}
 \end{equation}
where $|2^{-k\Mi}z|\le 2^{-k}|z|$.
Since $\theta(z)\in\,]0,1[\,$, the estimate $|\grad
f(x-\theta(z)2^{-k\Mi}z)|\le c_N\ang{x}^N\ang{\theta(z)2^{-k\Mi}z}^N\le
c_N\ang{x}^N\ang{z}^N$ holds for a big $N\in\N$. Hence
 \begin{equation}
 \sup_{x\in K}|f^k(x)-f(x)|\le  2^{-k}c_N\sup_{x\in K}\ang{x}^N 
 \int\ang{z}^{N+1}|\check\psi(z)|\,dz,
 \label{-3.8}   
 \end{equation}
and it is seen that the left hand side tends to zero for $k\to\infty$.
\end{proof}

The result above generalises the observation made in \cite{S} that
$\pi(f,u)=fu$ when $f\in\cal S(\Rn)$ and $u\in\cal S'(\Rn)$.

Since $u$ and $v$ in Definition~\ref{pi-defn} are assumed only to lie
in $\cal S'(\Rn)$, one may now ask for stricter conditions on $u$
which allows $f$ to be more general than an element of the space $\cal
O_M(\Rn)$.

Carried to the extreme, when $u\in L_{p,\op{loc}}\cap\cal S'$ it is possible to
take $f\in L_{q,\op{loc}}\cap\cal S'$ provided only that $\fracp+\fracc1q\le1$.
Before we show this in Proposition~\ref{lp-prop} below,
the next result on a local property (of $\pi$) is included as a preparation.

\begin{prop} \label{local-prop} Let $\psi\in\cal S(\Rn)$ with
$\psi(0)=1$ and let $\M$ be a weight.
 
If $u$ and $v\in\cal S'(\Rn)$ and $\Omega\subset\Rn$ is an open set 
such that either $r_\Omega u=0$ or $r_\Omega v=0$, then
 \begin{equation}
 r_\Omega\cal F^{-1}(\psi_k\cal Fu)\cdot\cal F^{-1}(\psi_k\cal F v)
 \to 0\quad\text{ in}\quad\cal D'(\Omega)\quad\text{ for}\quad k\to\infty.
 \label{3.1'}
 \end{equation}
In particular, $r_\Omega\pi(u,v)=0$ when $\pi(u,v)$ is defined.
\end{prop}

\begin{proof} It can be assumed that $r_\Omega u=0$, for if not the
roles of $u$ and $v$ can be interchanged. It suffices 
for each $\varphi\in e_\Omega C^\infty_0(\Omega)$ to show the convergence
 \begin{equation}
 \dual{r_\Omega u^kv^k}{r_\Omega\varphi}=
 \dual{v^k}{u^k\varphi}\xrightarrow[k\to\infty]{\ }0.
 \label{3.1}
 \end{equation}
By equicontinuity and the relation $\dual{v^k}{u^k\varphi}=
\dual{v^k-v}{u^k\varphi}+\dual{v}{u^k\varphi}$, \eqref{3.1} follows if
$u^k\varphi\to0$ in $\cal S(\Rn)$ for $k\to\infty$.

For completeness' sake we supply a proof of the fact that
$(\psi_k*u)\cdot\varphi\to 0$ in $\cal S(\Rn)$ for $k\to\infty$, 
when $u\in\cal S'(\Rn)$ and $\psi\in\cal S(\Rn)$ are arbitrary such
that $r_\Omega u=0$, and $\psi_k(x)=2^{k|\Mi|}\psi(2^{k\Mi}x)$.
It suffices to show, for $K=\supp\varphi\subset\Omega$, that
 \begin{equation}
 \sup\bigl\{\,|\psi_k*u|\bigm| x\in K\,\bigr\}\to0\quad\text{ for}\quad
 k\to\infty,
 \label{3.3}
 \end{equation}
for then $\norm{(\psi_k*u)\cdot\varphi}{\cal S,\alpha,\beta}\to0$ for
every $\alpha$ and $\beta$. 

With closed sets $F_1$ and $F_2$ satisfying 
$K\subset F_1^\circ\subset F_1\subset F_2^\circ\subset
F_2\subset\Omega$, we take $\eta\in C^\infty(\Rn)$ such that $\eta=1$ on
$\Rn\setminus F_2^\circ$ and $\eta=0$ on $F_1$. Then
 \begin{equation}
 |\psi_k*u|=|\ang{u,\eta\cdot\psi_k(x-\cdot)}|
 \le\,c \norm{\eta\psi_k(x-\cdot)}{\cal S,\alpha,\beta}
 \label{3.4}
 \end{equation}
For each $x\in K$ it follows here that $\eta(y)\psi_k(x-y)=0$ when 
$|x-y|<a:=\op{dist}(K,\Rn\setminus F_1^\circ)$. Because 
$2^k|\cdot|\le|2^{k\Mi}\cdot|$ one has 
$|z|<2^ka\Rightarrow |x-y|<a$, when $z=2^{k\Mi}(x-y)$. Moreover 
$2^{kb}|x-y|^b\le|z|^b$ holds for $b>0$, so for each $\gamma\le\beta$ 
\begin{equation}
 \begin{split}
 |y^\alpha D^{\beta-\gamma}_y\eta D^\gamma_y\psi_k(x-\cdot))|&\le
 |D^{\beta-\gamma}\eta|\cdot2^{|\alpha|}(|x-y|^{|\alpha|}+|x|^{|\alpha|})
 \\
 &\hphantom{\le|D^{\beta-\gamma}}
 \times 2^{k(|\Mi|+\Mi\cdot\gamma)}|D^\gamma\psi(z)|\cdot\kd{|z|\ge2^{k}a}
  \\
 &\le\,2^{|\alpha|}a^{-|\Mi|-\Mi\cdot\gamma}
 \norm{\eta}{C^{|\beta|}}\norm{\psi}{\cal S,|\alpha|+\M\cdot\gamma+|\M|+1}
  \\
 &\hphantom{\le|D^{\beta-\gamma}}
 \times(2^{-k|\alpha|}+\sup_{x\in K}|x|^{|\alpha|})
 \sup_{z\in\Rn}|z|^{-1}\kd{|z|\ge2^ka}.
 \end{split}  
 \label{3.5}
\end{equation}
Because the last factor $\to0$ for $k\to\infty$ the convergence in
\eqref{3.3} is inferred from Leibniz' formula, \eqref{3.5} and \eqref{3.4}.
 \end{proof}

In Section \ref{subset-sect} the full generality of the result 
above will be used 
to define the product $\pi$ on an open set $\Omega$. Here Proposition
 \ref{local-prop} is used to prove the next result, where the simple case 
with $f\in L_p$, $g\in L_q$ and $\fracp+\fracc1q=1$ is included in \cite{S}.

\begin{prop} \label{lp-prop}
Let $\psi\in\cal S(\Rn)$ with
$\psi(0)=1$ and a weight $\M$ be given.
 
For $f\in L_{p,\op{loc}}(\Rn)\cap\cal S'(\Rn)$ and 
$g\in L_{q,\op{loc}}(\Rn)\cap\cal S'(\Rn)$ such that 
$\fracc1r:=\fracp+\fracc1q\le1$ there is convergence $f^kg^k\to fg$ in
$\cal D'(\Rn)$.
In particular $\pi(f,g)$ is defined and
 \begin{equation}
 \pi(f,g)=f(x)\cdot g(x)\in L_{r,\op{loc}}
 \label{3.2'}
 \end{equation}
for such $f$ and $g$.
\end{prop}

\begin{proof} First we assume that $f\in L_p$, $g\in L_q$ and that 
$p$ and $q<\infty$. It follows (by inspection of the usual convolution
proofs, where $\M=(1,\dots,1)$ and $\cal F^{-1}\psi\ge0$) 
that $f^k\to f$ in $L_p$ when $f\in L_p$ for $1\le p<\infty$.
Similarly  $g^k\to g$ in $L_q$.  Then
 \begin{equation}
 f^k g^k-f g=  (f^k-f)(g^k-g)+(f^k-f) g+f(g^k-g)
 \label{3.3'}
 \end{equation}
and \eqref{1.2} imply that $\pi(f,g)$ is defined as an element
of $L_r\hookrightarrow\cal D'(\Rn)$.

In general, let $\varphi\in C_0^\infty(\Rn)$ be given
and take $\eta_0\in C_0^\infty(\Rn)$ such that $\eta_0=1$ on
$\Omega\supset\supp \varphi$, where $\Omega$ is open and bounded. Then
it is found with $\eta_1=1-\eta_0$ that 
 \begin{equation}
  r_\Omega f^kg^k=r_\Omega(\eta_0 f)^k(\eta_0 g)^k+
 r_\Omega\bigl((\eta_1 f)^k(\eta_0 g)^k
 +f^k(\eta_1 g)^k \bigr),
 \label{3.4'} 
 \end{equation}
where the second term on the right hand side goes to zero in
$\cal D'(\Omega)$ by Proposition~\ref{local-prop}, 
while the first term converges to 
$r_\Omega(\eta_0 f\cdot\eta_0 g)$ in virtue of the special case treated above. 
Hence $\dual{f^kg^k}{\varphi}\to \dual{f\cdot g}{\varphi}$.
\end{proof}

Propositions \ref{slw-prop} and \ref{lp-prop} are used in
Section~\ref{indp-ssect} below to obtain $\psi$-independence in  
connection with the application of para-multiplication.

\begin{rem} \label{cmpr-rem}
In Definition~\ref{pi-defn} above the mollifiers $\psi$ were required
to have compact support and to equal 1 on a neighbourhood of the
origin. This is because the validity of the spectral conditions
in formulae \eqref{3.5'}--\eqref{3.6} (cf. the numbers $r$ and $R$ 
there) is crucial for the
application of Yamazaki's theorems to the operators
$\pi_j(\cdot,\cdot)$ in Sections~\ref{estm-sect} and \ref{suff-sect} below,
cf.~also Lemma~\ref{pi-lem}.

An introduction of further restrictions on the $\psi$'s, say, positivity or
dependence on $[\xi]$ alone, leads to a product
defined on a larger subset of $\cal S'(\Rn)\times\cal S'(\Rn)$, of
course. However, \eqref{3.5'} and \eqref{3.6} remain valid, so the
method of para-multiplication applies to such more general products also.

In contrast to this, by taking $\psi$ merely in $\cal S(\Rn)$ or just
satisfying $\psi(0)=1$ one obtains a product that is a restriction
of $\pi(\cdot,\cdot)$. To this restriction Propositions~\ref{slw-prop} and
\ref{lp-prop} apply, but since \eqref{3.5'} and \eqref{3.6} do not
hold, it is not clear whether para-multiplication may be used to
analyse such products. 

The observations above serve as a justification of the definition of
$\pi(\cdot,\cdot)$, which may also be seen as an $\cal S'$-version of
the Antosik--Mikusi\'nski--Sikorski model product. The relation
to this product, or to its restrictions such as the duality product, 
the wave front product etc., is not clear. The reader is referred to 
\name{Oberguggenberger's} book \cite{Ober} and \name{A.~Kaminski}
\cite{Kam82} and the references there. 
\end{rem}

\section{Necessary conditions for multiplication} \label{ncss-sect}

In this section conditions necessary for boundedness of 
$\mu$ and $\pi$ as bilinear operators $A_0\oplus A_1\to A_2$ are proved. 
But first we include a lemma concerning  the existence of  auxiliary 
functions with convenient norms.

To prepare for this, observe that there exist $\rho$,
$\theta$ and $\omega$ in $\cal S(\Rn)\setminus\{0\}$ for which:
 \begin{equation}
 \begin{aligned}
 \supp \cal F\theta &\subset\{\,\xi\mid[\xi]\le\tfrac1{20}\,\},\quad
 \text{ and $\theta(0)=1$}, \\
 \supp \cal F\!\rho &\subset\{\,\xi\mid\tfrac34\le[\xi]\le1\,\},\quad
 \text{ and $\rho$ is real valued}, \\
 \supp \cal F \omega &\subset\{\,\xi\mid\tfrac34\le[\xi]\le1\,\}\cap B,
 \quad\text{ and $\omega(0)=1$}, 
 \end{aligned}
 \label{2.1}
 \end{equation}
for $B=\{\,\xi\mid[\xi-\zeta]\le\tfrac3{10}\,\}$, with $\zeta=
(\zeta_j)_{j=1,\dots,n}$ for $\zeta_j=\kd{j=j_0}$, where $j_0$
is chosen so that $m_{j_0}=1$ (cf.~the assumptions on $\M$).

The functions may, for example, be constructed 
as $\hat\theta(\xi)=\Psi(26[\xi])(\int\Psi(26[\cdot]))^{-1}$,
$\hat\rho(\xi)=\Psi(\tfrac{13}{10}[\xi])-\Psi(\tfrac{44}{30}[\xi])$ 
and $\hat\omega(\xi)=\chi(\xi)\hat\rho(\xi)\Psi(\tfrac{13}{3}[\xi-\zeta])
(\int\chi\hat\rho\Psi(\tfrac{13}{3}[\cdot-\zeta]))^{-1}$ for some
$\chi$ in $C^\infty_0(\Rn)$.
Indeed, $\rho$ is real-valued if and only if $\hat\rho(-\xi)=
\overline{\hat\rho(\xi)}$. If $0<\Psi(t)<1$ whenever $t$ is in 
$\,]\frac{11}{10},\frac{13}{10}[\,$, as we may assume, 
this condition is satisfied, and for
$\xi=\frac{11}{13}\zeta$ the factor $\hat\rho(\xi)$ is $>0$ and
$\Psi(\frac{13}{3}[\xi-\zeta])=1$, so
$\int\chi\hat\rho\Psi(\frac{13}{3}[\cdot-\zeta])\ne0$ for some $\chi\in
C^\infty_0(\Rn)$.

For such $\rho$, $\omega$ and $\theta$ we denote
 \begin{equation}
  \begin{aligned}
 \rho_k(x)&=\rho(2^{k\Mi}x) \qquad(k\in\Z), \qquad
 \omega_k(x)=\omega(2^{k\Mi}x) \qquad(k\in\N)\\
 \theta_k(x)&=\theta(x)\exp(i\operatorname{sgn}k\,2^{|k|}x_{j_0}) 
 \qquad(k\in\Z),
 \end{aligned}
 \label{2.1'}
 \end{equation}
whereby $\theta_0$ should be read as $\theta$.

\begin{lem} \label{ncss-lem} 
$1^\circ$ For any admissible  $s$, $p$ and $q$, the functions 
$\rho_k$, $\omega_k$ and $\theta_k$ satisfy
 \begin{equation}
 \begin{alignedat}{4}
 &\norm{\rho_k}{F^{\Mi,s}_{p,q}}&=&\norm{\rho_k}{B^{\Mi,s}_{p,q}}&=
 &\norm{\rho}{L_{p}}2^{k(s-\fracci{|\Mi|}{p})}, &\quad(k&\in\N) \\
 &\norm{\rho_k}{F^{\Mi,s}_{p,q}}&=&\norm{\rho_k}{B^{\Mi,s}_{p,q}}&=
 &\norm{\rho}{L_{p}}2^{-k\fracci{|\Mi|}{p}}, &\quad(-k&\in\N_0) \\
 &\norm{\rho^2_k}{F^{\Mi,s}_{p,q}}&=&\norm{\rho^2_k}{B^{\Mi,s}_{p,q}}&=
 &\norm{\rho^2}{L_{p}}2^{-k\fracci{|\Mi|}{p}}, &\quad(-k&\in\N) \\
 &\norm{\omega_k}{F^{\Mi,s}_{p,q}}&=&\norm{\omega_k}{B^{\Mi,s}_{p,q}}&=
 &\norm{\omega}{L_{p}}2^{k(s-\fracci{|\Mi|}{p})}, && \\
 &\norm{\theta_k}{F^{\Mi,s}_{p,q}}&=&\norm{\theta_k}{B^{\Mi,s}_{p,q}}&=
 &\norm{\theta}{L_{p}}2^{|k|s}, &&\\
 &\norm{\omega_k^2}{F^{\Mi,s}_{p,q}}&=&\norm{\omega_k^2}{B^{\Mi,s}_{p,q}}&=
 &\norm{\omega^2}{L_{p}}2^{k(s-\fracci{|\Mi|}{p})+s}, &&\\
 &\norm{\theta\theta_k}{F^{\Mi,s}_{p,q}}&=
 &\norm{\theta\theta_k}{B^{\Mi,s}_{p,q}}&=
 &\norm{\theta^2}{L_{p}}2^{|k|s}, &&\\
 &\norm{\theta\omega_k}{F^{\Mi,s}_{p,q}}\:&=
 &\norm{\theta\omega_k}{B^{\Mi,s}_{p,q}}\:&=
 &\norm{\omega\theta(2^{-k\Mi}\cdot)}{L_{p}}2^{k(s-\fracci{|\Mi|}{p})};  &&
 \end{alignedat}
 \label{2.2}
 \end{equation}
and for any $p\in\,]0,\infty]$ one has moreover
 \begin{equation}
 \lim_{k\to\infty} 2^{k|\Mi|}\cal F^{-1}\Phi_0\cal F(\rho_k^2)(x)=
 \norm{\rho^2}{L_1}\cal F^{-1}\Phi_0(x)\quad\text{ in}\quad L_{p}(\Rn).
 \label{2.3}
 \end{equation}

$2^\circ$ The functions $\theta^{(t)}_{N,\pm}$, $\rho^{(t)}_{N,l}$, 
$\omega^{(t)}_{N}$ and $\Omega^{(t)}_N$, given by 
$ \theta^{(t)}_{N,\pm}=\sum_{k=1}^{N}2^{-kt}\theta_{\pm k}$ and 
 \begin{equation}
 \rho^{(t)}_{N,l}=\sum_{k=l+1}^{l+N}2^{-kt}\rho_k,
 \qquad
 \omega^{(t)}_{N}=\sum_{k=N+1}^{2N}2^{-kt} \omega_k,
 \qquad
 \Omega^{(t)}_{N}=\sum_{k=2}^{N+1}2^{-2^kt} \omega_{2^k},
 \label{2.4}
 \end{equation}
have for the indicated values of $t$ norms with the
characterisations (for $N\ge1$ respectively $N\ge4$ in the last line)
 \begin{equation}
 \begin{alignedat}{2}
 &\norm{\rho^{(s-\fracci{|\Mi|}{p})}_{N,l}}{B^{\Mi,s}_{p,q}}&=
 &\norm{\rho}{L_{p}}N^{\fracci{1}{q}},
 \\
 \norm{\omega^{(s-\fracci{|\Mi|}{p})}_{N}}{B^{\Mi,s}_{p,q}}=&
 \norm{\Omega^{(s-\fracci{|\Mi|}{p})}_{N}}{B^{\Mi,s}_{p,q}}\;&=&
 \norm{\omega}{L_{p}}N^{\fracci{1}{q}},
 \\
 \norm{\theta^{(s)}_{N,\pm}}{F^{\Mi,s}_{p,q}}
 =&\norm{\theta^{(s)}_{N,\pm}}{B^{\Mi,s}_{p,q}}&=
 &\norm{\theta}{L_{p}}N^{\fracci{1}{q}},
 \\
 \norm{\theta\theta^{(s)}_{N,\pm}}{F^{\Mi,s}_{p,q}}
 =&\norm{\theta\theta^{(s)}_{N,\pm}}{B^{\Mi,s}_{p,q}}&=
 &\norm{\theta^2}{L_{p}}N^{\fracci{1}{q}},
 \\
 \norm{\omega_{3N}\omega^{(0)}_{N}}{F^{\Mi,s}_{p,q}}=
 &\norm{\omega_{3N}\omega^{(0)}_{N}}{B^{\Mi,s}_{p,q}}&=
 &  \norm{
  \begin{smash}[b]
 \omega\sum_{k=N}^{2N-1}\omega(2^{-k\Mi}\cdot)
 \end{smash} }{L_{p}}  
 2^{3N(s-\fracci{|\Mi|}{p})}.
 \end{alignedat}
 \label{2.5}
 \end{equation}

Furthermore, for $2\le k\le N+1$ respectively for each $p\in\,]0,\infty]$
 \begin{gather}
 \cal F^{-1}\Phi_{2^k+1} \cal F(\Omega^{(t_0)}_N\Omega^{(t_1)}_N)
 =2^{-2^k(t_0+t_1)}\omega^2_{2^k}
 \label{2.5''} \\
 \lim_{N\to\infty}\omega N^{-1}\sum_{k=N}^{2N-1}
 \omega(2^{-k\Mi}\cdot) =\omega\quad\text{ in $L_p$},
 \label{2.6'}
 \end{gather}
and under the conditions 
$t_0+t_1=-|\M|$ respectively $s_0+s_1=0$, with $N\in \N$,
 \begin{equation}
 \begin{aligned}
 \lim_{l\to\infty}\cal F^{-1}\Phi_0\cal F(\rho^{(t_0)}_{N,l}
 \cdot\rho^{(t_1)}_{N,l})&=
 N\norm{\rho}{L_2}^2\cal F^{-1}\Phi_0\quad  \text{in $L_{p}$},   \\
 \cal F^{-1}\Phi _0\cal F(\theta _{N,+}^{(s_0)}\cdot\theta _{N,-}^{(s_1)})
 &=N\theta^2.
 \end{aligned}
 \label{2.6}
 \end{equation}
\end{lem}

\begin{proof} $1^\circ$ The main point is to show the relations
 \begin{equation}
 \begin{alignedat}{3}
 \supp \cal F\rho_k&\subset\{\,\xi\mid\tfrac{15}{20}2^k\le[\xi]\le2^k\,\}
 &&\subset\{\,\xi\mid\Phi_{k_+}(\xi)=1\,\},\quad&& \\
 \supp \cal F\rho_k^2&\subset\{\,\xi\mid [\xi]\le2^{k+1}\,\}
 &&\subset\{\,\xi\mid\Phi_0(\xi)=1\,\},\quad&&\text{ for }\,-k\in\N, \\
 \supp \cal F\omega_k&\subset\{\,\xi\mid\tfrac{15}{20}2^k\le[\xi]\le2^k\,\}
 &&\subset\{\,\xi\mid\Phi_k(\xi)=1\,\},&& \\
 \supp \cal F\theta_k&\subset\{\,\xi\mid\tfrac{19}{20}2^{|k|}\le[\xi]\le
 \tfrac{21}{20}2^{|k|}\,\}&&\subset\{\,\xi\mid\Phi_{|k|}(\xi)=1\,\},&& \\
 \supp\cal F\omega_k^2&\subset\{\,\xi\mid\tfrac{28}{20}2^{k}\le[\xi]
 \le2^{k+1}\,\}&&\subset\{\,\xi\mid\Phi_{k+1}(\xi)=1\,\},&& \\
 \supp\cal F(\theta\theta_k)&\subset\{\,\xi\mid\tfrac{18}{20}2^{|k|}\le[\xi]
 \le\tfrac{22}{20}2^{|k|}\,\}&&\subset\{\,\xi\mid\Phi_{|k|}(\xi)=1\,\},
  \quad&&\text{ for }\,k\ne0, \\
 \supp\cal F(\theta\omega_k)&\subset\{\,\xi\mid\tfrac{14}{20}2^{k}\le[\xi]
 \le\tfrac{21}{20}2^{k}\,\}&&\subset\{\,\xi\mid\Phi_k(\xi)=1\,\}.&&
 \end{alignedat}
 \label{2.7}
 \end{equation}
Indeed, \eqref{2.7} gives $\cal F^{-1}(\Phi_l\hat\omega_k)=
\kd{k=l}\omega_k$ etc., so that
$\norm{\omega_k}{F^{\Mi,s}_{p,q}}=\norm{2^{sk}\omega_k}{L_p}$ etc. Taking the
dilations or the exponential factors into account \eqref{2.2} follows.

The support conditions on $\rho_k$ and $\omega_k$
follow from \eqref{2.1} since $[t^\Mi\xi]=t[\xi]$ and
$\hat\rho_k(\xi)=2^{-k|\Mi|}\hat\rho(2^{-k\Mi}\xi)$; and by \eqref{1.15}
$\{\,\xi\mid\Phi_k(\xi)=1\,\}\supset\{\,\xi\mid\tfrac{13}{20}2^k\le
[\xi]\le\tfrac{22}{20}2^k\,\}$ for $k\ge1$. 
For $\omega^2_k$ and $\theta\omega_k$ one can use  that 
$\supp \cal F(\varphi\psi)\subset\supp\hat\varphi+\supp\hat\psi$ and that 
$|[\xi]-[\eta]|\le[\xi+\eta]\le[\xi]+[\eta]$. Indeed, for
$\xi',\xi''\in\supp\hat\omega_k$
one has $2^{-k\Mi}\xi^{(j)}=\zeta+\eta^{(j)}$ with $\eta',\eta''\in B$,
hence $[\xi'+\xi'']\le[\xi']+[\xi'']\le2^{k+1}$ and $[\xi'+\xi'']\ge
2^{k}([2\zeta]-[\eta'+\eta''])\ge2^k(2-2\tfrac{3}{10})=\tfrac{14}{10}2^k$.
(The definition of $\zeta$ gives $[2\zeta]=2$ here.) 
$\theta_k$ and $\theta\theta_k$ are treated along the
latter lines, since their spectra are obtained from the~case $k=0$
by translation in both directions along the $\xi_{j_0}$-axis.

To obtain \eqref{2.3} note that there is strong convergence
$2^{k|\Mi|}\rho_k^2(\int\!\rho^2)^{-1}*\cdot\to1$ on $C(\Rn)$
($\int\rho^2>0$ follows since $\rho$ is real).
This gives \eqref{2.3} for $p=\infty$. For $p<\infty$ we use the pointwise
convergence thus shown together with $|\check\Phi_0\int\!\rho^2-2^{k|\Mi|}
\check\Phi_0*\rho_k^2|^p\le2^p|\check\Phi_0\int\!\rho^2|^p+2^p
|2^{k|\Mi|}\check\Phi_0*\rho_k^2|^p$ and the following majorisation,
 \begin{equation}
 \begin{aligned}
 \bigl| 2^{k|\Mi|}\cal F^{-1}\Phi_0\cal F(\rho_k^2)(x)\bigr|
 &\le\langle x\rangle^{-\frac{n+1}{p}}\sup_x\bigl|\cal F^{-1}
 ((1-\lap_\xi)^N\Phi_0 2^{k|\Mi|}\cal F(\rho_k^2))(x)\bigr|\\
 &\le\,\langle x\rangle^{-\frac{n+1}{p}}
 \sum_{|\alpha|,|\beta|\le 2N}c_{\alpha,\beta}
 \norm{x^\beta\rho^2}{L_1}\int|D^\alpha\Phi_0|\,d\xi,   
 \end{aligned}
 \label{2.8}
 \end{equation}
where $\langle x\rangle=(1+|x|^2)^{\frac12}$ and $N\ge\frac{n+1}{2p}$.

$2^\circ$ \eqref{2.5} is obtained like \eqref{2.2} from 
\eqref{2.7} and the norm definitions.
In particular the $k^{th}$ summand in $\theta\theta^{(s)}_{N,\pm}$ has
spectrum in $\{\,\xi\mid\Phi_k(\xi)=1\,\}$, implying, e.g.,
 \begin{equation}
 \norm{\theta\theta^{(s)}_{N,\pm}}{B^{\Mi,s}_{p,q}}=
 \bigl(\sum_{k=1}^N(2^{ks}\norm{2^{-ks}\theta^2e^{i2^kx_{j_0}}}{L_p})^q
 \bigr)^{\fracci{1}{q}}=\norm{\theta^2}{L_p}N^{\fracci{1}{q}}.
 \label{2.9}
 \end{equation}
Concerning $\omega_{3N}\omega^{(0)}_N$, each term
$\omega_{3N}\omega_k$ with $N+1\le k\le 2N$ can be treated like
$\omega^2_k$ in \eqref{2.7}, thus $\supp\cal
F(\omega_{3N}\omega_k)\subset \{\,\xi\mid
2^{3N}(\tfrac{15}{20}-2^{k-3N})\le [\xi]\le 2^{3N}(1+2^{k-3N})\,\}$.
This set is contained in $\{\,\xi\mid \Phi_{3N}(\xi)=1\,\}$ for $N\ge4$,
so \eqref{2.5} follows.

The product $\Omega^{(t)}_N\Omega^{(\tau)}_N$ consists in part of
terms $2^{-2^k(t+\tau)}\omega^2(2^{(2^k)\Mi}x)$ with spectrum in 
$\{\,\xi\mid \Phi_{2^k+1}(\xi)=1\,\}$, cf.~\eqref{2.7}, and in
part of terms stemming from $\omega_{2^j}\omega_{2^k}$ with $j<k$.
Since $2^k\ge 2^{j+1}\ge 2^j +4$ (because $j\ge2$) it is found that 
$\supp\cal F(\omega_{2^j}\omega_{2^k})\subset
\{\,\xi\mid \Phi_{2^k}(\xi)=1\,\}$, and here $2^k\ne 2^l+1$ holds for
all $k$ and $l\ge2$. 
Hence $\cal F^{-1}\Phi_{2^l+1}\cal F(\omega_{2^j}\omega_{2^k})$ equals $0$ if
$j\ne k$ and if $j=k\ne l$. 

The limit in \eqref{2.6'} is found by majorisation using $\omega(0)=1$.
When $t_0+t_1=-|\M|$ we have for $l\to \infty$ by use of \eqref{2.3}
that in the topology of $L_p$,
 \begin{equation}
 \cal F^{-1}\Phi_0\cal F(\rho^{(t_0)}_{N,l}\rho^{(t_1)}_{N,l})
 =\,\sum_{j=l+1}^{l+N}2^{j|\Mi|}\cal F^{-1}\Phi_0\cal F\rho_j^2 
 \to\, N\norm{\rho}{L_2}^2\cal F^{-1}\Phi_0.
 \label{2.11}
 \end{equation}
Indeed, for $j\ne k$ the spectra of $\rho_j\rho_k$ and $\supp
\Phi_0$ are disjoint: E.g., for $j\ge k+1$ any element of $\supp
\cal F(\rho_j\rho_k)$ is of the form $\xi_j+\xi_k$, with $\xi_m\in
\supp\cal F\rho_m$, for which $[\xi_j+\xi_k]\ge |\tfrac34 2^j-2^k|\ge
\tfrac12 2^k\ge\tfrac{13}{10}$ for $k\ge2$.
For $s_0+s_1=0$ it is seen that
 \begin{equation}
 \cal F^{-1}\Phi _0\cal F(\theta _{N,+}^{(s_0)}\theta _{N,-}^{(s_1)})=
 \cal F^{-1}\Phi _0\cal F\big(\sum_{k,l=1}^{N}2^{-ks_0-ls_1}
 e^{(i(2^{k}-2^l)x_{j_0})} \theta^2\big)
 =N\theta^2,
 \label{2.12}
 \end{equation}
since the terms with $l\ne k$ in the sum have their spectrum disjoint from
$\supp \Phi _0$, while those with $l=k$ have their spectrum in
$\{\,\xi \mid \Phi _0(\xi )=1\,\}$. 
\end{proof}

The present versions of the functions $\rho_k$, $\theta_k$ and
$\theta^{(t)}_{N,\pm}$, that via \cite{S} and \cite{F3} go back at
least to \cite[2.3.9]{T2}, are introduced in order to obtain
greater clarity via characterisations of the norms rather than
estimates. The other functions are
introduced to show some of the new parts of Theorem~\ref{ncss-thm}
below, and so is the technique of considering the limits in
\eqref{2.3}, \eqref{2.6'} and \eqref{2.6}.

In the next result (1), $(1')$, (2) and $(2')$ set limits for the 
admissible spaces $A_0\oplus A_1$, while (3)--(7) etc.~restrict the 
best obtainable $A_2$, cf.~(II) and (III).

\begin{thm} \label{ncss-thm} 
If there exists a constant $c<\infty$ such that the inequality
 \begin{equation}
 \norm{f\cdot g}{A_2}\,\le \,c\norm f{A_0}\norm g{A_1}
 \quad\text{ holds for all $f$ and $g\in\cal S(\Rn)$},
 \label{2.13}
 \end{equation}
where $A_j=B^{\Mi,s_j}_{p_j,q_j}(\Rn)$ or $A_j=F^{\Mi,s_j}_{p_j,q_j}(\Rn)$ 
for $j=0$, $1$ and $2$, then it follows  that 
 \begin{equation}
 \begin{alignedat}{2}
 &\text{\rom{(1)}}&\quad s_0+s_1&\ge|\M|(\fracc1{p_0}+\fracc1{p_1}-1),  \\
 &\text{\rom{(2)}}&\quad s_0+s_1&\ge0,  \\
 &\text{\rom{(3)}}&\quad s_2&\le \min(s_0,s_1),\\
 &\text{\rom{(4)}}&\quad \fracc1{p_2}&\le \fracc1{p_0}+\fracc1{p_1},  \\
 &\text{\rom{(5)}}&\quad s_2-\fracc{|\M|}{p_2}&\le
                       \min(s_0-\fracc{|\M|}{p_0},s_1-\fracc{|\M|}{p_1}) \\
 &\text{\rom{(6)}}&\quad s_2-\fracc{|\M|}{p_2}&\le
 s_0+s_1-|\M|(\fracc1{p_0}+\fracc1{p_1}),\\
 &\text{\rom{(7)}}&\quad s_2-\fracc{|\M|}{p_2}&=s_1-\fracc{|\M|}{p_1}
  \quad\text{and}\quad s_0=\fracc{|\M|}{p_0}\\
 &&&\text{ implies}\ \left\{
  \begin{aligned}
   q_0&\le1\text{ in $B\bigdot\bigdot$ cases},\\ 
   p_0&\le1\text{ in $F\bigdot\bigdot$ cases}.
  \end{aligned} \right.
 \end{alignedat}
 \label{2.14}
 \end{equation}
Furthermore it also follows, for $j=0$ respectively $j=1$, that 
 \begin{equation}
 \begin{alignedat}{3}
 (1')&\quad & s_0+s_1&=\fracc{|\M|}{p_0}+\fracc{|\M|}{p_1}-|\M| 
    &&\ \text{ implies}
  \begin{cases}
  \fracc{1}{q_0}+\fracc{1}{q_1}\ge1 \ \text{in $BB\bigdot$ cases},\\
  \fracc{1}{q_0}+\fracc{1}{p_1}\ge1 \ \text{in $BF\bigdot$ cases};
  \end{cases}
 \\
 (2')&\quad & s_0+s_1&=0 &&\ \text{ implies} 
 \ \fracc1{q_0}+\fracc1{q_1}\ge1, 
 \\
 (3')&\quad & s_2&=s_j &&\ \text{ implies}\ q_2\ge q_j, 
 \\
 (5')&\quad & s_2-\fracc{|\M|}{p_2}&=s_j-\fracc{|\M|}{p_j} &&\ \text{ implies}
 \ q_2\ge q_j \ \text{in $B{\bigdot}B$ resp. $\bigdot BB$ cases}, 
 \\
 (6')&\quad & s_2-\fracc{|\M|}{p_2}&=s_0-\fracc{|\M|}{p_0}+s_1-
     \fracc{|\M|}{p_1}&&\ \text{ implies}\ q_2\ge 
     (\fracc1{q_0}+\fracc1{q_1})^{-1} \ \text{in $BBB$ cases}.
 \end{alignedat}
 \label{2.15}
 \end{equation}
By Proposition~\ref{slw-prop}, the same conclusions can be drawn 
for $\pi$ when it satisfies \eqref{2.13}.
\end{thm}

\begin{proof} Observe first, that in any case one
has $\norm{\cal F^{-1}\Phi_0\cal Fu}{L_{p_2}}\le\norm u{A_2}$
by the definition of $\norm{\cdot}{B^{\Mi,s}_{p,q}}$ 
and $\norm{\cdot}{F^{\Mi,s}_{p,q}}$.

By application of \eqref{2.2} to \eqref{2.13} it follows that
 \begin{equation}
 \begin{aligned}
 \norm{\cal F^{-1}\Phi_0\cal F\rho_k^2}{L_{p_2}}&\le
 \norm{\rho_k^2}{A_2}\le c\norm{\rho_k}{A_0}\norm{\rho_k}{A_1}\,\\
 &\le\,c\norm{\rho}{L_{p_0}}\norm{\rho}{L_{p_1}}
 2^{k(s_0+s_1-|\M|(\fracci1{p_0}+\fracci1{p_1}))},   
 \end{aligned}
 \label{2.16}
 \end{equation}
and  taken together, since $\rho\ne 0$, \eqref{2.3} and \eqref{2.16} show that 
 \begin{align}
 0\,<\, \norm{\cal F^{-1}\Phi_0}{L_{p_2}}\norm\rho{L_2}^2
 &=\liminf2^{k|\Mi|}\norm{\cal F^{-1}\Phi_0\cal F\rho_k^2}{L_{p_2}}
 \notag \\
 &\le\,c\norm{\rho}{L_{p_0}}\norm{\rho}{L_{p_1}}
 \liminf2^{k(s_0+s_1-|\Mi|(\fracci1{p_0}+\fracci1{p_1}-1))}. 
 \label{2.17}
 \end{align}
Here $s_0+s_1-|\M|(\fracc1{p_0}+\fracc1{p_1}-1)<0$ would be
absurd, so (1) in \eqref{2.14} follows.

For $s_0+s_1=|\M|(\fracc1{p_0}+\fracc{1}{p_1}-1)$ we  conclude 
from \eqref{2.13} that, with $t_j=s_j-\fracc{|\M|}{p_j}$,
 \begin{equation}
 \norm{\check\Phi_0*(\rho^{(t_0)}_{N,l}
 \rho^{(t_1)}_{N,l})}{L_{p_2}}\le
 \norm{\rho^{(t_0)}_{N,l}\rho^{(t_1)}_{N,l}}{A_2} 
 \le\,c\norm{\rho}{L_{p_0}}\norm{\rho}{L_{p_1}}
 N^{\fracci{1}{q_0}+\fracci{1}{q_1}},              
 \label{2.18}
 \end{equation}
when $A_0$ and $A_1$ are Besov spaces. By \eqref{2.6} there exists  
for each $N$ an $l$ such that
 \begin{equation}
 \tfrac{1}{2}\norm{\rho}{L_2}^2\norm{\cal F^{-1}\Phi_0}{L_{p_2}}N
 \le\norm{\cal F^{-1}\Phi_0\cal F(\rho^{(t_0)}_{N,l}
 \rho^{(t_1)}_{N,l})}{L_{p_2}},
 \label{2.19}
 \end{equation}
and therefore \eqref{2.18} and \eqref{2.19} gives a contradiction for
a big $N$ unless $\fracc{1}{q_0}+\fracc{1}{q_1}\ge1$. 

When $A_0$ is a Besov space and $A_1=F^{\Mi,s_1}_{p_1,q_1}$ the
embedding $B^{\Mi,t}_{r,r}\hookrightarrow F^{\Mi,s_1}_{p_1,q_1}$ 
holds for every $r<p$ when $t-\fracc{|\M|}r=s_1-\fracc{|\M|}{p_1}$.
Then \eqref{2.13} holds with $A_1$ replaced by $B^{\Mi,t}_{r,r}$. Since
$s_0+t=\fracc{|\M|}{p_0}+\fracc{|\M|}{r}-|\M|$ the statement on
the $BB\bigdot$ cases gives that $\fracc1{q_0}+\fracc1r\ge1$. Then
$\fracc1{q_0}+\fracc1{p_1}=\inf\{\,\fracc1{q_0}+\fracc1r\mid r<p_1\,\}\ge1$.
This proves $(1')$.

\smallskip

The proof of (2), $(2')$, (3) and $(3')$ is due to \name{Franke}, 
who treated some of the eight cases with $\M=(1,\dots,1)$ in \cite{F3}.
(2) is found from \eqref{2.13} with $f=\theta_k$, $g=\theta_{-k}$ and
the fact that $\theta _k\theta _{-k}=\theta^2$, and \eqref{2.5}
together with \eqref{2.6}  gives $(2')$.

To show (3) one can take $f$ and $g$ equal to $\theta_k$ and $\theta$
respectively $\theta$ and $\theta_k$, and $(3')$ is obtained with $f$ 
and $g$ equal to $\theta^{(s_0)}_{N,+}$ and $\theta$ respectively 
$\theta$ and $\theta^{(s_1)}_{N,+}$.

(4) is due to \name{Sickel}, \cite{S}. The proof consists of an insertion of 
$f=g=\rho_k$ for $-k\in\N$ into \eqref{2.13} and an application 
of \eqref{2.2}. 

\smallskip

Concerning (5) for $j=0$ one has $\omega\theta(2^{-l\Mi}\cdot)\to
\omega$ in $L_p$ for $l\to\infty$ (by a majorisation), so for $k$ large enough
 \begin{equation}
 0<\tfrac{1}{2}\norm{\omega}{L_{p_2}}2^{k(s_2-\fracci{|\Mi|}{p_2})}
 \le c\norm{\omega}{L_{p_0}}\norm{\theta}{L_{p_1}}
 2^{k(s_0-\fracci{|\Mi|}{p_0})}
 \label{2.23}
 \end{equation}
by \eqref{2.2}; for $j=1$ the roles of $\omega_k$ and $\theta$ can be 
interchanged. 

$(5')$ is obtained analogously from 
$\omega^{(s_0-\fracci{|\Mi|}{p_0})}_N$ and $\theta$ respectively 
$\theta$ and $\omega^{(s_1-\fracci{|\Mi|}{p_1})}_N$.

Condition (6) can  be shown by insertion of $f=g=\omega_k$
into \eqref{2.13} followed by use of \eqref{2.2}. For $(6')$ formula
\eqref{2.5''} leads to the inequalities, where $t_j=s_j-\fracc{|\M|}{p_j}$, 
 \begin{equation}
 2^{s_2}\norm{\omega^2}{L_{p_2}}N^{\fracci1{q_2}}\le
 \norm{\Omega^{(t_0)}_N
   \Omega^{(t_1)}_N}{B^{\Mi,s_2}_{p_2,q_2}}\le
   c\norm{\omega}{L_{p_0}}\norm{\omega}{L_{p_1}}
   N^{\fracci1{q_0}+\fracci1{q_1}}
 \label{2.23'}
 \end{equation}
when only terms with $\Phi_{2^k+1}$ are kept in the 
$B^{\Mi,s_2}_{p_2,q_2}$ norm.

Concerning (7) in the $B\bigdot\bigdot$ cases it is found from
\eqref{2.5}, \eqref{2.6'} with a large $N$ and the assumption
$s_0=\fracc{|\Mi|}{p_0}$ that
 \begin{equation}
 \tfrac12 2^{3N(s_2-\fracci{|\Mi|}{p_2})}N \norm{\omega}{L_{p_2}}\le
 c\norm{\omega}{L_{p_0}}\norm{\omega}{L_{p_1}} N^{\fracci1{q_0}}
 2^{3N(s_1-\fracci{|\Mi|}{p_1})}.
 \label{2.24}
 \end{equation}
The second assumption, $s_2-\fracc{|\M|}{p_2}=s_1-\fracc{|\M|}{p_1}$,
then leads to the conclusion $q_0\le1$. The $F\bigdot\bigdot$ cases
can be reduced to the $B\bigdot\bigdot$ cases. Indeed, if $p_0>1$ is
assumed, there is a Sobolev embedding $B^{\Mi,t}_{r,o}\hookrightarrow
F^{\Mi,s_0}_{p_0,q_0}$ with $t-\fracc{|\M|}{r}=s_0-\fracc{|\M|}{p_0}$,
$p_0>r$ and $1<o<p_0$ according to \eqref{1.21}. But then \eqref{2.13}
holds with $A_0=F^{\Mi,s_0}_{p_0,q_0}$ replaced by $B^{\Mi,t}_{r,o}$,
hence $o\le1$ is necessary and the assumption $p_0>1$ is absurd.
\end{proof}

When (5) and (6) in Theorem~\ref{ncss-thm} are taken
together, they may be written
 \begin{align}
 s_2-\fracc{|\M|}{p_2}&\le\min(s_0-\fracc{|\M|}{p_0},s_1-
  \fracc{|\M|}{p_1},s_0-\fracc{|\M|}{p_0}+s_1-\fracc{|\M|}{p_1})
 \nonumber \\
 &=:\textstyle{\min^+}(s_0-\fracc{|\M|}{p_0},s_1-\fracc{|\M|}{p_1})
  =:\textstyle{\min^+_{j=0,1}}(s_j-\fracc{|\M|}{p_j})
 \label{2.30} \\
 &=:\textstyle{\min^+}(s_j-\fracc{|\M|}{p_j}).
 \nonumber
 \end{align}
For later reference it is observed that for $s_2=s_1$ formula
\eqref{2.30} is equivalent to 
 \begin{equation}
 \fracc{|\M|}{p_2}\ge \max(\fracc{|\M|}{p_0}+s_1-s_0,
 \fracc{|\M|}{p_1}, \fracc{|\M|}{p_1}+\fracc{|\M|}{p_0}-s_0).
 \label{2.31}
 \end{equation}

\begin{rem} \label{cnd1-rem}
When the embeddings $A_j\hookrightarrow L_{t_j}$ hold for $j=0$ and $1$ with
$-\frac{|\M|}{t_j}=s_j-\frac{|\M|}{p_j}$, cf.~\eqref{1.24} and \eqref{1.25}, 
(1) in \eqref{2.14} amounts to
$\frac{1}{t_2}:=\frac{1}{t_0}+\frac{1}{t_1}\le1$. Then
Proposition \ref{lp-prop} shows that $\pi$ equals $\mu$ 
on $A_0\oplus A_1$ and that $\pi(A_0\oplus A_1)\subset L_{t_2}$ 
where $t_2\ge1$. We may therefore interprete (1) in \eqref{2.14} as a
condition assuring that  $\pi(A_0\oplus A_1)$ is a distribution space.
\end{rem}

\begin{rem} \label{algbr-rem}
Applied to the situation where $A_0=A_1=A_2$ the
condition (6) in \eqref{2.14} amounts to $s\ge\fracc{|\M|}{p}$, and
for the borderline case $s=\fracc{|\M|}{p}$ condition (7) gives
$q\le1$ and $p\le1$ in the $BBB$ respectively $FFF$ cases. 

For $\M=(1,\dots,1)$ these conditions are known to be necessary (and
sufficient too for $s>0$) for $B^{\Mi,s}_{p,q}$ respectively 
$F^{\Mi,s}_{p,q}$ to be algebras. The proof of the necessity given here, 
for general $\M$, seems simpler than those in \cite{T0} and \cite{F3}. 
\end{rem}

\begin{rem} \label{ncss-rem} 
The conditions (1), $(1')$, (5), $(5')$, (6)
and $(6')$ above have seemingly not been published before,
but from a personal conversation the author knows that \name{W.~Sickel} has
obtained some of these independently.

Condition (7) generalises \cite[Rem.~III.13]{S}, where
$s_0=s_1=\fracc{n}{p_0}=\fracc{n}{p_1}$ is assumed. (However,
the intersection of one factor with $L_\infty$ is included there.)
\end{rem}

\begin{rem} \label{embd-rem} 
As an exercise one may use Lemma \ref{ncss-lem} to analyse 
the optimality of the linear embeddings in Section \ref{embd-ssect}.
See also \cite{ST} for sharp results in the isotropic case.

In particular, if $B^{\Mi,s}_{p,q}\hookrightarrow L_\infty$,
the $\rho_k$-part of \eqref{2.2} yields $s\ge\fracc{|\M|}p$. For
$s=\fracc{|\M|}p$ the inequality $\norm{\cdot}{L_\infty}\le c
\norm{\cdot}{B^{\Mi,s}_{p,q}}$ gives with $\omega^{(0)}_N$ inserted
that
 \begin{equation}
 N=N\omega(0)\le c\norm{\omega}{L_p}N^{\fracci1q},
 \label{2.40}
 \end{equation}
so $q\le1$ follows. This shows the optimality of \eqref{1.22}.
 
Similarly it is found from the properties of $\rho_k$ that
$F^{\Mi,s}_{p,q}\hookrightarrow L_\infty$ imply that
$s\ge\fracc{|\M|}p$. And as in the proof of (7) in
Theorem~\ref{ncss-thm} above it is found for $s=\fracc{|\M|}p$ that
$p\le1$ is necessary. Hence \eqref{1.23} is optimal for
$\M=(1,\dots,1)$.
Otherwise it is open whether $F^{\Mi,|\Mi|}_{1,q}$ with $1<q\le\infty$
is embedded into $L_\infty$ or not.

These counterexamples concerning $L_\infty$ are not
only valid for general $\M$, but they also seem simpler than the
arguments for the isotropic cases in \cite{F3,T0}.
\end{rem}

\section{Estimates of para-multiplication operators}  \label{estm-sect}

First the basic consequences of Yamazaki's theorems are collected.
The approach is essentially known since it is
adopted from \cite{Y1} and \cite{S}. However, these
references are inadequate for our purposes, so we state and prove 
Theorem~\ref{basic-thm}. 

It should be noted that \name{Sickel} for the isotropic versions of
\eqref{3.8}--\eqref{3.11} below has introduced a shorter formulation
by means of the local Hardy spaces $h_p=F^{0}_{p,2}$ ($0<p<\infty$),
cf.~\cite{Sic91,ST}, but the proof becomes less elementary, then.

\begin{thm} \label{basic-thm} 
Let $s$, $s_0$ and $s_1\in\R$ be given
together with $p_0$, $p_1$, $q$, $q_0$ and $q_1$ in $\,]0,\infty]$, and let 
$s_2=s_0+s_1$, $\fracc{1}{p_2}=\fracc{1}{p_0}+\fracc{1}{p_1}$ and 
$\fracc{1}{q_2}=\fracc{1}{q_0}+\fracc{1}{q_1}$. Then the operators
 \begin{alignat}{4}
 \pi_1&\colon L_{p_0}(\Rn)&&\oplus B^{\Mi,s}_{p_1,q}(\Rn)&&\to
 B^{\Mi,s}_{p_2,q}(\Rn), &\quad\text{ for}\quad& 1\le p_0\le\infty,
 \label{3.8} \\
 \pi_1&\colon L_{p_0}(\Rn)&&\oplus F^{\Mi,s}_{p_1,q}(\Rn)&&\to
 F^{\Mi,s}_{p_2,q}(\Rn), &\quad\text{ for}\quad& 1<p_0\le\infty,\quad
 p_1<\infty, 
 \label{3.9} \\
 \pi_1&\colon F^{\Mi,0}_{p_0,1}(\Rn)&&\oplus B^{\Mi,s}_{p_1,q}(\Rn)&&\to
 B^{\Mi,s}_{p_2,q}(\Rn), &\quad\text{ for}\quad& 0< p_0<\infty,
 \label{3.10} \\
 \pi_1&\colon F^{\Mi,0}_{p_0,1}(\Rn)&&\oplus F^{\Mi,s}_{p_1,q}(\Rn)&&\to
 F^{\Mi,s}_{p_2,q}(\Rn), &\quad\text{ for}\quad& 0< p_0<\infty,
 \quad p_1<\infty
 \label{3.11} 
 \end{alignat}
are bounded, and $\pi_3$ has similar properties when the summands are
interchanged (since $\pi_3(u,v)=\pi_1(v,u)$).

For $s_0+s_1>|\M|\max(0,\fracc{1}{p_0}+\fracc{1}{p_1}-1)$, i.e., 
$s_2>|\M|(\fracc1{p_2}-1)_+$, the operator 
 \begin{equation}
 \pi_2\colon B^{\Mi,s_0}_{p_0,q_0}(\Rn)\oplus B^{\Mi,s_1}_{p_1,q_1}(\Rn)\to
 B^{\Mi,s_2}_{p_2,q_2}(\Rn),
 \label{3.16}
 \end{equation}
is bounded, and if $s_2>|\M|(\fracc{1}{\min(p_2,q_2)}-1)_+$ and 
both $p_0$ and $p_1<\infty$, so is
 \begin{equation}
 \pi_2\colon F^{\Mi,s_0}_{p_0,q_0}(\Rn)\oplus F^{\Mi,s_1}_{p_1,q_1}(\Rn)\to
 F^{\Mi,s_2}_{p_2,q_2}(\Rn).
 \label{3.17}
 \end{equation}

Furthermore, if $s_0<0$ the operator $\pi_1$ is continuous
 \begin{alignat}{3} 
 \pi_1&\colon B^{\Mi,s_0}_{p_0,q_0}(\Rn)&&\oplus B^{\Mi,s_1}_{p_1,q_1}(\Rn)
 &&\to B^{\Mi,s_2}_{p_2,q_2}(\Rn),
 \label{3.18} \\
 \pi_1&\colon F^{\Mi,s_0}_{p_0,q_0}(\Rn)&&\oplus F^{\Mi,s_1}_{p_1,q_1}(\Rn)
 &&\to F^{\Mi,s_2}_{p_2,q_2}(\Rn),\quad\text{ for$\quad p_0$ and $p_1<\infty$},
 \label{3.19}
 \end{alignat}
and for $s_1<0$ the operator $\pi_3$ has similar properties.
\end{thm}

\begin{proof} 
It may be assumed that $\psi_j=\Psi_j$ and $\varphi_j=\Phi_j$ for if
not the estimates below are valid with 
$\Norm{ \{2^{sj}\norm{u_j}{L_p}\} }{\ell_q}$ instead of 
$\norm{u}{B^{\Mi,s}_{p,q}}$ etc., and then Remark~\ref{PSI-rem} applies.
We begin with \eqref{3.18}.

For $u\in B^{\Mi,s_0}_{p_0,q_0}$ and $v\in B^{\Mi,s_1}_{p_1,q_1}$ one has,
when $r_0=\min(1,p_0)$ and $a_j=\norm{u_j}{L_{p_0}}$,
 \begin{align}
 \Norm{\bigl\{ \norm{2^{s_2j}u^{j-2} v_j}{L_{p_2}}\bigr\}
 ^\infty_{j=0}}{\ell_{q_2}}&\le
 \Norm{ \bigl\{ \norm{2^{s_0j}u^{j-2}}{L_{p_0}}
 \norm{2^{s_1j}v_j}{L_{p_1}} \bigr\}^\infty_{j=0} }{\ell_{q_2}}
 \notag \\
 & \le\, \norm{ \{ 2^{s_0j}( a_0^{r_0}+\dots
 +a_j^{r_0})
 ^{\frac{1}{r_0}} \}  }{\ell_{q_0}}
 \norm{v}{B^{\Mi,s_1}_{p_1,q_1}}
 \notag \\
 &\le\, c\norm{u}{B^{\Mi,s_0}_{p_0,q_0}}\norm{v}{B^{\Mi,s_1}_{p_1,q_1}},
 \label{3.20} 
 \end{align}
by use of H\"older's inequality, \eqref{1.27} and Lemma \ref{Y-lem}.
Since \eqref{3.5'} and \eqref{3.20} show that the conditions in
Theorem \ref{Y1-thm} are satisfied, \eqref{3.18} follows. Similarly one can 
prove \eqref{3.19} and the analogous properties of
$\pi_3$ when $s_1<0$.

For the treatment of $\pi_2$ one can use the estimate
 \begin{align}
 \Norm{\norm{ \{2^{s_2j}u_{j-1} v_j\} }{\ell_{q_2}}(\cdot) }{L_{p_2}}&\le
 \Norm{ 2^{s_0}\norm{ \{2^{s_0j}u_j\} }{\ell_{q_0}}
 \norm{ \{2^{s_1j}v_j\} }{\ell_{q_1}} }{L_{p_2}}
 \nonumber \\
 &\le\,2^{s_0}\norm{u}{F^{\Mi,s_0}_{p_0,q_0}}\norm{v}{F^{\Mi,s_1}_{p_1,q_1}}
 \label{3.21}
 \end{align}
along with similar estimates of $u_j v_j$ and $u_j v_{j-1}$ to
conclude that
 \begin{equation} 
 \Norm{\norm{ \{2^{s_2j}(u_{j-1} v_j+u_j v_j+
 u_j v_{j-1})\} }{\ell_{q_2}} }{L_{p_2}} \le  
 c\norm{u}{F^{\Mi,s_0}_{p_0,q_0}} \norm{v}{F^{\Mi,s_1}_{p_1,q_1}},
 \label{3.22}
 \end{equation}
where $c$ is proportional to $2^{s_0}+1+2^{s_1}$. In view of \eqref{3.6}
and Theorem \ref{Y2-thm} this proves \eqref{3.17}. 
\eqref{3.16} is proved similarly.

The formulae \eqref{3.8} and \eqref{3.10} are deduced from the estimate 
 \begin{equation}
 \Norm{ \bigl\{2^{sj}\norm{u^{j-2}v_j}{L_{p_2}}\bigr\}^\infty_{j=0} }{\ell_q}
 \le\sup_k\norm{u^k}{L_{p_0}}
 \Norm{ \bigl\{2^{sj}\norm{v_j}{L_{p_1}}\bigr\} }{\ell_q},
 \label{3.23}
 \end{equation}
while \eqref{3.9} and \eqref{3.11} are based on a version
of \eqref{3.23} with $\Norm{\sup_k |u^k|}{L_{p_0}}$.
Indeed, in \eqref{3.23} we can introduce the estimates, where 
$a=\norm{\check\Psi}{L_1}$,
 \begin{align} 
 \sup_k \norm{u^k}{L_{p_0}}&\le\sup_k\norm{\check\Psi_k}{L_1}
 \norm{u}{L_{p_0}} \le\, a \norm{u}{L_{p_0}},
 \label{3.25} \\
 \sup_k \norm{u^k}{L_{p_0}}&\le
 \sup_k\Norm{|u_0|+\dots+|u_k|}{L_{p_0}}=
 \norm{u}{F^{\Mi,0}_{p_0,1}},
 \label{3.26} 
 \end{align}
which by application of Theorem \ref{Y1-thm} shows 
\eqref{3.8} and \eqref{3.10}, respectively.

In the $F$ case one can estimate, for $0<p_0<\infty$ 
respectively $1<p_0\le\infty$,
 \begin{align} 
 \Norm{\sup_k|u^k|}{L_{p_0}} &\le
 \Norm{\sup_k(|u_0|+\dots+|u_k|)}{L_{p_0}}\le\norm{u}{F^{\Mi,0}_{p_0,1}},
 \label{3.27} \\
 \Norm{\sup_k|u^k|}{L_{p_0}} &\le  
 c\norm{u}{L_{p_0}}, 
 \label{3.28} 
 \end{align}
where $c=a$ for $p_0=\infty$,
while for $1<p_0<\infty$ Corollary 2.9 in \cite{Y1} applies.
\end{proof}

Next we include in Corollary~\ref{basic-cor} various properties
that are directly applicable to the sufficient conditions with $p_0\ne
p_1\ne p_2$ in Section~\ref{suff-sect} below.

In the sequel,  $s_0\ge s_1$ is assumed for simplicity
(as we may by commutativity of $\pi$). 
For a sum-exponent $t$, the requirement $q_0\kd{s_0=s_1}\le t\le\infty$
reduces to $0<t\le \infty$ for $s_0\ne s_1$, since only $t>0$ is allowed. 
A similar remark applies to integral-exponents in the $F$~case.
Recall the $s_2$, $p_2$ and $q_2$ notation of Theorem.~\ref{basic-thm}.

\begin{cor} \label{basic-cor}
Let $s_j\in\R$, $p_j\in\,]0,\infty]$ and $q_j\in\,]0,\infty]$ be given for
$j=0$  and $1$ such that $s_0+s_1>|\M|(\fracc1{p_0}+\fracc1{p_1}-1)_+$ and
$s_0\ge s_1$.

Then, if $p_0$ and $p_1<\infty$, the bilinear operators
 \begin{align}
  F^{\Mi,s_0}_{p_0,q_0}\oplus F^{\Mi,s_1}_{p_1,q_1}& \xrightarrow{\;\pi_1\;}
  \bigcap\bigl\{\,F^{\Mi,s_1}_{r,q_1}\bigm| \fracc{|\M|}{p_1}+
   (s_0-\fracc{|\M|}{p_0})_-<\fracc{|\M|}{r}\le\fracc{|\M|}{p}\,\bigr\},
   \label{4.8}
 \\
   F^{\Mi,s_0}_{p_0,q_0}\oplus F^{\Mi,s_1}_{p_1,q_1} &\xrightarrow{\;\pi_2\;}
   \bigcap\bigl\{\,F^{\Mi,s_1}_{r,t}\bigm| 0<t\le\infty,\ (\fracc{|\M|}{p_1}
   +\fracc{|\M|}{p_0}-s_0)_+\le\fracc{|\M|}{r}\le\fracc{|\M|}{p}\,\bigr\},
   \label{4.9}
 \\
   F^{\Mi,s_0}_{p_0,q_0}\oplus F^{\Mi,s_1}_{p_1,q_1} &\xrightarrow{\;\pi_3\;}
   \bigcap\bigl\{\,F^{\Mi,s_1}_{r,t}\bigm| q_0\kd{s_0=s_1}\le t\le\infty,
 \nonumber
 \\
   &\hphantom{\xrightarrow{\;\pi_3\;} \bigcap\bigl\{\,F_{r}}
   (\fracc{|\M|}{p_1}+\fracc{|\M|}{p_0}-s_0+(s_1-\fracc{|\M|}{p_1})_+)_+
   <\fracc{|\M|}{r}\le\fracc{|\M|}{p}\,\bigr\},
   \label{4.10}   
 \end{align}
are bounded. In addition the value $\fracc{|\M|}r= \fracc{|\M|}{p_1}+
(s_0-\fracc{|\M|}{p_0})_-$ may be included in \eqref{4.8} under the
condition that
$F^{\Mi,s_0}_{p_0,q_0}\hookrightarrow L_\infty$ holds if
$s_0=\fracc{|\M|}{p_0}$. Similarly $\fracc{|\M|}{r}=
(\fracc{|\M|}{p_1}+\fracc{|\M|}{p_0}-s_0+(s_1-\fracc{|\M|}{p_1})_+)_+$
may be included in \eqref{4.10} provided  $F^{\Mi,s_1}_{p_1,q_1} 
\hookrightarrow L_\infty$ holds if $s_1=\fracc{|\M|}{p_1}$.

Furthermore, there is boundedness of the operators
 \begin{align}
   B^{\Mi,s_0}_{p_0,q_0}\oplus B^{\Mi,s_1}_{p_1,q_1} &\xrightarrow{\;\pi_1\;}
   \bigcap\bigl\{\,B^{\Mi,s_1}_{r,q_1}\bigm| \fracc{|\M|}{p_1}+
   (s_0-\fracc{|\M|}{p_0})_-<\fracc{|\M|}{r}\le\fracc{|\M|}{p}\,\bigr\},
   \label{4.8'}
 \\
   B^{\Mi,s_0}_{p_0,q_0}\oplus B^{\Mi,s_1}_{p_1,q_1} &\xrightarrow{\;\pi_2\;}
   \bigcap\bigl\{\,B^{\Mi,s_1}_{r,t}\bigm| 0<t\le\infty,\ (\fracc{|\M|}{p_1}
   +\fracc{|\M|}{p_0}-s_0)_+<\fracc{|\M|}{r}\le\fracc{|\M|}{p}\,\bigr\},
   \label{4.9'}
 \\
   B^{\Mi,s_0}_{p_0,q_0}\oplus B^{\Mi,s_1}_{p_1,q_1} &\xrightarrow{\;\pi_3\;}
   \bigcap\bigl\{\,B^{\Mi,s_1}_{r,t}\bigm| q_0\kd{s_0=s_1}\le t\le\infty,
 \nonumber
 \\
   &\hphantom{\xrightarrow{\;\pi_3\;}\bigcap\bigl\{\,B_{r}}
   (\fracc{|\M|}{p_1}+\fracc{|\M|}{p_0}-s_0+(s_1-\fracc{|\M|}{p_1})_+)_+
   <\fracc{|\M|}{r}\le\fracc{|\M|}{p}\,\bigr\}.
   \label{4.10'}
 \end{align}
When $\frac{|\M|}{t_j}=\frac{|\M|}{p_j}-s_j$ one can include 
$B^{\Mi,s_1}_{r,q_1}$ with
$\fracc{|\M|}r=\fracc{|\M|}{p_1}+(s_0-\fracc{|\M|}{p_0})_-$ 
in the intersection in \eqref{4.8'} if
 \begin{alignat}{2}
 s_0&>\fracc{|\M|}{p_0}, &&
 \label{4.8a}\\
 s_0&=\fracc{|\M|}{p_0} &\quad\text{and}\quad q_0&\le1, 
 \label{4.8b}\\
 s_0&<\fracc{|\M|}{p_0} &\quad\text{and}\quad q_0&< t_0,
 \label{4.8c} \\
 s_0&<\fracc{|\M|}{p_0} &\quad\text{and}\quad q_0&\le t_0
  \quad\text{and either $\M=(1,\dots,1)$ or $t_0\le2$}.
 \label{4.8d} 
 \end{alignat}
In \eqref{4.9'} the space $B^{\Mi,s_1}_{r,q_2}$ may be included   
when $\fracc{|\M|}{r}=(\fracc{|\M|}{p_1}+\fracc{|\M|}{p_0}-s_0)_+$.

$B^{\Mi,s_1}_{r,q_0}$ with $\fracc{|\M|}{r}=(\fracc{|\M|}{p_1}+
\fracc{|\M|}{p_0}-s_0+(s_1-\fracc{|\M|}{p_1})_+)_+$ may be included in
\eqref{4.10'} if
 \begin{alignat}{2}
 s_1&>\fracc{|\M|}{p_1}, &&
 \label{4.10a}\\
 s_1&=\fracc{|\M|}{p_1} &\quad\text{and}\quad q_1&\le1 , 
 \label{4.10b}\\
 0<s_1&<\fracc{|\M|}{p_1} &\quad\text{and}\quad q_1&< t_1 ,
 \label{4.10c} \\
 0<s_1&<\fracc{|\M|}{p_1} &\quad\text{and}\quad  q_1&\le t_1
  \quad\text{and either $\M=(1,\dots,1)$ or $t_1\le2$},
 \label{4.10d} \\
\intertext{and $B^{\Mi,s_1}_{r,q_2}$ with $\fracc{|\M|}{r}=(\fracc{|\M|}{p_1}+
\fracc{|\M|}{p_0}-s_0)_+$ may be included if}
 s_1&\le0. &&  
 \label{4.10e}
 \end{alignat}
\end{cor}

\begin{proof} Suppose 
$s_2>|\M|(\fracc1{\min(p_2,q_2)}-1)_+$. From \eqref{3.17} and 
$F^{\Mi,s_2}_{p_2,q_2}\hookrightarrow F^{\Mi,s_1}_{p_2,t}$,
it is inferred that
$\pi_2(F^{\Mi,s_0}_{p_0,q_0}\oplus F^{\Mi,s_1}_{p_1,q_1})\subset
F^{\Mi,s_1}_{r,t}$ holds for
$\fracc{|\M|}r=\fracc{|\M|}{p_0}+\fracc{|\M|}{p_1}$ and any $t$. Using
Sobolev embeddings $F^{\Mi,s_2}_{p_2,q_2}\hookrightarrow
F^{\Mi,s_1}_{r,t}$ it is not only required that
$\fracc{|\M|}r\ge\fracc{|\M|}{p_1}+\fracc{|\M|}{p_0}-s_0$, 
also $\fracc{|\M|}r>0$ must hold. Then \eqref{1.26} gives the intermediate
values.\,---\,Observe that when $s_2\not>|\M|(\fracc1{q_2}-1)$ one can 
consider $\pi_2$ on the larger space $F^{\Mi,s_0}_{p_0,\infty}\oplus
F^{\Mi,s_1}_{p_1,\infty}$; the embedding procedure above
gives the same result, eventually.

To prove \eqref{4.8}, one can  combine $F^{\Mi,s_0}_{p_0,q_0}
\hookrightarrow L_\infty\cap F^{\Mi,0}_{p_0,1}$ (when this holds)
with \eqref{3.9} and  \eqref{3.11}, and it is seen that even 
$\fracc{|\M|}r=\fracc{|\M|}{p_1}$ is possible. For $s_0\le
\fracc{|\M|}{p_0}$ use of \eqref{3.11} together with \eqref{1.20'} gives 
the lower bound $\fracc{|\M|}r=\fracc{|\M|}{p_1}+\fracc{|\M|}{p_0}-s_0$, except
when $s_0=\fracc{|\M|}{p_0}$ where only `$<$' is obtained.

When $F^{\Mi,s_1}_{p_1,q_1}\hookrightarrow L_\infty$ application of
the $\pi_3$ version of
\eqref{3.9} and \eqref{3.11} gives $\pi_3(F^{\Mi,s_0}_{p_0,q_0}\oplus 
F^{\Mi,s_1}_{p_1,q_1})\subset F^{\Mi,s_0}_{r,q_0}$ for $\fracc{|\M|}{p_0}
\le\fracc{|\M|}r\le\fracc{|\M|}p$. From $s_0\ge s_1$ it follows that
\eqref{4.10} holds even with `$<$' replaced by `$\le$'. For
$0<s_1<\fracc{|\M|}{p_1}$ there is an inclusion
$\pi_3(F^{\Mi,s_0}_{p_0,q_0}\oplus F^{\Mi,s_1}_{p_1,q_1})\subset 
F^{\Mi,s_0}_{r,q_0}$ for $\fracc{|\M|}r=\fracc{|\M|}{p_0}+
\fracc{|\M|}{p_1}-s_1$. Hence \eqref{4.10} holds with `$\le$' for 
$0<s_1<\fracc{|\M|}{p_1}$ and with `$<$' for $s_1=\fracc{|\M|}{p_1}$. 
When $s_1<0$ the last statement in Theorem~\ref{basic-thm} gives that
 \begin{align}
 F^{\Mi,s_0}_{p_0,q_0}\oplus F^{\Mi,s_1}_{p_1,q_1}&\xrightarrow{\;\pi_3\;} 
                               F^{\Mi,s_0+s_1}_{p,\infty}
 \notag \\
 &\hookrightarrow \bigcap\bigl\{\,F^{\Mi,s_1}_{r,t}\bigm| (\fracc{|\M|}{p_0}
   +\fracc{|\M|}{p_1}-s_0)_+\le\fracc{|\M|}{r}\le\fracc{|\M|}{p},
     \  0<t\le\infty\,\bigr\}.
 \label{4.13}
 \end{align}
The inclusion
$F^{\Mi,0}_{p_1,q_1}\hookrightarrow \cap\{\,F^{\Mi,-\fracc{s_0}2}_{r,q_1}
\mid\fracc{|\M|}{p_1}-\tfrac{s_0}{2}\le\fracc{|\M|}{r}\le
\fracc{|\M|}{p_1}\,\}$ combined with the techniques for $s_1<0$ can
be used to show that $\pi_3(F^{\Mi,s_0}_{p_0,q_0}\oplus
F^{\Mi,0}_{p_1,q_1})$ is contained in the space on the right hand side
of \eqref{4.13} for $s_1=0$.

\smallskip

Formulae \eqref{4.8'}--\eqref{4.10'} are proved in the same manner: Since
\eqref{3.16} holds under fewer conditions 
than \eqref{3.17}, \eqref{4.9} carry over to
the Besov case with the modification that $t\ge q_2$ is necessary 
when $\fracc{|\M|}{r}=\fracc{|\M|}{p_1}+
\fracc{|\M|}{p_0}-s_0\ge0$, cf.~\eqref{1.20}. 

\eqref{4.8'} and \eqref{4.8a}--\eqref{4.8d}
are proved in the same way as \eqref{4.8} 
using \eqref{3.8}, \eqref{3.10} and \eqref{1.21}. The properties
\eqref{4.10'} and \eqref{4.10a}--\eqref{4.10e} follow from
Theorem~\ref{basic-thm} and \eqref{1.21} analogously to \eqref{4.10}.
\end{proof}

\section{Sufficient conditions for multiplication} 
\label{suff-sect}
Before we establish the sufficient conditions, an overview of the
results in Sections~\ref{ncss-sect} and \ref{suff-sect} pertinent 
to the questions (II) and (III) is given. 
Recall that $A_j$, for each $j=0$, 1 and 2, denotes either
$B^{\Mi,s_j}_{p_j,q_j}$ or $F^{\Mi,s_j}_{p_j,q_j}$.
(In this section all spaces are over $\Rn$, so for simplicity $\Rn$ is
omitted here.)

\bigskip

Concerning question (II) in Section~\ref{intr-sect}, it is 
{\em necessary \/} for $\pi\colon A_0\oplus A_1\to A_2$ to be bounded 
that $(s_j,p_j,q_j)_{j=0,1}$ satisfy (1) and (2) in \eqref{2.14},
i.e., it is necessary that
 \begin{equation}
 s_0+s_1\,\ge \, |\M|(\fracc1{p_0}+\fracc1{p_1}-1)_+.
 \label{4.1}
 \end{equation}
Here one can denote by $\Dm(\pi,BB)$ the domain of parameters
$(s_j,p_j,q_j)_{j=0,1}$ such that there is
continuity of $\pi\colon B^{\Mi,s_0}_{p_0,q_0}\oplus
B^{\Mi,s_1}_{p_1,q_1}\to A_2$ for some Besov or Triebel--Lizorkin
space $A_2$; in a similar way one can define 
domains $\Dm(\pi,BF)$ etc. (The notation should remind one that these
domains consists of {\em numbers\/} and not of vectors.)

In this section we shall show that in any of the
$\bigdot\bigdot\bigdot$ cases the
{\em sharp\/} inequality in \eqref{4.1} is {\em sufficient\/} for the
continuity of $\pi\colon A_0\oplus A_1\to A_2$, and we suggest to write
 \begin{alignat}{2}
 (s_j,p_j,q_j)_{j=0,1}&\in\Dm(\pi)&\quad\text{ if}\quad
 s_0+s_1\,&> \,|\M|(\fracc1{p_0}+\fracc1{p_1}-1)_+,
 \label{4.2}\\
 (s_j,p_j,q_j)_{j=0,1}&\in\uDm(\pi)&\quad\text{ if}\quad
 s_0+s_1\,&\ge \, |\M|(\fracc1{p_0}+\fracc1{p_1}-1)_+
 \label{4.3}
 \end{alignat}
(with a possible omission of ``$j=0,1$''). Here $\Dm(\pi)\subset
\Dm(\pi,\bigdot\bigdot)\subset\uDm(\pi)$
(cf.~the abovementioned sufficiency and \eqref{4.1}). In these terms $(1')$ 
and $(2')$ in \eqref{2.15} state that $\Dm(\pi,BB)\ne\uDm(\pi)$ respectively 
$\Dm(\pi,\bigdot\bigdot)\ne\uDm(\pi)$.

We say that $(s_j,p_j,q_j)_{j=0,1}\in\Dm(\pi)$ is a pair of {\em
generic\/} parameters.

Moreover, when $\max(s_0,s_1)>0$ a complete characterisation of
$\Dm(\pi,BB)$ and $\Dm(\pi,FF)$ is found, cf.~Corollary~\ref{dscr-cor} below. 

\bigskip

When $\pi$ is defined on $A_0\oplus A_1$ it is a
question for which $A_2$ there is continuity of $\pi\colon A_0\oplus A_1\to
A_2$, cf.~(III). To have a convenient notation for this 
we define the set 
 \begin{equation}
  \begin{aligned}
 \Pl(A_0,A_1)&=\bigl\{\,(t,r,o)\bigm|\pi\colon A_0\oplus A_1\to A_2
  \quad\text{is bounded}
 \\[-2\jot]
 &\qquad\qquad\qquad\text{when $A_2$ has the parameter $(t,r,o)$}\,\bigr\}.
 \end{aligned}
 \label{P3}
 \end{equation}
$\Pl(A_0,A_1)$ refers to two specified spaces on which
$\pi(\cdot,\cdot)$ makes sense, cf.~(II). To distinguish between the 
various $\bigdot\bigdot B$ and $\bigdot\bigdot F$ cases one could
write $\Pl(A_0,A_1;B)$  and $\Pl(A_0,A_1;F)$,
respectively, but usually it is unnecessary.

For each $A_0\oplus A_1$
with $(s_j,p_j,q_j)\in\Dm(\pi,\bigdot\bigdot)$ one may wish to
{\em determine\/} the set $\Pl(A_0,A_1)$, and this will be done below for
$(s_j,p_j,q_j)\in\Dm(\pi)$ in the isotropic $FFF$ cases, whereas in the
$BBB$ cases a certain `vertex' question remains open.

However, in the general case much information on $\Pl(A_0,A_1)$ is
contained in Theorem~\ref{ncss-thm} already. Indeed, if
$(s_2,p_2,q_2)\in\Pl(A_0,A_1)$, then (3)--(7) in
\eqref{2.14} hold\,---\,regardless of which of the
$\bigdot\bigdot\bigdot$ cases that
are under consideration. The set that contains
$(s_2,\fracc{|\M|}{p_2})$ for all $(s_2,p_2,q_2)$ satisfying (3)--(6)
is pictured in Figure~\ref{A2-fig}.

\begin{figure}[ht]
\hfil
\setlength{\unitlength}{0.0125in}
\begin{picture}(406,316)(0,0)
\path(163,180)(280,180)(280,20)
\path(200,180)(280,140)
\path(280,100)(224,128)      
\path(182,149)(148,166)
\path(280,60)(120,140)
\path(240,40)(95,115)
\path(200,20)(68,86)
\path(120,20)(40,60)
\path(60,10)(13,33.5)
\dottedline{5}(111,180)(160,180)
\dottedline{5}(177,217)(125,165)
\dashline{4.000}(157,177)(10,30)
\thicklines
\path(7,100)(370,100)
\path(362.000,98.000)(370.000,100.000)(362.000,102.000)
\path(10,0)(10,290)
\path(12.000,282.000)(10.000,290.000)(8.000,282.000)
\put(32,135){\makebox(0,0)[lb]{\raisebox{0pt}[0pt][0pt]
{\shortstack[l]{{$\scriptstyle{s-\fracci{|\Mi|}p=}$}}}}}
\put(32,122){\makebox(0,0)[lb]{\raisebox{0pt}[0pt][0pt]
{\shortstack[l]{{$\scriptstyle{\min^+(s_j-\fracci{|\Mi|}{p_j})}$}}}}}
\put(195,185){\makebox(0,0)[lb]{\raisebox{0pt}[0pt][0pt]
{\shortstack[l]{{$\scriptstyle{s=\min(s_0,s_1)}$}}}}}
\put(285,170){\makebox(0,0)[lb]{\raisebox{0pt}[0pt][0pt]
{\shortstack[l]{{$\scriptstyle\fracpi=\fracci1{p_0}+\fracci1{p_1}$}}}}}
\put(177,135){\makebox(0,0)[lb]{\raisebox{0pt}[0pt][0pt]
{\shortstack[l]{{$\Pl(A_0,A_1)$}}}}}
\put(7,295){\makebox(0,0)[lb]{\raisebox{0pt}[0pt][0pt]
{\shortstack[l]{{$s$}}}}}
\put(0,97){\makebox(0,0)[lb]{\raisebox{0pt}[0pt][0pt]
{\shortstack[l]{{$0$}}}}}
\put(158,178){\makebox(0,0)[lb]{\raisebox{0pt}[0pt][0pt]
{\shortstack[l]{{$\scriptstyle\circ$}}}}}
\put(105,178.5){\makebox(0,0)[lb]{\raisebox{0pt}[0pt][0pt]
{\shortstack[l]{{$\scriptscriptstyle\times$}}}}}
\put(178,219.7){\makebox(0,0)[lb]{\raisebox{0pt}[0pt][0pt]
{\shortstack[l]{{$\scriptscriptstyle\times$}}}}}
\put(95,192){\makebox(0,0)[lb]{\raisebox{0pt}[0pt][0pt]
{\shortstack[l]{{$A_1$}}}}}
\put(375,97){\makebox(0,0)[lb]{\raisebox{0pt}[0pt][0pt]
{\shortstack[l]{{$\fracc{|\M|}p$}}}}}
\put(172,232){\makebox(0,0)[lb]{\raisebox{0pt}[0pt][0pt]
{\shortstack[l]{{$A_0$}}}}}
\end{picture}
\hfil
\caption[]{The set $\Pl(A_0,A_1)$ contains 
                  the possible $A_2$ spaces.\\ (`The gap' disappears 
  when $\min_{j=0,1}(s_j-\fracc{|\M|}{p_j})\ge 0$.)} 
\label{A2-fig}
\end{figure}
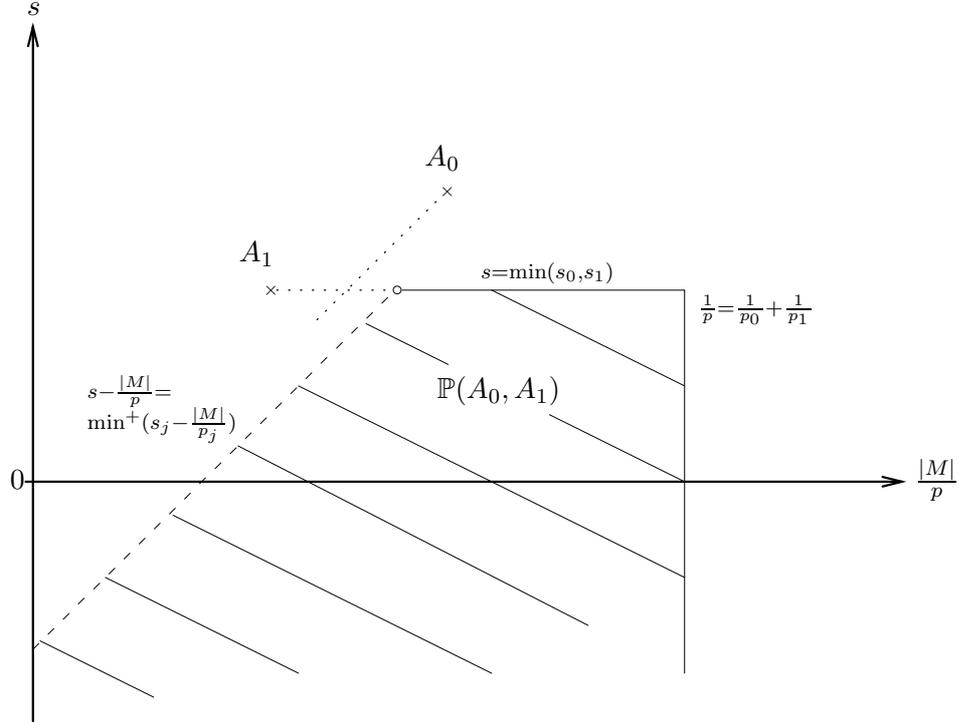

The dashed line in the figure, that corresponds to 
$s=\fracc{|\M|}{p}+\min_{j=0,1}^+(s_j-\fracc{|\M|}{p_j})$,
cf.~\eqref{2.30} for the notation, should remind one that it is not always
possible to obtain $(s_2,\fracc{|\M|}{p_2})$ here for a parameter
$(s_2,p_2,q_2)\in \Pl(A_0,A_1)$. In fact (7) in
\eqref{2.14} is a necessary condition on $A_0\oplus A_1$ for this.
The ``$\circ$'' at the left vertex has been used to indicate that
this point is subject to two set of necessary conditions
(the open question for the generic $BBB$ cases also concerns this point).

For all the $\bigdot\bigdot\bigdot$ cases it is found when
$(s_j,p_j,q_j)\in \Dm(\pi)$ that $\Pl(A_0,A_1)$ contains all
$(s_2,p_2,q_2)$ with $(s_2,\fracc{|\M|}{p_2})$ in the interior of the
set in Figure~\ref{A2-fig} and $q_2\in\,]0,\infty]$.

\bigskip

The possibility of having $(s_2,p_2,q_2)\in\Pl(A_0,A_1)$ with
$(s_2,\fracc{|\M|}{p_2})$ on the boundary is restricted somewhat by
the necessary conditions in (7), $(3')$, $(5')$ and $(6')$. 

In the following
it is verified that (7) and $(3')$ are both necessary and sufficient
for this in the generic isotropic $FFF$ cases. 

For the $BBB$ cases the four conditions (7), $(3')$, $(5')$ and $(6')$
are also necessary and sufficient in this respect, {\em except\/} at
``$\circ$'' in Figure~\ref{A2-fig}. At this particular point it is in
general further required that $q_j<|\M|(\fracc{|\M|}{p_j}-s_j)^{-1}$
for $j=0$ and or $1$ when $s_j<\fracc{|\M|}{p_j}$, 
cf.~Theorem~\ref{BBB-thm}. The necessity of these
requirements constitute the abovementioned open `vertex' question for
the generic $BBB$ cases.

\subsection{The $FFF$ cases} \label{FFF-ssect} 
In the rest of this section we assume that $s_0\ge s_1$ as we may
because $\pi$ is commutative on $\cal S'\times\cal S'$.
Despite Definition~\ref{pi-defn} above of $\pi(\cdot,\cdot)$ we shall not pay
attention to the independence of $\psi$ here, for this will be
obtained afterwards in Section~\ref{indp-ssect} below. 

When $s_0\ge s_1$ the notation
 \begin{equation}
 \begin{gathered}
 \fracp=\fracc{1}{p_0}+\fracc{1}{p_1},\qquad
 q=\max(q_0\kd{s_0=s_1},q_1)
 \\
 \tfrac{|\M|}{p^*_1}=\fracc{|\M|}{p_1}+(s_0-\fracc{|\M|}{p_0})_-
   +(s_1-\fracc{|\M|}{p_1}-(s_0-\fracc{|\M|}{p_0})_+)_+
 \end{gathered}
 \label{4.4'}
 \end{equation}
will be useful in this and the following subsection. 

The next result concerns $F^{\Mi,s_0}_{p_0,q_0}\oplus F^{\Mi,s_1}_{p_1,q_1}$
with $(s_j,p_j,q_j)_{j=0,1}\in\Dm(\pi)$:

\begin{thm} \label{FFF-thm} 
Let $\M$ and the numbers
$s_0,s_1\in\R$, $p_0,p_1\in\,]0,\infty[\,$ and
$q_0,q_1\in\,]0,\infty]$ be given and suppose that
$s_0+s_1>|\M|\max(0,\fracc1{p_0}+\fracc1{p_1}-1)$.

For $s_0\ge s_1$ and with $p$, $p^*_1$ and $q$ given by \eqref{4.4'},
the product $\pi(\cdot,\cdot)$ is continuous
 \begin{equation}
 F^{\Mi,s_0}_{p_0,q_0}\oplus F^{\Mi,s_1}_{p_1,q_1}
 \xrightarrow{\;\pi(\cdot,\cdot)\;}
 \bigcap\bigl\{\,F^{\Mi,s_1}_{r,q}\bigm| \tfrac{|\M|}{p^*_1}
         <\fracc{|\M|}{r}\le \fracc{|\M|}{p} \,\bigr\}.
 \label{4.4}
 \end{equation}
Furthermore one can include $\fracc{|\M|}{r}=\frac{|\M|}{p^*_1}$ 
when $F^{\Mi,s_0}_{p_0,q_0}\hookrightarrow L_\infty$ holds in addition
to one of the following conditions:
 \begin{equation}
 \alignedat2
 &\text{\rom{(1a)}}\quad s_1-\fracc{|\M|}{p_1}&> s_0-\fracc{|\M|}{p_0},&
  \\
 &\text{\rom{(1b)}}\quad s_1-\fracc{|\M|}{p_1}&\le s_0-\fracc{|\M|}{p_0},
    &\ \text{and $F^{\Mi,s_1}_{p_1,q_1}\hookrightarrow L_\infty$ holds}
 \\
 &&&\hphantom{\text{and $F^{\Mi,s_1}_{p_1,q_1}$}}
    \smash[b]{\text{if $s_1-\fracc{|\M|}{p_1}=s_0-\fracc{|\M|}{p_0}=0$}}, 
 \endalignedat
 \label{4.6}
 \end{equation}
and under each of the conditions
 \begin{equation}
 \alignedat2
 &\text{\rom{(2a)}}\quad s_1&>\fracc{|\M|}{p_1},&
 \\
 &\text{\rom{(2b)}}\quad s_1&<\fracc{|\M|}{p_1},&\ \text{ and 
    $F^{\Mi,s_0}_{p_0,q_0}\hookrightarrow L_\infty$ holds if $s_0=
        \fracc{|\M|}{p_0}$},
 \\
 &\text{\rom{(2c)}}\quad s_1&=\fracc{|\M|}{p_1} &\ \text{ and 
          $F^{\Mi,s_1}_{p_1,q_1}\hookrightarrow L_\infty$, and    
        $F^{\Mi,s_0}_{p_0,q_0}\hookrightarrow L_\infty$ holds if $s_0=
        \fracc{|\M|}{p_0}$}, 
 \endalignedat
 \label{4.7}
 \end{equation}
the value $\fracc{|\M|}{r}=\frac{|\M|}{p^*_1}$ can be included when
$0<s_0\le\fracc{|\M|}{p_0}$. 
\end{thm}

\begin{proof} 
It is clear from \eqref{4.8}--\eqref{4.10} that $A_2$ can be obtained
with sum-exponent $q_2=\min(q_0\kd{s_0=s_1},q_1)$. Hence we focus on the
$r$-values. When $0<s_0\le\fracc{|\M|}{p_0}$  the lower bounds 
for $\fracc{|\M|}{r}$ in \eqref{4.8} and \eqref{4.9} 
are equal to $\fracc{|\M|}{p_1}+\fracc{|\M|}{p_0}-s_0$, and they satisfy
the inequality
 \begin{equation}
 0\,<\,\fracc{|\M|}{p_1}+\fracc{|\M|}{p_0}-s_0\le
 \fracc{|\M|}{p_1}+\fracc{|\M|}{p_0}-s_0+(s_1-\fracc{|\M|}{p_1})_+,
 \label{4.11}
 \end{equation}
so the continuity of $\pi_1,\pi_2$ and $\pi_3$ follows
under the assumptions in \eqref{4.4} and \eqref{4.7}.

For the case $F^{\Mi,s_0}_{p_0,q_0}\hookrightarrow L_\infty$, note that 
 \begin{equation}
 (\fracc{|\M|}{p_1}+\fracc{|\M|}{p_0}-s_0)_+\le\fracc{|\M|}{p_1}
 \le\fracc{|\M|}{p_1}+(s_1-\fracc{|\M|}{p_1}-(s_0-\fracc{|\M|}{p_0}))_+,
 \label{4.12}
 \end{equation}
and that the quantities are the lower bounds for $\fracc{|\M|}{r}$ in
\eqref{4.9}, \eqref{4.8} and \eqref{4.4} respectively. When 
$s_1-\fracc{|\M|}{p_1}-(s_0-\fracc{|\M|}{p_0})>0$ the right hand side
of \eqref{4.12} is equal to the lower bound of $\fracc{|\M|}r$ in 
\eqref{4.10}, so this case of \eqref{4.4} and (1a) is proved. When 
$s_1-\fracc{|\M|}{p_1}-(s_0-\fracc{|\M|}{p_0})\le0$, $\fracc{|\M|}{p_1}$
is the largest of the lower bounds of $\fracc{|\M|}r$. For
$s_0>\fracc{|\M|}{p_0}$ it is seen that the lower bound of
$\fracc{|\M|}r$ in \eqref{4.10} is $\le\fracc{|\M|}{p_1}$ with
equality only if $s_1-\fracc{|\M|}{p_1}=s_0-\fracc{|\M|}{p_0}$. But
then, since $s_0>\fracc{|\M|}{p_0}$, the $L_\infty$-condition in
\eqref{4.10} ff.\ is satisfied. For $s_0=\fracc{|\M|}{p_0}$ the lower
bound in \eqref{4.10} equals $\fracc{|\M|}{p_1}$, but it can be
included since $F^{\Mi,s_1}_{p_1,q_1}\hookrightarrow L_\infty$ if
$s_1=\fracc{|\M|}{p_1}$.
\end{proof} 

\begin{rem} \label{FFFopt-rem}
Concerning the spaces on the right hand side of \eqref{4.4} 
and concerning \eqref{4.6} and \eqref{4.7} it
should be observed explicitly that 
 \begin{itemize}
  \item $s_1$ is largest possible index that can occur there by (3) 
        in Theorem~\ref{ncss-thm} and the assumption $s_0\ge s_1$;
  \item the integral-exponent $r$ must satisfy $\fracc{|\M|}r\le
        \fracc{|\M|}p$ according to (4) in \eqref{2.14}, and for $s_2=s_1$
        it follows from (5) and (6) there that $\fracc{|\M|}r
        \ge\frac{|\M|}{p^*_1}$, cf.~\eqref{2.31};
  \item for $s_2=s_1$ the sum-exponent  $q_2=q$ is
        best possible according to $(3')$,
  \item for the cases of (1b), (2b) and (2c) with
        $s_j=\fracc{|\M|}{p_j}$, condition (7) in \eqref{2.14} gives that
        $p_j\le1$ if $\fracc{|\M|}{r}=\frac{|\M|}{p^*_1}$ is to be 
        obtained. But for $p_j\le1$ one has  
        $F^{\Mi,|\Mi|/p_j}_{p_j,q_j}\hookrightarrow L_\infty$
        except for $\M\ne(1,\dots,1)$, cf.~Remark~\ref{embd-rem}. Hence the
        $L_\infty$-conditions in \eqref{4.6} and \eqref{4.7} are 
        optimal for the isotropic $FFF$ cases.
 \end{itemize}
\end{rem}

Also every $A_2$ space with $s_2<s_1$ can be obtained from
Theorem \ref{FFF-thm}: When the result in the theorem is combined with Sobolev
embeddings, it is seen that $A_2$ can be obtained with 
$(s_2,\fracc{|\M|}{p_2})$ at (or arbitrarily close to) each point on 
the line characterised by 
$s-\fracc{|\M|}p=\min^+_{j=0,1}(s_j-\fracc{|\M|}{p_j})$, and simple
embeddings gives that the set $A_2$ ranges through is unbounded below.
Figure~\ref{embd-fig} illustrates this. 

According to (4)--(7) in \eqref{2.14} this procedure gives essentially
all the possible cases with $s_2<s_1$; only when $\M\ne(1,\dots,1)$ 
and the relevant $(s_j,p_j,q_j)$ equals
$(|\M|,1,q)$ for $1<q\le\infty$ is it open whether some of the
$L_\infty$-conditions stemming from \eqref{4.6} and \eqref{4.7} above can be
relaxed for the $A_2$ spaces with $s_2<s_1$. For $\M=(1,\dots,1)$ they cannot.

\begin{figure}[htb]
\hfil
\setlength{\unitlength}{0.0125in}
\begin{picture}(396,276)(0,0)
\path(270,180)(270,20)
\dashline{4.000}(147,177)(10,40)
\path(153,180)(270,180)
\path(240,175)(240,60)
\path(238.000,68.000)(240.000,60.000)(242.000,68.000)
\path(165,175)(120,130)
\path(124.243,137.071)(120.000,130.000)(127.071,134.243)
\path(120,125)(120,40)
\path(118.000,48.000)(120.000,40.000)(122.000,48.000)
\thicklines
\path(7,80)(360,80)
\path(352.000,78.000)(360.000,80.000)(352.000,82.000)
\path(10,0)(10,250)
\path(12.000,242.000)(10.000,250.000)(8.000,242.000)
\put(7,255){\makebox(0,0)[lb]{\raisebox{0pt}[0pt][0pt]
{\shortstack[l]{{$s$}}}}}
\put(0,77){\makebox(0,0)[lb]{\raisebox{0pt}[0pt][0pt]
{\shortstack[l]{{$0$}}}}}
\put(375,77){\makebox(0,0)[lb]{\raisebox{0pt}[0pt][0pt]
{\shortstack[l]{{$\fracc{|\M|}p$}}}}}
\put(148,178){\makebox(0,0)[lb]{\raisebox{0pt}[0pt][0pt]
{\shortstack[l]{{$\scriptstyle\circ$}}}}}
\end{picture}
\hfil
\caption{Embeddings giving the structure of $\Pl(\bigdot,\bigdot)$} 
 \label{embd-fig}
\end{figure}
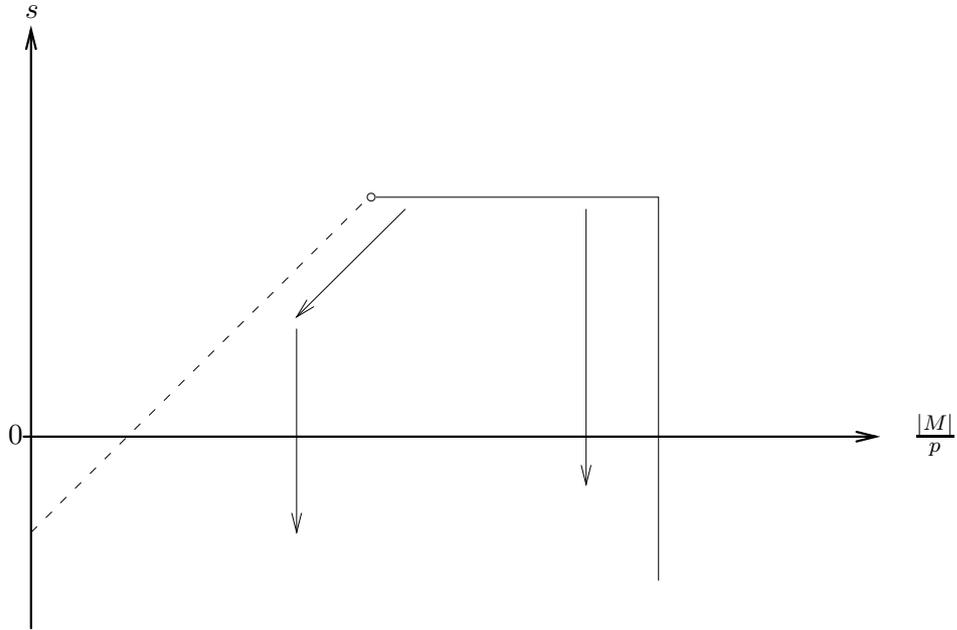

These additional results are summed up as follows:

\begin{thm} \label{FFF2-thm}
With assumptions as in Theorem~\ref{FFF-thm}, the product
$\pi(\cdot,\cdot)$ is bounded for each $s_2<s_1$ and $o\in\,]0,\infty]$
 \begin{equation}
 F^{\Mi,s_0}_{p_0,q_0}\oplus F^{\Mi,s_1}_{p_1,q_1}
 \xrightarrow{\;\pi(\cdot,\cdot)\;}
 \bigcap\bigl\{\,F^{\Mi,s_2}_{r,o}\bigm| 
(\tfrac{|\M|}{p^*_1}-(s_1-s_2))_+
         <\fracc{|\M|}{r}\le \fracc{|\M|}{p} \,\bigr\}.
 \label{4.4-2}
 \end{equation}
In addition, when $\frac{|\M|}{p^*_1}-s_1+s_2>0$, this value of 
$\fracc{|\M|}r$ may be included when
$F^{\Mi,s_0}_{p_0,q_0}\hookrightarrow L_\infty$ and 
$0<s_0\le\fracc{|\M|}{p_0}$ under each of the conditions in 
\eqref{4.6} and \eqref{4.7} respectively. 
\end{thm}

Altogether it can be concluded for the isotropic $FFF$ cases,
where $\M=(1,\dots,1)$, that the set
$\Pl(F^{\Mi,s_0}_{p_0,q_0}, F^{\Mi,s_1}_{p_1,q_1})$
is completely described by Theorems~\ref{FFF-thm} and \ref{FFF2-thm} 
when $(s_j,p_j,q_j)_{j=0,1}\in\Dm(\pi)$.

\bigskip

For the borderline cases with $s_0=-s_1=:s>0$ condition (1) in
Theorem~\ref{ncss-thm} reduces to $1\ge\fracc1{p_0}+\fracc1{p_1}$.
Concerning the notation in \eqref{4.4'} we observe that $p^*_1=p_1$ 
for $F^{\Mi,s_0}_{p_0,q_0}\hookrightarrow L_\infty$ whilst
$\frac{|\M|}{p^*_1}=\fracc{|\M|}{p_1}+\fracc{|\M|}{p_0}-s$ for $0<s_0\le\fracc{|\M|}{p_0}$. Moreover, for $s=\fracc{|\M|}{p_0}\ge
|\M|$ there is not any space $F^{\Mi,-s}_{p_1,q_1}$ such that (1) in
Theorem~\ref{ncss-thm} holds (since $0<p_1<\infty$); hence the case with
$F^{\Mi,s}_{p_0,q_0}\hookrightarrow L_\infty$ is reduced to 
$s>\fracc{|\M|}{p_0}$.

\begin{thm} \label{FFF3-thm}
Let $(s,p_0,q_0)$ and $(-s,p_1,q_1)$ satisfy the three inequalities $s>0$,
$1\ge\fracc1{p_0}+\fracc1{p_1}$ and $\fracc1{q_0}+\fracc1{q_1}\ge 1$.
Then $\pi(\cdot,\cdot)$ is bounded
 \begin{equation}
 F^{\Mi,s}_{p_0,q_0}\oplus F^{\Mi,-s}_{p_1,q_1}
 \xrightarrow{\;\pi(\cdot,\cdot)\;}
 \bigcap\bigl\{\,F^{\Mi,-s}_{r,q_1}\bigm| \tfrac{|\M|}{p^*_1}
         <\fracc{|\M|}{r}\le \fracc{|\M|}{p} \,\bigr\}.
 \label{4.4-3}
 \end{equation}
When $s>\fracc{|\M|}{p_0}$, or when 
$s<\fracc{|\M|}{p_0}$ and $\fracc1{p_0}+\fracc1{p_1}<1$, the space
$F^{\Mi,-s}_{p^*_1,q_1}$ may be included, whereas when $s<\fracc{|\M|}{p_0}$
and $\fracc1{p_0}+\fracc1{p_1}=1$ the space $B^{\Mi,-s}_{p^*_1,\infty}$ can
be used instead.
\end{thm}

\begin{proof} We have by Proposition~\ref{zero-prop} continuous mappings
 \begin{equation}
 F^{\Mi,s}_{p_0,q_0}\oplus F^{\Mi,-s}_{p_1,q_1}\xrightarrow{\;\pi_2(\cdot,\cdot)\;}
 L_p\hookrightarrow 
 \begin{cases}F^{\Mi,-s}_{r_0,q} &\quad\text{for $p>1$},\\
 B^{\Mi,-s}_{r_0,\infty} &\quad\text{for $p=1$},  \end{cases}
 \label{4.13'}
 \end{equation}
when $\fracp=\fracc1{p_0}+\fracc1{p_1}$,
$\frac1q=\fracc1{q_0}+\fracc1{q_1}$ and
$\fracc{|\M|}{r_0}=\fracc{|\M|}{p_1}+\fracc{|\M|}{p_0}-s$ 
(since the estimate there is shown as in \eqref{3.21} ff.). Moreover
$L_p\hookrightarrow\bigcap\bigl\{\,F^{\Mi,-s}_{r,q}\bigm|
 \fracc{|\M|}{r_0}<\fracc{|\M|}{r}\le\fracc{|\M|}{p} \,\bigr\} $
for $p\ge1$. Since $s_1<0$ the same result holds for
$\pi_3(\cdot,\cdot)$, and when this is combined with \eqref{3.8} for
$\pi_1(\cdot,\cdot)$ the theorem follows.
\end{proof}

The optimality of the receiving spaces above may be verified in the
same way as for Theorem~\ref{FFF-thm} in Remark~\ref{FFFopt-rem} ff., 
except for the statement on $B^{\Mi,-s}_{p^*_1,\infty}$. 

Also in this case does one get $A_2$ spaces
with $s_2<s_1$ by means of embeddings.

Concerning earlier treatments of $FFF$ cases with $s_0+s_1=0$ we
mention \cite{GS2}, where $F^{\Mi,s}_{2,2}\oplus F^{\Mi,-s}_{2,2}$
with $\M=(1,\dots,1,m_n)$ is given a treatment directly based on the
Fourier transformation.

\bigskip

For the borderline cases with
$s_0+s_1=\fracc{|\M|}{p_0}+\fracc{|\M|}{p_1}-|\M|$ one can assume that
$s_0+s_1>0$ since the case with $s_0+s_1=0$ is covered above (for
$s_0>0$ at least). Then
$\fracp:=\fracc1{p_0}+\fracc1{p_1}>1$. Here $p^*_1=p_1$ for
$F^{\Mi,s_0}_{p_0,q_0}\hookrightarrow L_\infty$ while
$\frac{|\M|}{p^*_1}=\fracc{|\M|}{p_1}+\fracc{|\M|}{p_0}-s_0$ for
$0<s_0\le\fracc{|\M|}{p_0}$. 

\begin{thm} \label{FFF4-thm}
Let $s_0+s_1=\fracc{|\M|}{p_0}+\fracc{|\M|}{p_1}-|\M|>0$ and suppose
that $s_1<0$. Then $\pi(\cdot,\cdot)$ is bounded
 \begin{equation}
 F^{\Mi,s_0}_{p_0,q_0}\oplus F^{\Mi,s_1}_{p_1,q_1}
 \xrightarrow{\;\pi(\cdot,\cdot)\;}
 \bigcap\bigl\{\,F^{\Mi,s_1}_{r,q_1}\bigm| \tfrac{|\M|}{p^*_1}
         <\fracc{|\M|}{r}\le |\M| \,\bigr\}.
 \label{4.4-4}
 \end{equation}
When $F^{\Mi,s_0}_{p_0,q_0}\hookrightarrow L_\infty$ the space 
$F^{\Mi,s_1}_{p_1,q_1}$ may be included for
$s_0>\fracc{|\M|}{p_0}$, and $B^{\Mi,s_1}_{p_1,\infty}$ 
can receive for $s_0=\fracc{|\M|}{p_0}$. Similarly
$B^{\Mi,s_1}_{p^*_1,\infty}$ may be included for $0<s_0<\fracc{|\M|}{p_0}$.
\end{thm}

\begin{proof} For $\pi_2$ one has continuous mappings (with
$\fracc{|\M|}{r_0}=\fracc{|\M|}{p_1}+\fracc{|\M|}{p_0}-s_0$) 
 \begin{equation}
 F^{\Mi,s_0}_{p_0,q_0}\oplus F^{\Mi,s_1}_{p_1,q_1}\xrightarrow{\;\pi_2\;}
 L_1\hookrightarrow B^{\Mi,s_1}_{r_0,\infty}\cap 
 (\bigcap\bigl\{\,F^{\Mi,s_1}_{r,q_1}\bigm| r_0\ge r\ge 1\,\bigr\} )
 \label{4.5-4'}
 \end{equation}
at least if $\fracc1q:=\fracc1{q_0}+\fracc1{q_1}\ge 1$,
cf.~Proposition~\ref{border-prop}. In general $q\le 1$ may be
achieved from Sobolev embeddings of the $F^{\Mi,s_j}_{p_j,q_j}$.
Since $s_1<0$ the proof may be conducted along the lines of that of
Theorem~\ref{FFF3-thm}. 
\end{proof}

On one hand it is not clear whether the $B^{\Mi,s_1}_{p^*_1,\infty}$ are
optimal, on the other hand\,---\,by a Sobolev embedding of
$F^{\Mi,s_1}_{p_1,q_1}$\,---\,the cases with $s_0\ge s_1\ge0$ and 
 \begin{equation}
 s_0+s_1=\fracc{|\M|}{p_0}+\fracc{|\M|}{p_1}-|\M|>0
 \label{4.5-4''}
 \end{equation}
are covered too,
but only with receiving $A_2$ spaces for which $s_2<\min(s_0,s_1)$.

However, this suffices to see that $(s_j,p_j,q_j)_{j=0,1}\in\Dm(\pi,FF)$ when
\eqref{4.5-4''} holds.

\subsection{The $BBB$ cases} \label{BBB-ssect}
We shall now modify the arguments in the subsection above and obtain
analogous results for the Besov spaces.

\begin{thm} \label{BBB-thm} 
Let $\M$ and the numbers $s_0$ and $s_1\in\R$ and
$p_0,p_1,q_0$ and $q_1\in\,]0,\infty]$ be given such that
$s_0+s_1>|\M|\max(0,\fracc1{p_0}+\fracc1{p_1}-1)$.

For $s_0\ge s_1$ and with $p$, $p^*_1$ and $q$ given by
\eqref{4.4'}, the product $\pi(\cdot,\cdot)$ is continuous
 \begin{equation}
 B^{\Mi,s_0}_{p_0,q_0}\oplus B^{\Mi,s_1}_{p_1,q_1}
 \xrightarrow{\;\pi(\cdot,\cdot)\;}
 \bigcap\bigl\{\,B^{\Mi,s_1}_{r,q}\bigm| \tfrac{|\M|}{p^*_1}
          <\fracc{|\M|}{r}\le \fracc{|\M|}{p} \,\bigr\}.
 \label{4.14}
 \end{equation}
Furthermore one can include $B^{\Mi,s_1}_{p^*_1,o}$ on the right 
hand side of \eqref{4.14} when $B^{\Mi,s_0}_{p_0,q_0}\hookrightarrow
L_\infty$ holds in addition to one of the following conditions:
 \begin{equation}
 \begin{alignedat}{2}
 &\text{\rom{(1a)}} &\quad s_1-\fracc{|\M|}{p_1}&\ge
 s_0-\fracc{|\M|}{p_0},\ o\ge\max(q_0,q_1), \text{ and 
      $B^{\Mi,s_1}_{p_1,q_1}\hookrightarrow L_\infty$ if
           $s_1=\fracc{|\M|}{p_1}$},
 \\
 &\text{\rom{(1b)}} &\quad 0\le s_1-\fracc{|\M|}{p_1} &<
  s_0-\fracc{|\M|}{p_0},\ o\ge q,
 \\
 &\text{\rom{(1c)}} &\quad s_1-\fracc{|\M|}{p_1} &<
  s_0-\fracc{|\M|}{p_0},\ 0<s_1<\frac{|\M|}{p_1},\ o\ge q,
 \\
 &&&\hphantom{s_0-\fracc{|\M|}{p_0},\ 0}
  \text{and $q_1<t_1$ if $s_1=s_0=\fracc{|\M|}{p_0}$},
 \\ 
 &\text{\rom{(1d)}} &\quad s_1-\fracc{|\M|}{p_1} &<
  s_0-\fracc{|\M|}{p_0},\ s_1\le0,\ s_1<\frac{|\M|}{p_1}\text{ and } o\ge q_1,
 \end{alignedat}
 \label{4.16}
 \end{equation}
where $t_j=|\M|(\fracc{|\M|}{p_j}-s_j)^{-1}$ for $j=0$ and $1$,
and under each of the conditions
 \begin{equation}
 \begin{alignedat}{2}
 &\text{\rom{(2a)}}&\quad s_1&>\fracc{|\M|}{p_1} 
           \text{ and }o\ge\max(q_0,q_1),
 \\
 &\text{\rom{(2b)}}&\quad s_1&=\fracc{|\M|}{p_1},\
  B^{\Mi,s_1}_{p_1,q_1}\hookrightarrow L_\infty,\ s_0<\fracc{|\M|}{p_0},\
 o\ge\max(q_0,q_1) \text{ and $q_0<t_0$},
 \\
 &\text{\rom{(2c)}}&\quad 0<s_1&<\fracc{|\M|}{p_1},\
 s_0<\fracc{|\M|}{p_0},\ s_0>s_1,\ o\ge q_1
  \text{ and $q_0<t_0$},
 \\
 &\text{\rom{(2d)}}&\quad 0<s_1&<\fracc{|\M|}{p_1},\
 s_0<\fracc{|\M|}{p_0},\ s_0=s_1,\ o\ge q_j
  \text{ and $q_j<t_j$ for $j=0$, $1$},
 \\
 &\text{\rom{(2e)}} &\quad s_1&<\fracc{|\M|}{p_1},\ 
 s_0<\fracc{|\M|}{p_0},\ s_1\le0,
 \ o\ge q_1 \text{ and $q_0<t_0$},
 \end{alignedat}
 \label{4.17}
 \end{equation}
the space $B^{\Mi,s_1}_{p^*_1,o}$ can be included when
$0<s_0\le\fracc{|\M|}{p_0}$. 

In \eqref{4.16} and \eqref{4.17} it suffices with
$q_j=t_j$ if either $t_j\le2$ or $\M=(1,\dots,1)$.
\end{thm}
 
\begin{proof} 
Since \eqref{4.14} may be obtained analogously to
the corresponding $FFF$ cases, we focus on the modifications that 
lead to \eqref{4.16} and \eqref{4.17}.

In (1a) one can use that $B^{\Mi,s_j}_{p_j,q_j}\hookrightarrow 
L_\infty$ for $j=0$ and $1$, whereby the restriction $o\ge q_1$ occurs
in the treatment of $\pi_1(\cdot,\cdot)$ and $o\ge q_0$ comes from
$\pi_3(\cdot,\cdot)$, cf.~\eqref{4.8'}, \eqref{4.8a}--\eqref{4.8b},
respectively \eqref{4.10a}--\eqref{4.10b}. To treat (1b) for
$s_1=\fracc{|\M|}{p_1}>0$ is simpler since
$(\fracc{|\M|}{p_1}+\fracc{|\M|}{p_0}-s_0)_+<\fracc{|\M|}{p_1}$ (so
that it suffices to obtain $\fracc{|\M|}r$ equal to the lower bound in
\eqref{4.8'} only). For $s_1=\fracc{|\M|}{p_1}=0$ one can apply
\eqref{4.10e} for the estimate of $\pi_3$.

For $s_0>\fracc{|\M|}{p_0}$ condition (1c) is easy since \eqref{4.8a},
\eqref{4.9'} and \eqref{4.10'} provide the necessary results. However,
for $s_0=\fracc{|\M|}{p_0}$ the three lower bounds of $\fracc{|\M|}r$
in \eqref{4.8'}--\eqref{4.10'} all coincide with $\fracc{|\M|}{p_1}$,
but at least for $s_1=s_0$ one can apply \eqref{4.10c}. 

The case with $s_1<s_0=\fracc{|\M|}{p_0}$ requires a special treatment
of $\pi_3$ (to allow $q_1\ge t_1$). Note that $A_1\hookrightarrow 
B^{\Mi,-\varepsilon}_{p_\varepsilon,q_1}$ provided
$s_1-\fracc{|\M|}{p_1}\ge -\varepsilon-\fracc{|\M|}{p_\varepsilon}$
and $p_\varepsilon\ge p_1$, so that under the restriction
$\varepsilon\in \,]0,\fracc{|\M|}{p_1}-s_1]$ this embedding exists for
$\fracc{|\M|}{p_\varepsilon}=\fracc{|\M|}{p_1}-s_1-\varepsilon$.
We shall now take $\varepsilon<s_0-s_1$ and show that, with
$\fracc{|\M|}{r_\varepsilon}=\fracc{|\M|}{p_0}+\fracc{|\M|}{p_\varepsilon}$
and $\fracc1{q_2}=\fracc1{q_0}+\fracc1{q_1}$, 
 \begin{equation}
 A_0\oplus A_1\hookrightarrow A_0\oplus
 B^{\Mi,-\varepsilon}_{p_\varepsilon,q_1} \xrightarrow{\;\pi_3\;}
 B^{\Mi,s_0-\varepsilon}_{r_\varepsilon,q_2}\hookrightarrow
 B^{\Mi,s_1}_{p_1,q_1}.
 \label{4.20}
 \end{equation}
According to Theorem~\ref{basic-thm} it is sufficient to verify the last
embedding, i.e.,
 \begin{equation}
 s_0-\varepsilon-\fracc{|\M|}{r_\varepsilon}\ge s_1-\fracc{|\M|}{p_1},\qquad
 \fracc{|\M|}{r_\varepsilon}\ge \fracc{|\M|}{p_1}.
 \label{4.22}
 \end{equation}
Note that $\fracc{|\M|}{r_\varepsilon}=\fracc{|\M|}{p_0}+
\fracc{|\M|}{p_1}-s_1-\varepsilon$, so the former inequality reduces to
$s_0\ge\fracc{|\M|}{p_0}$ and the latter to 
$\fracc{|\M|}{p_1}-s_1\ge\varepsilon$.

Concerning (1d) it is seen that \eqref{4.10'} suffices for $\pi_3$
when $s_0>\fracc{|\M|}{p_0}$, whereas \eqref{4.10e} applies for
$s_0=\fracc{|\M|}{p_0}$. 

\medskip

(2a) is based on \eqref{4.10a} and (2b) combines \eqref{4.8b} and
\eqref{4.10b}. Concerning (2c) the cases $q_1\ge t_1$ are included as
in (1c) above, except that the embedding
$B^{\Mi,s_0-\varepsilon}_{r_\varepsilon,q_2}\hookrightarrow
B^{\Mi,s_1}_{p^*_1,q_2}$ is wanted. However, this is obtained since 
$\frac{|\M|}{p^*_1}=\fracc{|\M|}{r_\varepsilon}+s_1+
\varepsilon-s_0<\fracc{|\M|}{r_\varepsilon}$ (using
$\varepsilon<s_0-s_1$) and since 
$s_0-\varepsilon-\fracc{|\M|}{r_\varepsilon}=s_0-\varepsilon-
\fracc{|\M|}{p_0}-(\fracc{|\M|}{p_1}-s_1-\varepsilon)=s_1-
\frac{|\M|}{p^*_1}$.

From \eqref{4.8c} and \eqref{4.10c} one derives (2d), and \eqref{4.10e}
may be used for (2e).

The final statement in the theorem is based on \eqref{4.8d} and \eqref{4.10d}.
\end{proof}

\begin{rem} \label{BBBopt-rem}
As for the $FFF$ cases it is seen that the spaces on the right hand
sides of \eqref{4.14} can neither have  smoothness
indices larger than $s_1=\min(s_0,s_1)$ nor have integral-exponents outside
$[\frac{|\M|}{p^*_1}, \fracc{|\M|}{p}]$ when the index is $s_1$. 
And, still for index $s_1$, the sum-exponent can not be lower than $q$. 

The possibility of obtaining $\fracc{|\M|}{r}$ equal to
$\frac{|\M|}{p^*_1}$  is more
delicate. Observe that the conditions on $s_j-\fracc{|\M|}{p_j}$
and $s_j$ in (1a)--(1d) and (2a)--(2e) exhaust the possibilities.
(In (1d) the subcase $s_1=0=\fracc{|\M|}{p_1}$ of (1b) is omitted,
and in (2b)--(2e) the possibility $s_0=\fracc{|\M|}{p_0}$ is excluded 
because of overlap with the case $B^{\Mi,s_0}_{p_0,q_0}\hookrightarrow
L_\infty$.) 

Concerning the optimality of \eqref{4.16} and \eqref{4.17} we have that
 \begin{itemize}
  \item $o\ge\max(q_0,q_1)$ in (1a), (2a) and (2b) is necessary by
        $(3')$ and $(5')$ in Theorem~\ref{ncss-thm} and in (2d) by 
        $(3')$ applied for $j=0$ and $j=1$. Similarly $o\ge q$ is 
        necessary in (1b) and (1c) and $o\ge q_1$ in (1d), (2c) and (2e);
  \item the $L_\infty$-conditions in (1a) and (2b) are
        unremovable by (7) in \eqref{2.14} and Remark~\ref{embd-rem}.
  \item the conditions $q_j<t_j$ ocurring in (1c) and (2b)--(2e) are
        not in the present paper shown to be necessary. 

        For case (2d) there
        are counterexamples in \cite{ST} that show for $\M=(1,\dots,1)$ that
        $q_0\le t_0$ and $q_1\le t_1$ must hold, but for the other cases 
        it seems to be an open problem whether these conditions 
        are necessary. 
 \end{itemize}
\end{rem}

\begin{rem} \label{vertex-rem}
The conditions $q_j<t_j$ are connected to the question
of having $A_2$ represented by the ``$\circ$'' in Figure~\ref{A2-fig},
and the answer depends on $A_0\oplus A_1$ instead
of on $A_2$, contrary to the rest of the line segment with
$s=\min(s_0,s_1)$. It should be observed that such conditions apply 
only for {\em one value\/} of $j$, except for (2d) where
$s_0=s_1$. 

For this reason the cases with $s_0>s_1\ge0$ can not be
reduced to those with $s_0=s_1$ by means of a Sobolev embedding
$B^{\Mi,s_0}_{p_0,q_0}\hookrightarrow B^{\Mi,s_1}_{r,q_0}$. 
\end{rem}

\bigskip

Because of the condition on the sum-exponents in the Sobolev
embeddings for the Besov spaces, the question of having $A_2$ with
$s_2<s_1=\min(s_0,s_1)$ and $\fracc{|\M|}{p_2}=s_2-
\min^+(s_j-\fracc{|\M|}{p_j})$ is more complicated for the $BBB$ cases
than for the $FFF$ cases.  

However it suffices to use embeddings of $B^{\Mi,s_0}_{p_0,q_0}$ and
$B^{\Mi,s_1}_{p_1,q_1}$:

\begin{thm} \label{BBB2-thm}
With assumptions as in Theorem~\ref{BBB-thm}, the product
$\pi(\cdot,\cdot)$ is bounded for each $s_2<s_1$ and $o\in\,]0,\infty]$
 \begin{equation}
 B^{\Mi,s_0}_{p_0,q_0}\oplus B^{\Mi,s_1}_{p_1,q_1}
 \xrightarrow{\;\pi(\cdot,\cdot)\;}  
 \bigcap\bigl\{\,B^{\Mi,s_2}_{r,o}\bigm| 
 (\tfrac{|\M|}{p^*_1}-(s_1-s_2))_+
         <\fracc{|\M|}{r}\le \fracc{|\M|}{p} \,\bigr\}.
 \label{4.31}
 \end{equation}
Moreover, when $\frac{|\M|}{p^*_1}-s_1+s_2\ge0$ the space $B^{\Mi,s_2}_{r,o}$
with $\fracc{|\M|}r=\frac{|\M|}{p^*_1}-s_1+s_2$
may be included if $B^{\Mi,s_0}_{p_0,q_0}\hookrightarrow L_\infty$
holds in addition to one of the conditions 
 \begin{equation}
 \begin{alignedat}{2}
 &\text{\rom{(1a)}} &\quad s_1-\fracc{|\M|}{p_1}&>
 s_0-\fracc{|\M|}{p_0},\ o\ge q_0, 
 \\
 &\text{\rom{(1b)}} &\quad s_1-\fracc{|\M|}{p_1}&=
 s_0-\fracc{|\M|}{p_0},\ o\ge\max(q_0,q_1), \text{ and 
      $B^{\Mi,s_1}_{p_1,q_1}\hookrightarrow L_\infty$ if
           $s_1=\fracc{|\M|}{p_1}$},
 \\
 &\text{\rom{(1c)}} &\quad s_1-\fracc{|\M|}{p_1} &<
  s_0-\fracc{|\M|}{p_0},\ o\ge q_1,
 \end{alignedat}
 \label{4.33}
 \end{equation}
and  when $\frac{|\M|}{p^*_1}-s_1+s_2\ge0$ each of the conditions 
 \begin{equation}
 \begin{alignedat}{2}
 &\text{\rom{(2a)}}&\quad s_1&>\fracc{|\M|}{p_1} 
           \text{ and }o\ge q_0,
 \\
 &\text{\rom{(2b)}}&\quad s_1&=\fracc{|\M|}{p_1},\
  B^{\Mi,s_1}_{p_1,q_1}\hookrightarrow L_\infty,\ s_0<\fracc{|\M|}{p_0},\
 o\ge q_0,
 \\
 &\text{\rom{(2c)}}&\quad s_1&<\fracc{|\M|}{p_1},\
 s_0<\fracc{|\M|}{p_0},\ o\ge (\fracc1{q_0}+\fracc1{q_1})^{-1}
 \end{alignedat}
 \label{4.34}
 \end{equation}
allow the space $B^{\Mi,s_2}_{r,o}$ with $\fracc{|\M|}{r}=
\frac{|\M|}{p^*_1}-s_1+s_2$ to be included when $0<s_0\le\fracc{|\M|}{p_0}$.
\end{thm}

\begin{proof} Formula \eqref{4.31} is obtained from
Theorem~\eqref{BBB-thm} by means of embeddings,
cf.~Figure~\ref{embd-fig}. Hence it remains to show the sufficiency of
the conditions (1a)--(2c). 

In case (1a) the receiving space $A_2=B^{\Mi,s_2}_{r,o}$ may be
obtained for $\fracc{|\M|}{r}=\frac{|\M|}{p^*_1}-s_1+s_2$ and
$o=\max(q_0,q_1)$ by means of a Sobolev embedding of the space covered
by Theorem~\ref{BBB-thm}. To obtain $o$ as low as $q_0$ it suffices to
improve the estimate of $\pi_1$, for the $\pi_2$ estimate gives
$o=(\fracc1{q_0}+\fracc1{q_1})^{-1}$ or better. 

It is enough to treat $s_2=s_1-\varepsilon$ for
arbitrarily small $\varepsilon>0$. By \eqref{4.8'} 
 \begin{equation}
 A_0\oplus A_1\hookrightarrow
 B^{\Mi,s_0}_{p_0,q_0}\oplus B^{\Mi,s_1-\varepsilon}_{p_1,q_0}
 \xrightarrow{\;\pi_1\;}
 \bigcap\bigl\{\,B^{\Mi,s_2}_{r,q_0}\bigm|
   \fracc{|\M|}{p_1}< \fracc{|\M|}{r}\le
   \fracc{|\M|}{p_0}+\fracc{|\M|}{p_1} \,\bigr\}.
 \label{4.35} 
 \end{equation}
Taking $\varepsilon$ so small that $s_1-\varepsilon-\fracc{|\M|}{p_1}>
s_0-\fracc{|\M|}{p_0}$ the value $\fracc{|\M|}{r}=
\frac{|\M|}{p^*_1}-\varepsilon=\fracc{|\M|}{p_0}-s_0+s_1-\varepsilon$ 
is included on the right hand side. This proves (1a) and (1b) is
evident from Theorem~\ref{BBB-thm}.

For the treatment of (1c) it is necessary to modify the estimate of
$\pi_3(\cdot,\cdot)$. Consider $s_2=s_1-\varepsilon$ for small 
$\varepsilon\in\,]0,\fracc{|\M|}{p_1}-s_1]$ and take
$r\in\,]p_1,\infty]$ such that
$s_1-\fracc{|\M|}{p_1}=-\varepsilon-\fracc{|\M|}{r}$. Then
$B^{\Mi,s_1}_{p_1,q_1}\hookrightarrow B^{\Mi,-\varepsilon}_{r,q_1}$.
From the last statement in Theorem~\ref{basic-thm} it is seen that
$\pi_3$ maps $B^{\Mi,s_0}_{p_0,q_0}\oplus B^{\Mi,-\varepsilon}_{r,q_1}$ 
into $B^{\Mi,s_0-\varepsilon}_{r_0,o}$ when
$\fracc{|\M|}{r_0}=\fracc{|\M|}{p_0}+\fracc{|\M|}{r}$ and
$o=(\fracc1{q_0}+\fracc1{q_1})^{-1}$. But
 \begin{equation}
 s_0-\varepsilon-\fracc{|\M|}{r_0}=s_0-\fracc{|\M|}{p_0}+
 s_1-\fracc{|\M|}{p_1}\ge
 s_1-\fracc{|\M|}{p_1}=s_2-(\tfrac{|\M|}{p^*_1}-s_1+s_2),
 \label{4.36}
 \end{equation}
so $\pi_3$ maps $B^{\Mi,s_0}_{p_0,q_0}\oplus B^{\Mi,s_1}_{p_1,q_1}$
into $B^{\Mi,s_2}_{r,q_0}$ for
$\fracc{|\M|}{r}=\frac{|\M|}{p^*_1}-s_1+s_2$.

For (2a) the treatment of $\pi_1$ in case (1a) is easily modified by
taking $\varepsilon$ such that $s_1-\varepsilon>\fracc{|\M|}{p_1}$. In
(2b) one can handle $\pi_1$ by means of \eqref{3.18}: For
$s_2=s_1-\varepsilon$ with a small
$\varepsilon\in\,]0,\fracc{|\M|}{p_0}-s_0]$ the number
$r\in\,]p_0,\infty]$ such that
$s_0-\fracc{|\M|}{p_0}=-\varepsilon-\fracc{|\M|}{r}$ is considered.
Then $\pi_1(B^{\Mi,s_0}_{p_0,q_0}\oplus B^{\Mi,s_1}_{p_1,q_1})\subset
B^{\Mi,s_1-\varepsilon}_{r_0,o}$ for $\fracc{|\M|}{r_0}=
\fracc{|\M|}{p_1}+\fracc{|\M|}{r}$ and
$o=(\fracc1{q_0}+\fracc1{q_1})^{-1}$. Here $\fracc{|\M|}{r_0}=
\frac{|\M|}{p^*_1}-\varepsilon$ is equal to the value of 
$\fracc{|\M|}{r}$ pertinent to \eqref{4.34}. 

In (2c) the case with $s_1<0$ is easy concerning $\pi_3$ since we
already have a receiving space (for this operator) with sum-exponent
$(\fracc1{q_0}+\fracc1{q_1})^{-1}$, cf.~\eqref{4.10e}. Observe that
$\pi_1$ can be treated as in case (2b), and that for  $s_1\ge0$ one
may treat $\pi_3$ along these lines too. Indeed, if
$\varepsilon\in\,]0,\fracc{|\M|}{p_1}-s_1]$ and
$s_1-\fracc{|\M|}{p_1}=-\varepsilon-\fracc{|\M|}{r}$ it is found that
$\pi_1(B^{\Mi,s_0}_{p_0,q_0}\oplus B^{\Mi,s_1}_{p_1,q_1})\subset
B^{\Mi,s_0-\varepsilon}_{r_0,o}$ for $\fracc{|\M|}{r_0}=
\fracc{|\M|}{p_0}+\fracc{|\M|}{r}$. Here
$\fracc{|\M|}{r_0}=\fracc{|\M|}{r_2}-\varepsilon+s_0-s_1$, and hence
$B^{\Mi,s_0-\varepsilon}_{r_0,o}\hookrightarrow 
B^{\Mi,s_1-\varepsilon}_{r,o}$ for
$\fracc{|\M|}{r}=\frac{|\M|}{p^*_1}-\varepsilon$.
\end{proof}

The optimality of \eqref{4.31}--\eqref{4.34} is seen from
(4)--(7), $(5')$ and $(6')$ in Theorem~\ref{ncss-thm}, 
cf.~also \eqref{2.31}.

Thus Theorem~\ref{BBB2-thm} gives a complete description of the
$(s_2,p_2,q_2)\in\Pl(B^{\Mi,s_0}_{p_0,q_0},
B^{\Mi,s_1}_{p_1,q_1})$ with $s_2<\min(s_0,s_1)$ when
$(s_j,p_j,q_j)_{j=0,1} \in \Dm(\pi)$.

\bigskip

The borderline cases with
$s_0+s_1=(\fracc{|\M|}{p_0}+\fracc{|\M|}{p_1}-|\M|)_+$ are now given a
treatment analogous to the one in Section~\ref{FFF-ssect}.

\begin{thm} \label{BBB3-thm}
Let $(s,p_0,q_0)$ and $(-s,p_1,q_1)$ satisfy the three inequalities $s>0$,
$1\ge\fracc1{p_0}+\fracc1{p_1}$ and $\fracc1{q_0}+\fracc1{q_1}\ge 1$,
and recall \eqref{4.4'}. Then $\pi(\cdot,\cdot)$ is bounded
 \begin{equation}
 B^{\Mi,s}_{p_0,q_0}\oplus B^{\Mi,-s}_{p_1,q_1}
 \xrightarrow{\;\pi(\cdot,\cdot)\;}
 \bigcap\bigl\{\,B^{\Mi,-s}_{r,q_1}\bigm| \tfrac{|\M|}{p^*_1}
         <\fracc{|\M|}{r}\le \fracc{|\M|}{p} \,\bigr\}.
 \label{4.43}
 \end{equation}
In \eqref{4.43} the space $B^{\Mi,-s}_{p^*_1,o}$ may be included if
$B^{\Mi,s}_{p_0,q_0}\hookrightarrow L_\infty$ 
holds together with one of the conditions 
 \begin{equation}
 \begin{alignedat}{2}
 &\text{\rom{(1a)}}&\quad s&>\fracc{|\M|}{p_0} \text{ and } o\ge q_1,\\
 &\text{\rom{(1b)}}&\quad s&=\fracc{|\M|}{p_0},\ p>1 \text{ and }
  o\ge\max(q_1,p+\varepsilon) \text{ for $\varepsilon>0$},\\
 &\text{\rom{(1c)}}&\quad s&=\fracc{|\M|}{p_0},\ p=1 \text{ and } o=\infty,
 \end{alignedat}
 \label{4.45}
 \end{equation}
and if one of the following conditions (where
$t_0=|\M|(\fracc{|\M|}{p_0}-s)^{-1}$)
 \begin{equation}
 \begin{alignedat}{2}
 &\text{\rom{(2a)}}&\quad s&<\fracc{|\M|}{p_0},\ q_0<t_0,\ p>1
    \text{ and } o\ge \max(q_1,p+\varepsilon),\\
 &\text{\rom{(2b)}}&\quad s&<\fracc{|\M|}{p_0},\ q_0<t_0,\ p=1 \text{ and }
 o=\infty,
 \end{alignedat}
 \label{4.46}
 \end{equation} 
holds, the space $B^{\Mi,-s}_{p^*_1,o}$ can receive when 
$0<s<\fracc{|\M|}{p_0}$. 

It suffices with $\varepsilon=0$ above if either $p\ge2$ or
$\M=(1,\dots,1)$, and it suffices with $q_0\le t_0$ if either $t_0\le 2$
or $\M=(1,\dots,1)$.
\end{thm}

\begin{proof} The property analogous to the one in \eqref{4.13'} is  
 \begin{equation}
 B^{\Mi,s}_{p_0,q_0}\oplus B^{\Mi,-s}_{p_1,q_1}\xrightarrow{\;\pi_2(\cdot,\cdot)\;}
 L_p\hookrightarrow 
 \begin{cases}B^{\Mi,-s}_{r_0,p+\varepsilon} &\quad\text{for $p>1$},\\
 B^{\Mi,-s}_{r_0,\infty} &\quad\text{for $p=1$},  \end{cases}
 \label{4.47}
 \end{equation}
for $\varepsilon>0$; even $\varepsilon=0$ is possible if either
$p\ge2$ or $\M=(1,\dots,1)$. 
\end{proof}

It is not clear whether the sum-exponents $o$ in 
(1a)--(1c) and (2a)--(2b) are optimal. For this
reason, we shall not treat $A_2$ spaces with
$s_2<\min(s_0,s_1)$ here; results may be obtained by use of the same
methods as for Theorem~\ref{BBB2-thm} when needed.

\begin{thm} \label{BBB4-thm}
Let $s_0+s_1=\fracc{|\M|}{p_0}+\fracc{|\M|}{p_1}-|\M|>0$ and suppose
that $s_1<0$ and $\fracc1{q_0}+\fracc1{q_1}\ge 1$. 
Then $\pi(\cdot,\cdot)$ is bounded
 \begin{equation}
 B^{\Mi,s_0}_{p_0,q_0}\oplus B^{\Mi,s_1}_{p_1,q_1}
 \xrightarrow{\;\pi(\cdot,\cdot)\;}
 \bigcap\bigl\{\,B^{\Mi,s_1}_{r,q_1}\bigm| \tfrac{|\M|}{p^*_1}
         <\fracc{|\M|}{r}\le |\M| \,\bigr\}.
 \label{4.48}
 \end{equation}
Moreover, when $B^{\Mi,s_0}_{p_0,q_0}\hookrightarrow L_\infty$
the space $B^{\Mi,s_1}_{p_1,q_1}$ can receive in \eqref{4.48} for
$s_0>\fracc{|\M|}{p_0}$, while $B^{\Mi,s_1}_{p_1,\infty}$ can do so
for $s_0=\fracc{|\M|}{p_0}$. 

Similarly $B^{\Mi,s_1}_{p^*_1,\infty}$ may be included when
$0<s_0<\fracc{|\M|}{p_0}$ holds in addition to
$q_0<t_0:=|\M|(\fracc{|\M|}{p_0}-s_0)^{-1}$ (or just $q_0\le t_0$ if
$t_0\le 2$ or if $\M=(1,\dots,1)$).
\end{thm}

\begin{proof} For the treatment of $\pi_2$ one can use the continuity of
 \begin{equation}
 B^{\Mi,s_0}_{p_0,q_0}\oplus B^{\Mi,s_1}_{p_1,q_1}\xrightarrow{\;\pi_2\;}
 L_1\hookrightarrow B^{\Mi,s_1}_{r_0,\infty}\cap 
 (\bigcap\bigl\{\,B^{\Mi,s_1}_{r,q_1}\bigm| r_0> r\ge 1\,\bigr\} ),
 \label{4.50}
 \end{equation}
that one has in view of the condition $\fracc1{q_0}+\fracc1{q_1}\ge 1$,
cf.~Proposition~\ref{border-prop}. 
\end{proof}

It is seen by use of embeddings that all the borderline cases with 
 \begin{equation}
 s_0+s_1=\fracc{|\M|}{p_0}+\fracc{|\M|}{p_1}-|\M|\quad
 \text{and}\quad \fracc1{q_0}+\fracc1{q_1}\ge 1
 \label{4.51}
 \end{equation}
allow application of $\pi(\cdot,\cdot)$, that is, \eqref{4.51} implies
that $(s_j,p_j,q_j)_{j=0,1}\in \Dm(\pi,BB)$.

However, for $0\le s_1\le s_0$ the resulting $A_2$ spaces with
$s_2<0\le\min(s_0,s_1)$ are not optimal, cf.~\cite[Thm.~4.1]{Ama}.

\bigskip

Altogether we have now in Sections~\ref{FFF-ssect} and \ref{BBB-ssect}
given a fairly complete description of $\Dm(\pi,BB)$ and
$\Dm(\pi,FF)$. In fact, using \eqref{4.4} in Theorem~\ref{FFF-thm} for
the generic $FF\bigdot$ cases, \eqref{4.4-3} for the cases with
$s_0+s_1=0$ and \eqref{4.5-4''} ff.\,---\,
and analogously for the $BB\bigdot$ cases\,---\,we have found 

\begin{cor} \label{dscr-cor}
When $\max(s_0,s_1)>0$ in the $BB\bigdot$ or in the $FF\bigdot$ cases,
the simultaneous fulfilment of the conditions \rom{(1)}, \rom{($1'$)},
\rom{(2)} and \rom{($2'$)} in Theorem~\ref{ncss-thm} is necessary and 
sufficient for the boundedness of $\pi(\cdot,\cdot)$ in \eqref{1.6} .
\end{cor} 

\begin{rem} \label{ST-rem}
The cases with $s_0=0=s_1$ are treated in the preprint \cite{ST},
where it is shown that $A_j\hookrightarrow L_{p_j}$ for $j=0$ and $1$
together with $L_p\hookrightarrow A$ must hold when \eqref{1.6} and
$s_0=0=s_1$ do so. By H\"older's inequality these embeddings conversely
imply \eqref{1.6} in this case, but the author is unable to follow the
present proof in \cite{ST} for the cases with $p_j=1$. Hence a complete 
description of $\Dm(\pi,BB)$ and $\Dm(\pi,FF)$ is left for the future.
\end{rem}

\subsection{The mixed cases}  \label{othr-ssect}
Concerning the remaining $BBF$, $BF\bigdot$, $FB\bigdot$ and $FFB$
cases we shall verify the sufficiency claimed above \eqref{4.2}:
Theorem \ref{FFF-thm} shows that
$(s_j,p_j,q_j)\in\Dm(\pi,FF)$ when
$(s_j,p_j,q_j)\in\Dm(\pi)$. But then
$(s_j,p_j,q_j)\in\Dm(\pi,\bigdot\bigdot)$, since there
is in any case a simple embedding $A_0\oplus A_1\hookrightarrow 
F^{\Mi,s_0-\varepsilon}_{p_0,\infty}\oplus 
F^{\Mi,s_1-\varepsilon}_{p_1,\infty}$,
where $(s_j-\varepsilon,p_j,\infty)_{j=0,1}\in\Dm(\pi)$ when
$\varepsilon>0$.
Thus $\Dm(\pi)\subset\Dm(\pi,\bigdot\bigdot)$.

In addition it is clear from this that (by taking $\varepsilon>0$
sufficiently small) $A_2$ can be obtained with any parameter
$(s_2,p_2,q_2)$ for which $s_2<s_1$,
$\fracc1{p_2}\le\fracc1{p_0}+\fracc1{p_1}$ and $s_2-\fracc{|\M|}{p_2}<
\min^+(s_j-\fracc{|\M|}{p_j})$, cf.~the interior of the
region sketched in Figure~\ref{A2-fig}.

\subsection{$\psi$-independence} \label{indp-ssect}
Strictly speaking, boundedness has in this section only been obtained
for $\pi_\psi(\cdot,\cdot)$ with $\psi$
as an arbitrary function entering in Definition~\ref{pi-defn}. 

For $(s_j,p_j,q_j)_{j=0,1}\in\Dm(\pi,\bigdot\bigdot)$ and
$(s_2,p_2,q_2)\in \Pl(A_0,A_1)$ we shall now verify that the action of
$\pi_\psi(\cdot,\cdot)\colon A_0\oplus A_1\to A_2$
does not depend on $\psi$ (except for a
$BF\bigdot$ borderline case with $p_1=1$, $q_0=q_1=\infty$ that is left open
for $\M\ne(1,\dots,1)$). 

\bigskip

When applying $\pi_\psi(\cdot,\cdot)$ to spaces with $(s_j,p_j,q_j)_{j=0,1}\in
\Dm(\pi)$ we may
assume that $q_0$ and $q_1<\infty$ so that either $\cal S(\Rn)$ or
$C^\infty(\Rn)$ is dense in $A_0$ and $A_1$. When $A_j=
B^{\Mi,s_j}_{\infty,q_j}(\Rn)$ for $j=0$ and $1$, 
it is found from Proposition~\ref{lp-prop} that 
$\pi_\psi(\cdot,\cdot)$ on $A_0\oplus A_1$ is an extension
by continuity of the $\psi$-independent restriction of
$\mu(\cdot,\cdot)$ to $C^\infty(\Rn)\times C^\infty(\Rn)$. 
The argument carries over to the situation where 
one or both of the $p_j<\infty$.

For the borderline cases with $s_0+s_1=0$ the inequality
$\fracc1{q_0}+\fracc1{q_1}\ge1$ assures that $q_j<\infty$ for $j=0$ or
$1$. So on a dense subset of $A_0\oplus A_1$ the bilinear
operator $\pi_\psi(\cdot,\cdot)$ coincides with a restriction
of the product on $\cal O_M\times\cal S'$ by Proposition~\ref{slw-prop}.

When $s_0+s_1=\fracc{|\M|}{p_0}+\fracc{|\M|}{p_1}-|\M|$ one can for
the $BB\bigdot$ cases use the argument above, since the inequality
$\fracc1{q_0}+\fracc1{q_1}\ge1$ holds according to
Theorem~\ref{ncss-thm}. When $A_j=F^{\Mi,s_j}_{p_j,q_j}$ for $j=0$ and
$1$, one can reduce to this situation: Let
$\varepsilon=s_0+s_1=\fracc{|\M|}{p_0}+\fracc{|\M|}{p_1}-|\M|$.

For $\varepsilon=0$ the argument for $s_0+s_1=0$ applies. When
$\varepsilon>0$ there is an embedding $A_j\hookrightarrow
B^{\Mi,t_j}_{r_j,r_j}$ with $t_j=s_j-\varepsilon_j$,
$\fracc{|\M|}{r_j}=\fracc{|\M|}{p_j}-\varepsilon_j$ and
$\varepsilon_0+\varepsilon_1<\varepsilon$. Because
 \begin{equation}
 t_0+t_1=\fracc{|\M|}{r_0}+\fracc{|\M|}{r_1}-|\M|=
 \varepsilon-\varepsilon_0-\varepsilon_1>0
 \label{6.51}
 \end{equation}
it is found for the sum-exponents of $B^{\Mi,t_j}_{r_j,r_j}$ that
$\fracc1{r_0}+\fracc1{r_1}\ge1$. In connection with Theorem~\ref{BBB4-thm}
this shows that $\pi(\cdot,\cdot)$ is
defined on $B^{\Mi,t_0}_{r_0,r_0}\oplus B^{\Mi,t_1}_{r_1,r_1}\supset A_0\oplus
A_1$, and the $\psi$-independence on $A_0\oplus A_1$ follows.

In the  $BF\bigdot$ cases the inequality
$\fracc1{q_0}+\fracc1{p_1}\ge1$ holds, so the
possibility $q_0=\infty$ is excluded for $p_1>1$; hence either $\cal
S$ or $C^\infty$ is dense in $A_0$, so Proposition~\ref{slw-prop}
applies. For arbitrary $p_1<\infty$ we may define $\varepsilon$ and
$\varepsilon_1$ as above (and take $\varepsilon_0=0$). When
$\varepsilon>0$ and $p_1<1$ a  small $\varepsilon_1$ yields $r_1<1$ so
that $\fracc1{q_0}+\fracc1{r_1}>1$ and $\psi$-independence follows.
The case $p_1=1$ poses a problem only when $q_0=\infty=q_1$, but for
$\M=(1,\dots,1)$ one may use that 
$F^{\Mi,s_1}_{1,\infty} \hookrightarrow B^{\Mi,t}_{r,1}$ for
$s_1-|\M|=t-\fracc{|\M|}{r}$ when $r>1$ and note that
$\pi(\cdot,\cdot)$ is defined on $B^{\Mi,s_0}_{p_0,q_0}\oplus
B^{\Mi,t}_{r,1}$.

Altogether this shows the $\psi$-independence (with the mentioned
exception), and thus the formulation in Section~\ref{suff-sect} of
results for $\pi(\cdot,\cdot)$ has been justified.

\section{Multiplication on open sets} \label{subset-sect}
For an arbitrary open set $\Omega\subset\Rn$ the product
$\pi_{\Omega}(\cdot,\cdot)$ is defined by lifting to $\Rn$. 

\begin{defn} \label{subset-defn} 
The {\em product\/} $\pi_\Omega(u,v)$ is defined for
$u$ and $v\in\cal D'(\Omega)$ when there exists $u'$ and $v'\in\cal
S'(\Rn)$ such that $r_\Omega u'=u$ and $r_\Omega v'=v$ and such that for
every $\psi\in C^\infty_0(\Rn)$ with $\psi(x)=1$ for $x$ in a
neighbourhood of $x=0$ the sequence 
 \begin{equation}
 r_\Omega\bigl(\cal F^{-1}(\psi_k\cal Fu')\cdot
               \cal F^{-1}(\psi_k\cal Fv')\bigr)
 \label{5.1}
 \end{equation}
converges in $\cal D'(\Omega)$ with a $\psi$-independent limit. 

In the affirmative case $\pi_\Omega(u,v)=\lim_{k\to\infty}r_\Omega
\cal F^{-1}(\psi_k\cal Fu')\cdot\cal F^{-1}(\psi_k\cal Fv')$.
\end{defn}

Naturally this definition needs to be justified. Let $u_j$
and $v_j\in\cal S'(\Rn)$ satisfy $r_\Omega u_j=u$ and
$r_\Omega v_j=v$ for $j=1$ and $2$ and let
$r_\Omega u_1^k v_1^k:=r_\Omega(u_1^k v_1^k)$ converge in 
$\cal D'(\Omega)$. Then one can use the identity
 \begin{equation}
 r_\Omega u_1^k v_1^k-r_\Omega u_2^k v_2^k=
 r_\Omega (u_1^k-u_2^k)v_1^k+r_\Omega u_2^k(v_1^k-v_2^k)
 \label{5.2}
 \end{equation}
to infer from Proposition \ref{local-prop} 
that also $r_\Omega u_2^k v_2^k$ converges, and that $\lim
r_\Omega u_2^k v_2^k$ is equal to $\lim r_\Omega u_1^k v_1^k$. 
Thus the existence of a limit is independent of how $u'$ and $v'$ are
chosen, and there is $\psi$-independence for every pair $(u',v')$ if
there is for one. Observe
also that in Definition~\ref{subset-defn} $\pi(u',v')$ need not be 
defined, and that we get back the definition of $\pi$ itself when $\Omega=\Rn$.

In the sequel, $B^{\Mi,s}_{p,q}(\overline{\Omega})=r_\Omega
B^{\Mi,s}_{p,q}(\Rn)$, e.g., is equipped with the infimum quasi-norm.

\begin{thm} \label{subset-thm} 
Let $\pi(\cdot,\cdot)\colon A_0\oplus A_1\to A_2$ be bounded for  
spaces $A_j$ that for $j=0$, $1$ and $2$ satisfy either 
$A_j=B^{\Mi,s_j}_{p_j,q_j}(\Rn)$ or $A_j=F^{\Mi,s_j}_{p_j,q_j}(\Rn)$.

When $\Omega\subset\Rn$ is open, one has boundedness of 
 \begin{equation}
 \pi_\Omega(\cdot,\cdot)\colon A_0(\overline{\Omega})\oplus
 A_1(\overline{\Omega})\to A_2(\overline{\Omega}).
 \label{5.3}
 \end{equation} 
Moreover, if $f\in r_\Omega(L_{p,\op{loc}}(\Rn)\cap
\cal S'(\Rn))$ and $g\in r_\Omega(L_{q,\op{loc}}(\Rn)\cap\cal
S'(\Rn))$ then
 \begin{equation}
 \pi_\Omega(f,g)=f(x)\cdot g(x)\in L_{r,\op{loc}}(\Omega),
 \label{5.4}
 \end{equation}
when $0\le\fracc1r=\fracp+\fracc1q\le1$.
\end{thm}

\begin{proof} When $(u,v)\in A_0(\overline\Omega)\oplus
A_1(\overline\Omega)$ any lift $(u',v')\in A_0\oplus A_1$ admits
application of $\pi$, so a fortiori $\pi_\Omega(u,v)$ exists and
equals $r_\Omega\pi(u',v')$. From
 \begin{equation}
 \norm{\pi_\Omega(u,v)}{A_2(\overline\Omega)}\le
 \norm{\pi(u',v')}{A_2} \le
 \|\pi\|\norm{u'}{A_0}\norm{v'}{A_1}
 \label{5.5}
 \end{equation}
it follows by taking the infimum over $u'$ and $v'$ that 
$\pi_\Omega$ is bounded. Concerning $f$ and $g$ one can simply apply
Proposition \ref{lp-prop} to $e_\Omega f$ and $e_\Omega g$.
\end{proof}

In \cite{T2} and \cite{T3} there is not given any {\em precise\/} 
definition of 
the restriction to $\Omega$, but, because $\pi$ can not
be identified with $\mu$ in general, and because $q=\infty$ excludes
denseness of smooth functions, it is a point to show that
$\pi_\Omega(u,v)$ does not depend on the actual choice of the
lift. The approach used in Definition \ref{subset-defn} takes
care of this in a general way. (By doing it for each space 
$A_0(\overline{\Omega})\oplus A_1(\overline{\Omega})$ a consistency
problem arises, since $u$ and $v$ may belong to other spaces.)
 
In the particular case when $\Omega$ is of finite measure the
inclusion $L_r(\Omega)\hookrightarrow L_p(\Omega)$, valid for
$\infty\ge r\ge p>0$, carries over to the scales
$B^{\Mi,s}_{p,q}(\overline{\Omega})$ and 
$F^{\Mi,s}_{p,q}(\overline{\Omega})$ when $\Omega$ is suitably
`nice', cf.\ Lemma~\ref{reg-lem} below. 
Consequently\,---\,when the results of Section~\ref{suff-sect} are carried
over to $\pi_\Omega(\cdot,\cdot)$ by means of Theorem~\ref{subset-thm}
above\,---\,the restriction $\fracc{|\M|}{r}\le
\fracc{|\M|}{p_0}+\fracc{|\M|}{p_1}$ may be removed
when $\Omega$ is bounded, e.g.

This observation raises the question whether the other conditions in 
Theorem~\ref{ncss-thm} can be shown to hold for $\pi_\Omega$ or not.
The methods in Section \ref{ncss-sect} are only applicable for $\Omega=\Rn$, 
and it remains open how to proceed in general when $\Omega\ne\Rn$.
Even so it would be rather surprising if other modifications were necessary. 

\bigskip

In Lemma \ref{reg-lem} ff.\ below we address in particular the
embeddings $B^{\Mi,s}_{r,q}(\overline{\Omega})\hookrightarrow 
B^{\Mi,s}_{p,q}(\overline{\Omega})$ and 
$F^{\Mi,s}_{r,q}(\overline{\Omega})\hookrightarrow 
F^{\Mi,s}_{p,q}(\overline{\Omega})$ etc. that are shown to be valid
provided $\infty\ge r\ge p>0$  when $\Omega$ is a suitable set with
$\op{meas}(\Omega)<\infty$.

For isotropic spaces over a bounded $\Omega$ the results are  
identical to \cite[3.3.1]{T2}, 
except that $q=\infty$ is not included there in the $F$ case for $t=s$.
Seemingly the technique here gives a simpler proof of all cases with $r\ge p$.

\begin{lem}  \label{reg-lem} 
Let $\Omega\subset\Bbb R^n$ be an open bounded set, and let
$(s,p,q)$ and $(t,r,o)$ belong to $\Bbb R\times\left]0,\infty\right]
\times\left]0,\infty\right]$.

Then there is an embedding $B^{\Mi,t}_{r,o}(\overline\Omega)\hookrightarrow 
B^{\Mi,s}_{p,q}(\overline\Omega)$ when
 \begin{equation}
 \begin{gathered}
 \text{ $t\geq s$ and $\;t-\fracc{|\M|}r\ge s-\fracc{|\M|}p$ are satisfied},\\
 \text{together with $\;o\le q$ if $\;t=s$, or if 
 $\;t-\fracc{|\M|}r=s-\fracc{|\M|}p$};  
 \end{gathered}
 \label{4.5.7}
 \end{equation}
and when $r$ and $p<\infty$ and moreover
 \begin{equation}
 \begin{gathered}
 \text{ $t\ge s$ and $\;t-\fracc{|\M|}r\ge s-\fracc{|\M|}p$ are satisfied},\\
 \text{together with $\;o \le q$ if $\;t=s$}, 
 \end{gathered}
 \label{4.5.6} 
 \end{equation}     
an embedding $F^{\Mi,t}_{r,o}(\overline\Omega)
\hookrightarrow F^{\Mi,s}_{p,q}(\overline\Omega)$ exists.
\end{lem}

\begin{proof} 
We show for $s\in\Bbb R$ and $\infty>r\ge p>0$ the existence of an embedding
 \begin{equation}
 F^{\Mi,s}_{r,q}(\overline\Omega)\hookrightarrow 
 F^{\Mi,s}_{p,q}(\overline\Omega),
 \quad\text{\em for any}\quad \,q\in\left]0,\infty\right],
 \label{4.5.10}
 \end{equation}
by use of the result that on $\Bbb R^n$ one has, cf.\ Section~\ref{suff-sect},
 \begin{equation}
 F^{\Mi,s_0}_{2,2}\oplus F^{\Mi,s}_{r,q}
  \xrightarrow{\;\pi(\cdot,\cdot)\;}
 \bigcap\{\,F^{\Mi,s}_{t,q}\mid \fracc1r\le \fracc1t\le \fracc1r+\tfrac12\,\},
 \label{4.5.11}
 \end{equation}
for $s_0>\fracc{|\M|}2$, when $s_0>s$ and 
$s_0+s>\max(0,\fracc{|\M|}r-\frac{|\M|}{2})$.

Let $N$ be a natural number such that $p\ge(\fracc1r+\frac{N}{2})^{-1}$
and let $\Omega_j$ be open bounded sets such that $\Omega=\Omega_0
\subset\overline\Omega_0\subset\Omega_1\subset\dots
\subset\overline\Omega_N\subset\Omega_{N+1}$. For each $j=0,\dots,N$
we pick $\psi_j\in C^\infty(\Bbb R^n)$ such that
$\psi_j\equiv 1$ on $\Omega_j,$ and $\operatorname{supp}\psi_j
\subset\Omega_{j+1}$.

Now let $v\in F^{\Mi,s}_{r,q}(\overline\Omega)$. Since $r_\Omega 
\psi_0w=v$ when $r_\Omega w=v$ and since $\psi_k\dots\psi_0w=\psi_0w$,
a repeated application of \eqref{4.5.11} gives, with $\fracc1{r_j}=
\fracc1r+\frac{j}{2}$, 
 \begin{equation}
 \begin{split}
 \norm{v}{F^{\Mi,s}_{r_{k+1},q}(\overline{\Omega})}
 \le&\inf\bigl\{\,\norm{\psi_k\dots\psi_0w}{F^{\Mi,s}_{r_{k+1},q}} 
         \bigm| w\in F^{\Mi,s}_{r,q},\quad r_\Omega w=v \,\bigr\}
 \\
 \le&\,c\norm{\psi_k}{F^{\Mi,s_k}_{2,2}}\,\dots
 \norm{\psi_0}{F^{\Mi,s_0}_{2,2}} 
 \norm{v}{ F^{\Mi,s}_{r,q}(\overline{\Omega})},
 \end{split}
 \label{4.5.13}
 \end{equation}
for some $c<\infty$, if $s_0,\dots,s_k$ are big enough. For some $k\le N$ 
it is even possible to take the $F^{\Mi,s}_{p,q}$-norm on the left hand
side of \eqref{4.5.13}, cf.\  \eqref{4.5.11} and the definition of $N$.

A similar procedure works for $t=s$ in the Besov-case. The full
statements now follow by use the embeddings in Section~\ref{embd-ssect}. 
\end{proof}

Obviously one could equally well work with sets $\Omega$ of finite
measure for which there exists open sets $\Omega_j$ and $\psi_j\in\cal
S(\Rn)$ such that
 \begin{equation}
  \begin{gathered}
  \Omega=\Omega_0\subset\dots\subset\Omega_j\subset\overline{\Omega}_j
 \subset\Omega_{j+1}\subset\dots, \\
  \op{meas}(\Omega_j)<\infty\quad\text{for all}\quad j\in\N_0,\\
 \psi_j\equiv 1\ \text{on}\ \Omega_j,\quad \supp
\psi_j\subset\Omega_{j+1}.
 \label{4.5.14}
 \end{gathered}
 \end{equation}
Indeed, $\psi_k\in W^{\smlnt{s_k}}_2(\Rn)\hookrightarrow
F^{\Mi,s_k}_{2,2}(\Rn)$
since $\op{meas}(\Omega_{k+1})<\infty$. Thus we have

\begin{cor} \label{subset-cor}
When $\Omega\subset\Rn$ is open with $\op{meas}(\Omega)<\infty$ and
when \eqref{4.5.14} holds for some $\psi_j$ and $\Omega_j$, 
then Lemma~\ref{reg-lem} holds for $\Omega$.
\end{cor}

\section{Applications}  \label{appl-sect}

Already in Lemma~\ref{reg-lem} ff.\ above there is a theoretical
application of the results for $p_0\ne p_1$, cf.~\eqref{4.5.11}.

For boundary problems for non-linear partial differential equations
Theorem~\ref{subset-thm} is a tool, which allows one to treat, say,
products in $B^{\Mi,s}_{p,q}$ and $F^{\Mi,s}_{p,q}$ spaces. 
For the particular cases with $p_0=p_1=p_2$ one can hereby apply the
$\Rn$-results in \cite[Thm.~6.1]{Y1} to such problems. In the author's thesis
\cite{JJ93} this approach has been used for the stationary
Navier--Stokes equations. In particular the non-linear terms are
estimated in the $B^{s}_{p,q}$ and $F^{s}_{p,q}$ spaces in this way.

Moreover, a rather satisfactory set of {\em 
regularity properties\/} for these equations have been deduced there,
and in the particular case of boundary conditions of {\em class 2\/}
(similar to those in \cite{GS2}), the product results for $p_0=p_1\ne
p_2$ enter in a {\em decisive\/} way. See \cite[Thm.~5.5.3]{JJ93}
for further details. 

The general anisotropic results with $\M\ne(1,\dots,1)$ are applicable to the
time-dependent Navier--Stokes equations, cf.~\cite{GS2} and \cite{G5},
and to other non-linear parabolic problems, that are considered on a
cylinder like $\Omega\times\,]0,T[\,$ with the time variable running
in the interval $\,]0,T[$. For such problems it is most convenient that the
open sets in Definition~\ref{subset-defn} and Theorem~\ref{subset-thm}
are allowed to have non-smooth boundaries. 

The $BF\bigdot$ and $FB\bigdot$ cases, that are only given a minimal
treatment on $\Rn$ here, should be relevant for differentials of
non-linear operators, and hence for certain stability questions.
As an example one could take the differential $d\pi$ of the product 
$\pi(\cdot,\cdot)$ at a point $u\in B^{\Mi,s}_{p,q}$; this is the 
linear operator $v\mapsto d\pi(u)v=\pi(u,v)+\pi(v,u)$, that may act between
$F^{\Mi,s}_{p,q}$ spaces.

%


\begin{thebibliography}{Yam86b}

\bibitem[Ama91]{Ama}
H.~Amann, \emph{{Multiplication in Sobolev and Besov spaces}}, {Nonlinear
  Analysis}, {Scuola Normale Superiore}, Pisa, 1991, pp.~27--50.

\bibitem[Bon81]{Bon}
J.-M. Bony, \emph{{Calcul symbolique et propagations des singularit{\'e}s pour
  les \'eq\-ua\-tions aux d\'eriv\'ees partielles non lin\'eaires}}, Ann.
  scient.~\'Ec. Norm. Sup. \textbf{14} (1981), 209--246.

\bibitem[CO90]{ColOber90}
J.~F. Colombeau and M.~Oberguggenberger, \emph{{On a hyperbolic system with a
  compatible quadratic term: generalised solutions, delta waves, and
  multiplication of distributions}}, Comm. Part. Diff. Eq. \textbf{15} (1990),
  905--938.

\bibitem[Fra86]{F3}
J.~Franke, \emph{{On the spaces $F^s_{pq}$ of Triebel--Lizorkin type: pointwise
  multipliers and spaces on domains}}, Math. Nachr. \textbf{125} (1986),
  29--68.

\bibitem[Gru95]{G5}
G.~Grubb, \emph{{Nonhomogeneous time-dependent Navier--Stokes problems in $L_p$
  Sobolev spaces}}, Diff. Int. Eq. \textbf{8} (1995), 1013--1046.

\bibitem[GS91]{GS2}
G.~Grubb and V.~A. Solonnikov, \emph{{Boundary value problems for the
  non-station\-ary Navier--Stokes equations treated by pseudo-differential
  methods}}, Math. Scand. \textbf{69} (1991), 217--290.

\bibitem[Han85]{Han}
B.~Hanouzet, \emph{{Applications bilin\'eaires compatibles avec un systeme \`a
  co\'efficients variables continuit\'e dans les espaces de Besov}}, Comm.
  Part. Diff. Eq. \textbf{10} (1985), 433--465.

\bibitem[Jaw77]{J}
B.~Jawerth, \emph{{Some observations on Besov and Lizorkin--Triebel spaces}},
  Math. Scand. \textbf{40} (1977), 94--104.

\bibitem[Joh93]{JJ93}
J.~Johnsen, \emph{{The stationary Navier--Stokes equations in $L_p$-related
  spaces}}, Ph.D. thesis, University of Copenhagen, Denmark, 1993,
  {Ph.D.-series {\bf 1}}.

\bibitem[Kam82]{Kam82}
A.~Kami\'nski, \emph{{Convolution, product and Fourier transform of
  distributions}}, Studia Math. \textbf{74} (1982), 83--96.

\bibitem[Knu92]{K}
D.~E. Knuth, \emph{{Two notes on notation}}, Amer. Math. Monthly \textbf{99}
  (1992), 403--422.

\bibitem[MS85]{MazShap}
V.~A. Maz'ya and T.~O. Shaposhnikova, \emph{{Theory of multipliers in spaces of
  differentiable functions}}, Monographs and studies in mathematics, vol.~23,
  Pitman, London, 1985.

\bibitem[Obe92]{Ober}
M.~Oberguggenberger, \emph{{Multiplication of distributions and applications to
  partial differential equations}}, Pitman Research Notes in Mathematics
  series, vol. 259, Longman Scientific \& Technical, Harlow, UK, 1992.

\bibitem[Pal68]{Pal}
R.~P. Palais, \emph{{Foundations af non-linear global analysis}}, Benjamin, New
  York, 1968.

\bibitem[Pee76]{Pee}
J.~Peetre, \emph{{New thoughts on Besov spaces}}, Duke Univ. Math. Series,
  Durham, 1976.

\bibitem[Sic87]{S}
W.~Sickel, \emph{{On pointwise multipliers in Besov--Triebel--Lizorkin
  spaces}}, {Seminar Analysis of the Karl Weierstra\ss--Institute 1985/86}
  (Leipzig) ({Schulze, B.--W. and Triebel, H.}, ed.), {Teubner--Texte zur
  Mathematik}, vol.~96, {Teubner Verlagsgesellschaft}, 1987, pp.~45--103.

\bibitem[Sic91]{Sic91}
\bysame, \emph{{Pointwise multiplication in Triebel--Lizorkin spaces}}, Forum
  Math. \textbf{5} (1991), 73--91.

\bibitem[ST]{ST}
W.~Sickel and H.~Triebel, \emph{{H{\"o}lder inequalities and sharp imbeddings
  in function spaces of $B^s_{p,q}$ and $F^s_{p,q}$ type}}, (preprint).

\bibitem[Str67]{Strch}
R.~S. Strichartz, \emph{{Multipliers on fractional Sobolev spaces}}, J. Math.
  Mechanics \textbf{16} (1967), 1031--1060.

\bibitem[Tri77]{T-pmlt}
H.~Triebel, \emph{{Multiplication properties of the spaces $B^{s}_{p,q}$ and
  $F^{s}_{p,q}$}}, Ann. Mat. Pura Appl. \textbf{113} (1977), 33--42.

\bibitem[Tri78]{T0}
\bysame, \emph{{Spaces of Besov--Hardy--Sobolev type}}, Teubner--Texte zur
  Mathematik, vol.~15, {Teubner Verlagsgesellschaft}, Leipzig, 1978.

\bibitem[Tri83]{T2}
\bysame, \emph{{Theory of function spaces}}, Monographs in mathematics,
  vol.~78, Birkh{\"a}user Verlag, Basel, 1983.

\bibitem[Tri92]{T3}
\bysame, \emph{{Theory of function spaces II}}, Monographs in mathematics,
  vol.~84, Birkh{\"a}user Verlag, Basel, 1992.

\bibitem[Yam86a]{Y1}
M.~Yamazaki, \emph{{A quasi-homogeneous version of paradifferential operators,
  I. Boundedness on spaces of Besov type}}, J. Fac. Sci. Univ. Tokyo Sect. IA,
  Math. \textbf{33} (1986), 131--174.

\bibitem[Yam86b]{Y2}
\bysame, \emph{{A quasi-homogeneous version of paradifferential operators, II.
  A symbol calculus}}, J. Fac. Sci. Univ. Tokyo Sect. IA, Math. \textbf{33}
  (1986), 311--345.

\bibitem[Zol77]{Zol}
J.~L. Zolesio, \emph{{Multiplication dans les espaces des Besov}}, Proc. R. S.
  E. (A) \textbf{78} (1977), 113--117.

\end{thebibliography}

\providecommand{\bysame}{\leavevmode\hbox to3em{\hrulefill}\thinspace}
\providecommand{\MR}{\relax\ifhmode\unskip\space\fi MR }
\providecommand{\MRhref}[2]{%
  \href{http://www.ams.org/mathscinet-getitem?mr=#1}{#2}
}
\providecommand{\href}[2]{#2}

\end{document}